\documentclass{gtart}

\usepackage[inner=20mm, outer=20mm, textheight=225mm]{geometry}

\usepackage{amsmath, amssymb, latexsym, amscd, graphicx, epstopdf}
\usepackage{euscript, setspace}
\usepackage{subfigure}
\usepackage{epsfig}
\usepackage{psfrag}
\usepackage[makeroom]{cancel}
\usepackage{caption}
\usepackage{hyperref}
\hypersetup{final=true}
\hypersetup{colorlinks=true}
\usepackage[all]{hypcap}

\usepackage{tikz}
\usepackage{tikz-cd}
\usetikzlibrary{calc,matrix}

\newcommand{\sslash}{\mathbin{/\mkern-6mu/}}

\def\CC{\mathbb{C}}

\def\NN{\mathbb{N}} 
\def\PP{\mathbb{P}} 
 
\def\RR{\mathbb{R}}
\def\ZZ{\mathbb{Z}}

\def\Sph{\mathbb{S}}

\def\RP2{\RR\PP^2}
\def\C{\mathcal{C}}

\DeclareMathOperator{\GL}{GL}
\DeclareMathOperator{\SL}{SL}
\DeclareMathOperator{\SU}{SU}
\DeclareMathOperator{\PSL}{PSL}

\def\X{\mathfrak{X}}

\def\A{\mathcal{A}}
\def\graph{\mathcal{G}}
\def\P{\mathcal{P}}
\def\tri{\mathrm{T}}
\def\F{\mathcal{F}}

\def\Ttp{\mathcal{T}_3^+}
\def\Ttpd{\mathcal{T}_3^-}
\def\Tt{\mathcal{T}_3}
\def\Ttfin{\mathcal{T}_3^{f}}
\def\Ttm{\mathcal{T}_3^{m}}
\def\Ttframe{\mathcal{T}_3^\times}
\def\Ttpm{\mathcal{T}_3^\pm}

\def\Bf{\boldsymbol{f}}
\def\Bg{\boldsymbol{g}}
\def\Bh{\boldsymbol{h}}

\def\cOmega{\overline{\Omega}}
\def\bOmega{\partial\cOmega}

\DeclareMathOperator{\Teich}{Teich}

\DeclareMathOperator{\Hom}{Hom}
\DeclareMathOperator{\tr}{tr}

\DeclareMathOperator{\im}{im}

\DeclareMathOperator{\cross}{CR}

\DeclareMathOperator{\PGL}{PGL}
\DeclareMathOperator{\Tr}{Tr}

\DeclareMathOperator{\dev}{dev}
\DeclareMathOperator{\hol}{hol}
\DeclareMathOperator{\Hol}{Hol}

\DeclareMathOperator{\Conf}{Conf}

\DeclareMathOperator{\Mon}{Mon}

\DeclareMathOperator{\Par}{Par}
\DeclareMathOperator{\Sym}{Sym}

\DeclareMathOperator{\Gol}{Gol}

\def\n@te#1{\textsf{\boldmath \textbf{$\langle\!\langle$#1$\rangle\!\rangle$}}\leavevmode}
\def\note#1{\textcolor{red}{\n@te {#1}}}

\newcommand\restr[2]{{
		\left.\kern-\nulldelimiterspace 
		#1 
		\vphantom{\big|} 
		\right|_{#2} 
	}}

\theoremstyle{plain}

\newtheorem{thm}{Theorem}
\newtheorem*{thm*}{Theorem}
\newtheorem{lem}[thm]{Lemma}
\newtheorem*{lem*}{Lemma}
\newtheorem{cor}[thm]{Corollary}
\newtheorem*{cor*}{Corollary}

\newtheorem*{cla*}{Claim}
\newtheorem{pro}[thm]{Proposition}
\newtheorem*{pro*}{Proposition}
\newtheorem{rem}[thm]{Remark}
\newtheorem*{rem*}{Remark}

\newtheorem*{defn*}{Definition}
\newtheorem*{exa*}{Example}
\newtheorem{exa}{Example}

\newtheorem*{exe*}{Exercise}

\newtheorem*{que*}{Question}
	
\addtocounter{MaxMatrixCols}{4}
\numberwithin{equation}{section}

\begin{document}

\thispagestyle{empty}

\title{Moduli spaces of real projective structures on surfaces\\ \large Notes on a paper by V.V. Fock and A.B. Goncharov}
\author{Alex Casella, Dominic Tate and Stephan Tillmann}

\begin{abstract}
These notes grew out of our learning and applying the methods of Fock and Goncharov concerning moduli spaces of real projective structures on surfaces with ideal triangulations. We give a self-contained treatment of Fock and Goncharov's description of the moduli space of framed marked properly convex projective structures with minimal or maximal ends, and deduce results of Marquis and Goldman as consequences. We also discuss the Poisson structure on moduli space and its relationship to Goldman's Poisson structure on the character variety.
\end{abstract}
\primaryclass{57M50, 51M10, 51A05; 20F65}
\keywords{Convex projective surface, hyperbolic surface, moduli space, Teichm\"uller space, flag, triple ratio}


\maketitle

\vfill

\section*{Prologue}

The set of all projective structures on a surface turns out to be too large to yield interesting general results, and one therefore restricts to what are called \emph{properly convex projective structures} or \emph{strictly convex projective structures}. The foundations of this study were laid by the work of Kuiper (1953, \cite{Kuiper-surfaces-1953, Kuiper-convex-1954}), Vinberg and Kac (1967, \cite{Vinberg-quasi-1967}) and Benzecri (1960, \cite{Benzecri-varietes-1960}). Parameterisations, structure and dimension of the moduli spaces were  given by Goldman (1990, \cite{Goldman-convex-1990}), Choi--Goldman (1993, \cite{Choi-convex-1993}) and Marquis (2010, \cite{Marquis-espace-2010}). Finding good parameterisations is the key to applications and further advances of the theory. In these notes, a particular parameterisation due to Fock and Goncharov~(2007, \cite{Fock-moduli-2007}) is used to give a self-contained treatment of the key facts about the moduli spaces of these structures on surfaces. We have elaborated many subtle details that are not made explicit in \cite{Fock-moduli-2007}, and hope that this expanded treatment of their work gives a nice introduction to the study of moduli spaces of real projective structures on surfaces.
For foundational and further material on real projective manifolds we refer to Goldman~\cite{Goldman-convex-1990}, Marquis~\cite{Marquis-espace-2010}, Benoist~\cite{Benoist-survey-2008}, Cooper-Long-Tillmann~\cite{Cooper-convex-2015} and the references therein.

Much of the material presented here is based on a seminar series at the University of Sydney, and we heartily thank the following people who lectured or participated in the talks: Grace Garden, Montek Gill, Robert Haraway, James Parkinson, and Robert Tang. The last author learned most of what he knows about projective geometry from Daryl Cooper, who introduced him to this rich, puzzling and at times amusing field of research. He also thanks Bill Goldman for first bringing the work of Fock and Goncharov to his attention. The authors acknowledge support by the Commonwealth of Australia and the Australian Research Council (DP140100158).

\newpage

\begin{spacing}{0.5}
\tableofcontents
\end{spacing}

\newpage

\section{Introduction}

%
%
%
%

%

Let $\RP2$ denote the real projective plane and $\PGL(3, \RR)$ denote the group of projective transformations $\RP2\to\RP2$. A \emph{convex projective surface} is a quotient $\Omega / \Gamma,$ where $\Omega \subset \RP2$ is an open convex subset and $\Gamma$ is a discrete (torsion-free) subgroup of $\PGL(3, \RR)$ that leaves $\Omega$ invariant. Examples of convex projective surface include Euclidean tori (via the embedding of the Euclidean plane in $\RP2$ as an affine patch), and hyperbolic surfaces (via the Klein model of the hyperbolic plane). Some convex projective structures on the once-punctured torus are indicated in Figure~\ref{fig:clover_pictures}.

\begin{figure}[h]
	\centering	
		\includegraphics[width=4cm]{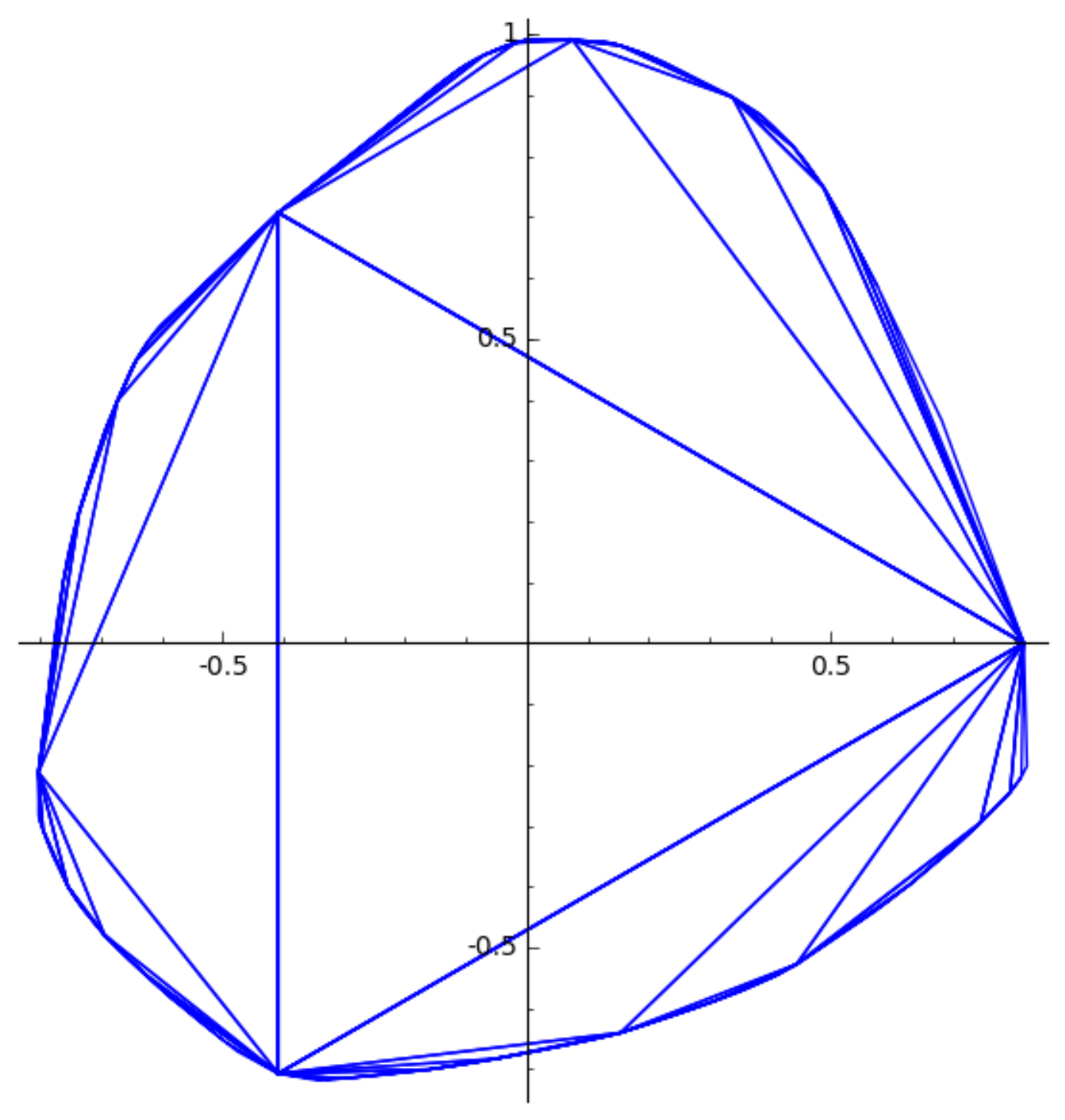} \qquad
		\includegraphics[width=4cm]{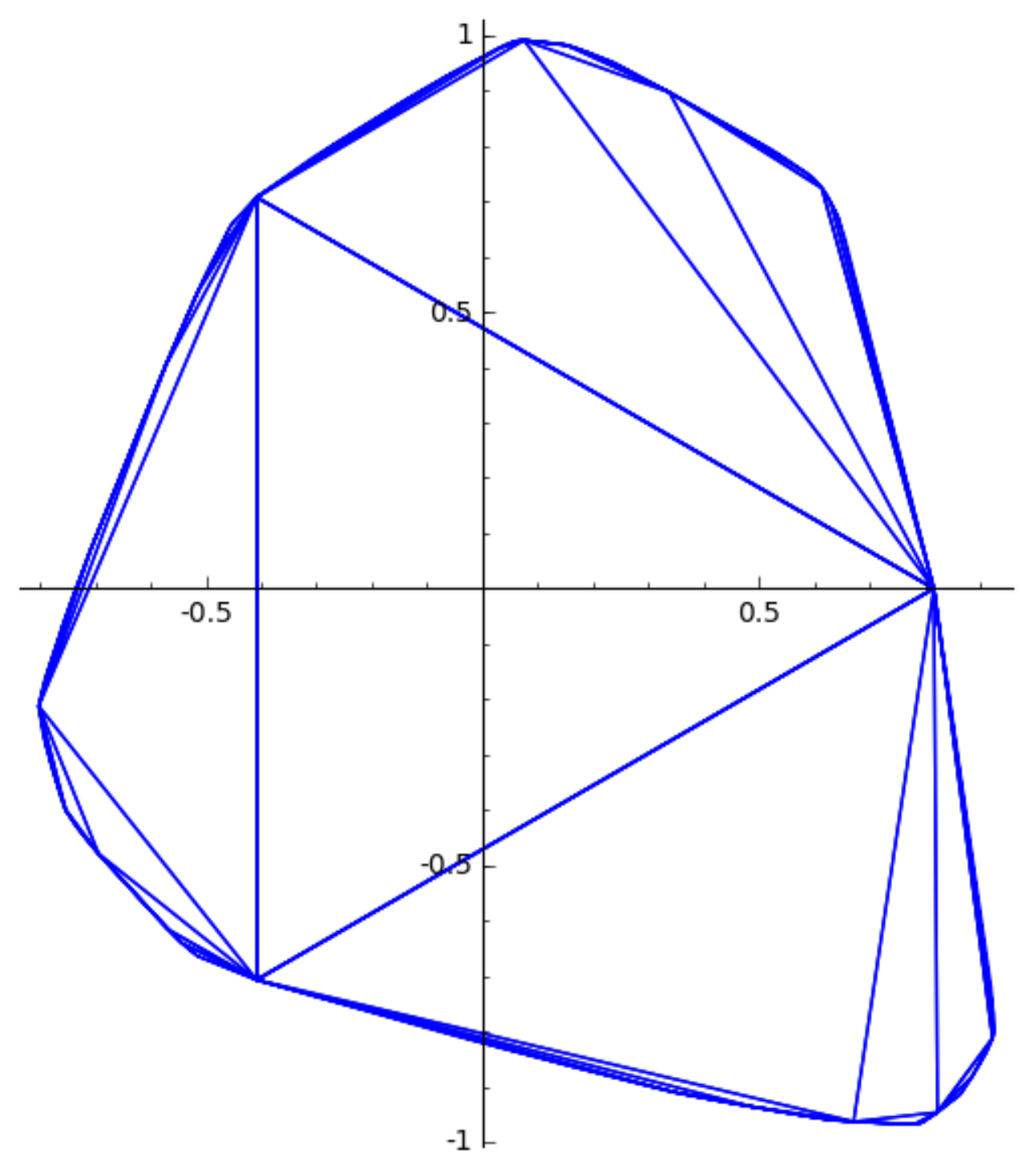} \qquad
		\includegraphics[width=4cm]{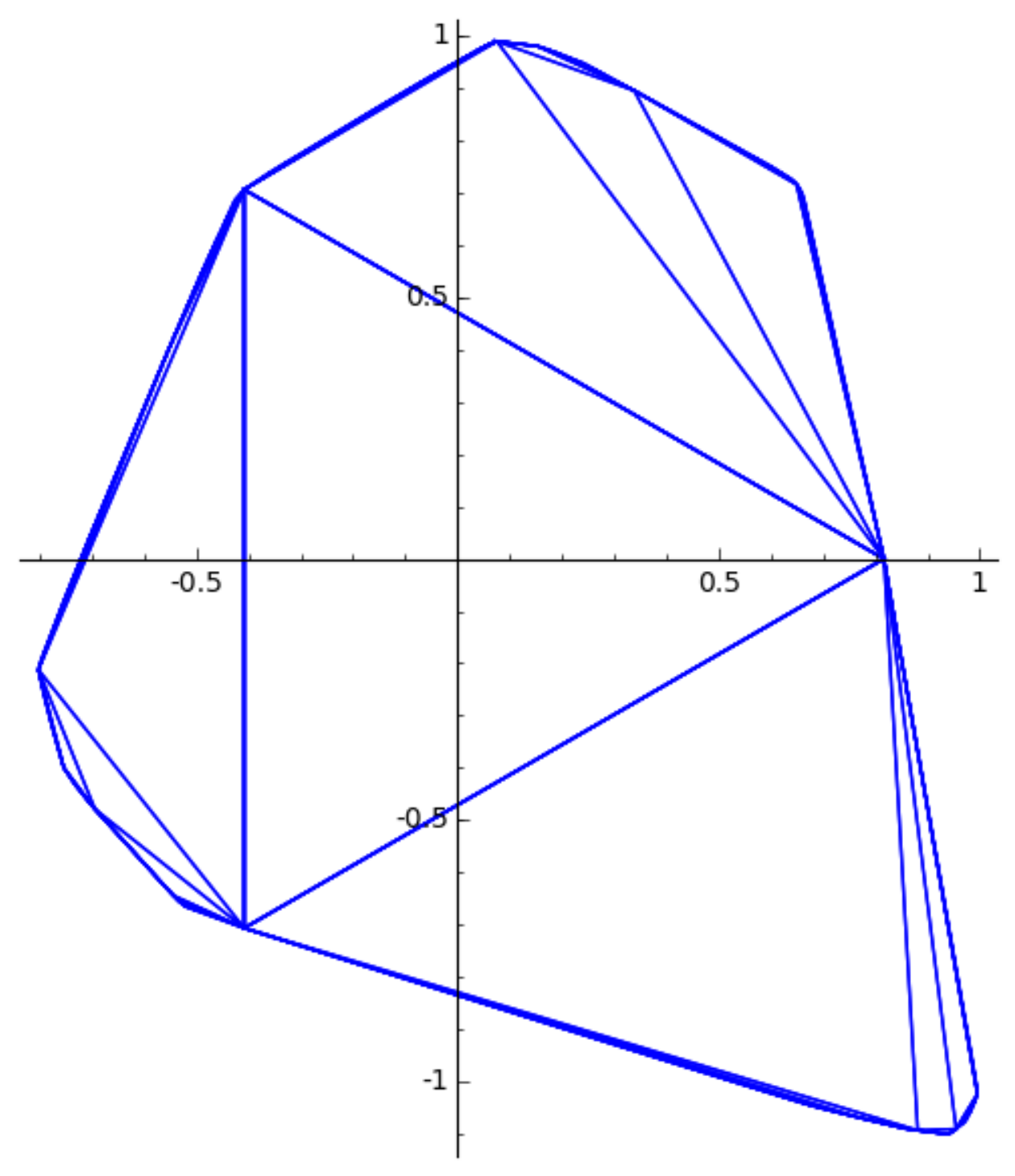}
		\caption{Developing maps for convex projective structures on the once-punctured torus \cite{Haraway-tessellating-2018}}
	\label{fig:clover_pictures}
\end{figure}

This paper is concerned with \emph{moduli spaces} of convex projective structures on surfaces satisfying additional hypotheses, which will now be described.
Let $S$ be a smooth surface. A \emph{marked properly convex projective structure} on $S$ is a triple $( \Omega, \Gamma, f),$ where
\begin{itemize}
\item $\Omega \subset \RP2$ is an open subset of the real projective plane whose closure is contained in an affine patch and which is convex in that patch,
\item $\Gamma$ is a discrete (torsion-free) subgroup of the group $\PGL(3, \RR)$ of projective transformations that leaves $\Omega$ invariant,
\item $f \co S \to \Omega / \Gamma$ is a homeomorphism.
\end{itemize}

Denote $\widetilde{S}$ the universal cover of $S.$ The map $f \co S \to \Omega / \Gamma$ lifts to the \emph{developing map} $\dev \co \widetilde{S} \to \RP2$ with image equal to $\Omega$, and induces the \emph{holonomy representation} $\hol \co \pi_1(S) \to \PGL(3, \RR)$  with image equal to $\Gamma.$

Two marked properly convex projective structures $( \Omega_0, \Gamma_0, f_0) $ and $( \Omega_1, \Gamma_1, f_1)$ are equivalent if there is $A \in \PGL(3, \RR)$ such that $A \cdot \Omega_0 = \Omega_1,$ $A \cdot \Gamma_0 \cdot A^{-1} = \Gamma_1$ and $f_1$ is isotopic to $A^* \cdot f_0$, where $A^*$ is the homeomorphism $\Omega_0 / \Gamma_0 \rightarrow \Omega_1 / \Gamma_1$ induced by $A$. The set of all such equivalence classes, which we will refer to as the \emph{moduli space of marked properly convex projective structures}, is denoted $\Tt(S)$. As such, we view a moduli space as a space of objects, and seek suitable parameterisations of this moduli space.
The subscript alludes to the place this moduli space holds in the more general study of discrete and faithful representations of $\pi_1(S)$ into $\PGL(n, \RR)$; see \cite{Fock-moduli-2006, Labourie-cross-2007, Guichard-convex-2008, Bonahon-parameterizing-2014, Gueritaud-compactification-2017, Kapovich-dynamics-2018}. In the case of properly convex projective structures on a closed surface, no further qualification is necessary:

\begin{thm}[Goldman 1990, \cite{Goldman-convex-1990}]\label{thm:Goldman-main}
If $S$ is a closed, orientable surface of negative Euler characteristic, then $\Tt(S)$ is an open cell of dimension $- 8 \chi(S).$
\end{thm}

Suppose that $S_g$ is a closed, orientable surface of finite genus $g.$ Let $D_1, \ldots, D_n$ be open discs on $S_g$ with pairwise disjoint closures. Choose $p_k \in D_k.$ Then $S_{g,n} = S_g \setminus \{p_1, \ldots, p_n\}$ is a \emph{punctured surface} and we call $E_k = \overline{D_k} \setminus \{p_k\}$ an \emph{end} of $S_{g,n}.$ A \emph{compact core} of $S_{g,n}$ is $S_{g,n}^c = S_g \setminus \cup D_k.$ We also write $S_g = S_{g,0}.$

Goldman's theorem states that $\Tt(S_g)$ is an open cell of dimension $16g-16.$ The most immediate generalisation of Goldman's theorem to punctured surfaces makes the additional hypothesis that the structures on $S$ have \emph{finite volume}. The volume arises in this context by taking the Hausdorff measure computed from the \emph{Hilbert metric} on the domain $\Omega,$ though we remark that there are other ways to define volume that can be used in our context. Denote $\Ttfin(S) \subseteq \Tt(S)$ the subset of all finite volume structures.

\begin{thm}[Marquis 2010, \cite{Marquis-espace-2010}]\label{thm:Marquis-main-intro}
If $S = S_{g,n}$ has negative Euler characteristic and at least one puncture, then $\Ttfin(S)$ is an open cell of dimension $16g-16+6n.$
\end{thm}

Note that $S = S_{g,n}$ has negative Euler characteristic if and only if $2g+n>2.$ All ends in a structure of finite volume are cusps in the same manner that non-compact complete hyperbolic surfaces of finite volume have cusps at their ends.

When considering structures of infinite area in $\Tt(S)$, the first consideration is to enumerate the various possibilities for the geometry at each end of $S$. Consider the interior of a hyperbolic surface $F$ with totally geodesic boundary. One can, in hyperbolic space, equivariantly add a convex collar to the boundary. This collar could be smooth or piecewise linear or a combination thereof. Hence there are infinitely many inequivalent points in $\Tt(F)$ whose structures agree on a compact core of $F.$ In particular, all of these points have the same holonomy. Returning to the general case of a surface $S = S_{g,n}$, denote by $\Hol(S)$ the set of all holonomies of strictly convex projective structures on $S.$ The equivalence relation between structures that defines $\Tt(F)$
descends to the action of $\PGL(3,\RR)$ by conjugation on $\Hol(S),$ and we denote the quotient by
\[ \X(S) = \Hol(S)/\PGL(3,\RR).\]
Following Fock and Goncharov~\cite{Fock-moduli-2007}, we will now \emph{frame} the holonomies. For each end $E_k$ of $S,$ there is a \emph{peripheral subgroup} $\Gamma_k\le \Gamma$ corresponding to $\im (\pi_1(E_k) \to \pi_1(S)).$ Since $\Gamma_k \cong \ZZ,$ it fixes at least one point in $\RP2$ and preserves at least one line in $\RP2$ through that point. 
A \emph{framing} of the holonomy $\hol \co \pi_1(S) \to \PGL(3, \RR)$ consists of the choice, for each end $E_k$, of such an invariant flag consisting of a point and a line. The set of framed holonomies is denoted
\[\Hol^\times(S_{g,n}) \subset \Hom\big(\pi_1(S_{g,n}), \PGL(3, \RR)\big) \times \big( \RP2 \times (\RP2)^*\big)^n\]
There is a natural action of $\PGL(3, \RR)$ on $\Hol^\times(S),$ where the action of an element $A$ of $\PGL(3, \RR)$ on the dual plane $(\RP2)^*$ is given by taking the inverse of the transpose of $A.$
We denote the quotient by
\[ \X^\times(S) = \Hol^\times(S)/ \PGL(3, \RR). \]
\begin{thm}[Fock-Goncharov 2007, \cite{Fock-moduli-2007}]\label{thm:FG-main-X}
If $S = S_{g,n}$ has negative Euler characteristic and at least one puncture, then $\X^\times(S)$ is an open cell of dimension $- 8 \chi(S) = 16g-16+8n.$
\end{thm}
The above result is not stated in this form in \cite{Fock-moduli-2007}. Indeed, what is hidden in the statement is the ingenious observation in \cite{Fock-moduli-2007} that there is a section of $\Hol^\times(S)\to \X^\times(S)$ which consists of the so-called \emph{positive representations}. To describe them, we will introduce two constructions of \emph{framed} marked properly convex projective structures on $S$ using developing maps of ideal triangulations. The distinction between these constructions is not made explicit in \cite{Fock-moduli-2007}, so we will briefly explain it.

In the above example of the hyperbolic surface $F$, we concluded that there are infinitely many inequivalent points in $\Tt(F)$ whose structures agree on a compact core of $F.$ This situation was remedied by focussing on the associated holonomies, which were subsequently decorated with flags. Similarly, one can frame a properly convex projective structure as follows. Given $( \Omega, \Gamma, f),$ a framing is the $\Gamma$--invariant choice for each peripheral subgroup of a fixed point $p$ in the frontier of $\Omega$ and an invariant supporting line to $\Omega$ at $p.$ Denote the resulting space $\Ttframe(S).$ Then there is a natural map $\Ttframe(S)\to \X^\times(S).$

The flags can then be used to construct two different sections for this map.
The first section chooses, for each element in $\X^\times(S)$ the interior of the convex hull of the $\Gamma$--orbits of the fixed points in all flags, framed with the corresponding lines. This convex hull will be represented as an infinite \emph{union} of ideal triangles, and the set of these framed structures is denoted $\Ttp(S)\subset \Ttframe(S).$ The second section chooses, for each element in $\X^\times(S)$ the intersection of the $\Gamma$--orbits of half-spaces associated to the lines in all flags. This is represented as an infinite \emph{intersection} of ideal triangles and the set of these with the appropriate framing is denoted $\Ttpd(S)\subset \Ttframe(S).$ 
In particular, Theorem~\ref{thm:FG-main-X} follows from the following result.

\begin{thm}[Fock-Goncharov 2007, \cite{Fock-moduli-2007}]\label{thm:FG-main}
If $S = S_{g,n}$ has negative Euler characteristic and at least one puncture, then $\Ttp(S)$ and $\Ttpd(S)$ are open cells of dimension $- 8 \chi(S) = 16g-16+8n.$
\end{thm}

The key lies in a parametrisation of these spaces using an ideal triangulation of $S.$ Fock and Goncharov associate to each triangle in the triangulation a \emph{triple ratio of flags} and to each oriented edge a \emph{quadruple ratio of flags}, and show how to turn this into bijections $\psi^+ \co \RR_{>0}^{- 8 \chi(S)}\to \Ttp(S)$ and $\psi^- \co \RR_{>0}^{- 8 \chi(S)}\to \Ttpd(S)$ with the property that the following diagram commutes:

\begin{center}
\begin{tikzpicture}
  \matrix (m) [matrix of math nodes,row sep=3em,column sep=4em,minimum width=2em] {
  & \Ttp(S) &    \\
    \RR_{>0}^{- 8 \chi(S)}  & & \X^\times(S)  \\
    & \Ttpd(S) &   \\};
  \path[-stealth]
    (m-2-1) edge node [above] {$\psi^+$} (m-1-2)
    (m-2-1) edge node [below] {$\psi^-$} (m-3-2)
    (m-1-2) edge node [above]{$\mu^+$} (m-2-3)
    (m-3-2) edge node [below]{$\mu^-$} (m-2-3);
\end{tikzpicture}
\end{center}

The map $\mu^+\circ\psi^+ = \mu^-\circ\psi^-$ has a lift $\RR_{>0}^{- 8 \chi(S)}\to \Hol^\times(S)$ that is explicitly described in terms of rational functions where each numerator and denominator is a polynomial with only positive coefficients. These are the so-called \emph{positive representations}. For each $\star \in \{ +, -\}$, the corresponding \emph{monodromy map} $\mu^\star\co \Tt^\star(S)\to \X^\times(S)\to \X(S)$ is surjective and generically $6^n$--to--$1$, and corresponds to an action of (the direct product of $n$ copies of) the symmetric group in three letters.

We remark that the set $\Ttpm(S) = \Ttp(S) \cap \Ttpd(S)$ where the two constructions agree has the property that the composition $\Ttpm(S)\to \X^\times(S)\to \X(S)$ is surjective and generically $2^n$--to--$1.$ Many people we talked to in the initial phase of this project had the impression that Fock and Goncharov's coordinates parameterise the set $\Ttpm(S).$ This may arise from the comment after Definition 2.1 in \cite{Fock-moduli-2007} that the set $\Ttp(S)$ is a $2^n$--to--$1$ cover of the space of non-framed convex real projective structures on $S$ \emph{with geodesic boundary}.
In this interpretation, one compactifies those ends of $S= \Omega/\Gamma$ by adding a circle, where a corresponding supporting line in the framing meets the closure of $\Omega$ in more than one point. 

The duality principle of projective geometry gives rise to natural \emph{duality maps} $\sigma^+ \co \Ttp(S) \to \Ttpd(S)$ and $\sigma^- \co \Ttpd(S) \to \Ttp(S)$ defined by taking a framed structure to a dually framed dual structure. These are inverse to each other. There also is a natural dual map $\sigma \co \X^\times(S) \to \X^\times(S)$ taking the character of the representation $\rho$ to the character of the inverse transpose of the representation, and each flag to the dual flag. We again obtain a commutative diagram:

\begin{center}
\begin{tikzpicture}
  \matrix (m) [matrix of math nodes,row sep=3em,column sep=4em,minimum width=2em] {
  \Ttp(S) &   \X^\times(S) \\
   \Ttpd(S) &   \X^\times(S)\\};
  \path[-stealth]
    (m-1-1) edge node [above] {$\mu^+$} (m-1-2)
    (m-1-1) edge [bend right=10] node [left] {$\sigma^+$} (m-2-1)
     (m-2-1) edge [bend right=10] node [right] {$\sigma^-$} (m-1-1)
    (m-2-1) edge node [above]{$\mu^-$} (m-2-2)
    (m-1-2) edge node [left]{$\sigma$} (m-2-2)
    (m-2-2) edge  (m-1-2);
\end{tikzpicture}
\end{center}

In particular, the action of duality on the coordinate space $\RR_{>0}^{- 8 \chi(S)}$ is given by $(\psi^-)^{-1} \circ \sigma^+ \circ \psi^+ = (\psi^+)^{-1} \circ \sigma^- \circ \psi^-.$

We will give a proof of Theorem~\ref{thm:FG-main} using the methods of \cite{Fock-moduli-2007}. We also highlight various properties of the moduli space, including an explicit description of the duality map and the action of the mapping class group, in \S\ref{sec:Properties of the parameterisation}.
As consequences of Theorem~\ref{thm:FG-main} we then deduce Goldman's Theorem~\ref{thm:Goldman-main}, Marquis' Theorem \ref{thm:Marquis-main-intro}, as well as a classical result due to Fricke and Klein about the classical Teichm\"uller space of finite area hyperbolic structures on $S$ in \S\ref{sec:Applications of the parameterisation}.

%
%
%
%
%

As a last application, we focus on the Poisson structure on the moduli space $\Ttp(S)$ in \S \ref{sec:poisson-structure}.
Natural symplectic structures and Poisson structures may often be used to distinguish between homeomorphic spaces and open the door to the world of integrable systems; see Audin~\cite{Audin-lectures-1997} for a discussion in the context of this paper. 

In the case of closed surfaces, Goldman~\cite{Goldman-symplectic-1984} gives a completely general result via the intersection pairing. Let $G$ be a Lie group preserving a non-degenerate bilinear form on its Lie algebra., and let $\X_G(S) := \Hom(\pi_1(S), G)\sslash G$ be the  $G$--character variety of $S.$ Then Goldman \cite{Goldman-symplectic-1984} showed that $\X_G(S_{g,0})$ has a natural symplectic structure.

%
In the case of a punctured surface, where a natural symplectic structure is not available on $\X_G(S_{g,n})$, it can still be foliated by symplectic leaves making it into a Poisson variety \cite{Guruprasad-group-1997}. For $G = \SL(m,\RR)$, character varieties of surfaces with the same Euler characteristic are isomorphic as varieties, while they may be distinguished using their Poisson structure. A natural question to ask is whether this also applies to the moduli space $\Ttp(S).$ Following Fock and Goncharov~\cite{Fock-moduli-2007}, we explicitly construct a Poisson bracket on $\Ttp(S).$ 
The monodromy map $\mu^+$ induces an injective homomorphism $\mu^*$ between the spaces of smooth functions on $\Ttp(S)$ and on $\X(S)$. When we restrict to the trace algebra $\Tr(\X(S))$, we have the following result (see Theorem \ref{thm:compatibility-poisson-bracket} for details):

\begin{thm}
	$\mu^*$ is a Poisson homomorphism between $\Ttp(S)$ and $\X(S)$.
\end{thm}

We provide a detailed geometric proof of this result, which appears to be known to experts (see, for example, Audin \cite{Audin-lectures-1997} or Xie \cite{Xie-structures-2013}), but the homomorphism does not appear to have been described explicitly before.

\newpage
\section{Ratios and Configurations of Flags} \label{sec:ratios-and-config-flags}

This section is devoted to the development of the main tools used in the proof of Theorem \ref{thm:global-coord}, in \S\ref{sec:Moduli-Spaces}. Most of the definitions and conventions are taken from \cite{Fock-moduli-2007}. We introduce flags of $\RP2$ and study their properties. The rest of the chapter is completely devoted to developing the necessary tools to prove the main results in \S\ref{sec:Moduli-Spaces}: triple ratios, cross ratios, quadruple ratios and configurations of flags.


\subsection{Flags}

We introduce some notation and recall well-known results from projective geometry. The homomorphism $\GL(3,\RR) \to \SL(3,\RR)$ defined by
\[ A \mapsto (\det(A))^{-1/3} A\]
descends to an isomorphism $\PGL(3,\RR) \to \SL(3,\RR).$ We will therefore work with $\SL(3,\RR),$ which allows us to talk about eigenvalues of maps, and it will be clear from context whether an element of $\SL(3,\RR)$ acts on $\RP2$ or on $\RR^3.$ We also remark that the action on $\RR^3$ is by orientation preserving maps.

Points and lines of $\RP2$ are denoted by column and row vectors, respectively. In particular, a line $\left[ a : b : c \right]$ corresponds to the set of points $\left[ x : y : z \right]^t$ of $\RP2$, satisfying $ax + by + cz = 0$. Thus a point $P$ belongs to a line $l$ if and only if $l(P) = 0$. Four points (resp. four lines) of $\RP2$ are in \emph{general position} if no three are collinear (resp. no three are incident). They are often referred as a \emph{projective basis}, because:
\begin{itemize}
	\item $\SL(3,\RR)$ is simply transitive on ordered $4$--tuples of points in general position;
	\item $\SL(3,\RR)$ is simply transitive on ordered $4$--tuples of lines in general position.
\end{itemize}
Up to projective transformation, there is a canonical dual map between points and lines of $\RP2$. That is the following one:
\begin{align*}
\perp : \RP2 &\rightarrow (\RP2)^*,\\
 P = \left[
\begin{array}{c}
x\\ y\\ z
\end{array} \right] &\mapsto P^\perp =  \left[ x : y : z\right].
\end{align*}

A \emph{flag} $\F_i := (V_i,\eta_i)$ of $\RP2$ is a pair consisting of a point $V \in \RP2$ and a line $\eta \subset \RP2$ passing through $V$. An $m$--tuple of flags $\{\F_1,\dots,\F_m \}$ is in \emph{general position} if
\begin{itemize}
	\item no three points are collinear;
	\item no three lines are coincident;
	\item $\eta_i(V_j) = \delta_{ij}$, the Kronecker delta.
\end{itemize}
Henceforth, we will denote by $((\F_1,\dots,\F_m))$ a cyclically ordered $m$--tuple of flags, as opposed to an ordered $n$--tuple $(\F_1,\dots,\F_m)$. The group of projective transformations acts on the space of flags via
$$
T \cdot F_i := ( T \cdot V_i , T \cdot \eta_i ), \qquad T \in \SL(3,\RR).
$$
The action naturally extends to $m$--tuples, ordered $m$--tuples and cyclically ordered $m$--tuples of flags.


\subsection{Triple ratio}
\label{subsec_triple_ratio}

Suppose $\mathfrak{F} = ((\F_0,\F_1,\F_2))$ is a cyclically ordered triple of flags in general position. We define the \emph{triple ratio} of $\mathfrak{F}$ as
$$
\cancel{3}(\mathfrak{F}) := \frac{\overline{\eta_0}(\overline{V_1}) \cdot \overline{\eta_1}(\overline{V_2}) \cdot \overline{\eta_2}(\overline{V_0})}{\overline{\eta_0}(\overline{V_2}) \cdot \overline{\eta_1}(\overline{V_0}) \cdot \overline{\eta_2}(\overline{V_1})},
$$
where $\overline{V_i}$ and $\overline{\eta_i}$ are some class representatives. The above definition is well-defined, as it is independent of the choice of representatives, and manifestly invariant under a cyclic permutation of the flags. The following property is easy to verify.

\begin{lem} \label{lem:triple_ratio_proj_invariant}
	$\cancel{3}$ is invariant under projective transformations, $i.e.$ for all $T \in\SL(3,\RR)$
	$$
	\cancel{3}(T \cdot \mathfrak{F}) = \cancel{3}(\mathfrak{F}).
	$$
\end{lem}


\subsection{Cross ratio}
\label{subsubsec:positivity_of_cross}

Let $l \subset \RP2$ be a line. Given $P_0,P_1,P_2,P_3 \in l$ with $P_0,P_1,P_2$ pairwise distinct, let $T: l \rightarrow \RR \PP^1$ be the unique projective map such that $T(P_0) = \infty$, $T(P_1) = -1$ and $T(P_2) = 0$. Then the \emph{cross ratio} of the ordered quadruple $(P_0,P_1,P_2,P_3)$ is
$$
\cross(P_0,P_1,P_2,P_3) = T(P_3).
$$
Just as for the triple ratio, it is easy to check that the cross ratio is invariant under projective transformation. 
Moreover, if $x_0,x_1,x_2,x_3$ are local coordinates for $P_0,P_1,P_2,P_3 \in l$, then
	$$
	\cross(P_0,P_1,P_2,P_3) = \frac{(x_0-x_1)(x_2-x_3)}{(x_0-x_3)(x_1-x_2)}.
	$$
	It follows that, if $\sigma$ is a permutation on four symbols and $\lambda = \cross(P_0,P_1,P_2,P_3)$, then
	$$
	\cross(P_{\sigma(0)},P_{\sigma(1)},P_{\sigma(2)},P_{\sigma(3)}) \in \left\{ \lambda,\frac{1}{\lambda},1-\lambda,\frac{1}{1-\lambda},1-\frac{1}{\lambda},\frac{\lambda}{\lambda-1}\right\}.
	$$
	In particular 
	$\cross(P_3,P_2,P_1,P_0) = \cross(P_0,P_1,P_2,P_3)$.

Similarly, we can define the cross ratio of four incident lines via duality. Let $P \in \RP2$ be a point and $l_0,l_1,l_2,l_3$ lines through $P$ with $l_0,l_1,l_2$ pairwise distinct. Then the \emph{cross ratio} of the ordered quadruple $(l_0,l_1,l_2,l_3)$ is 
$$
\cross(l_0,l_1,l_2,l_3) := \cross(l_0^\perp,l_1^\perp,l_2^\perp,l_3^\perp).
$$

A straightforward argument shows:
\begin{lem} \label{lem:cr_lines_points}
	Let $m$ be a line intersecting the lines $l_0,l_1,l_2,l_3$ transversely in $Q_0,Q_1,Q_2,Q_3$ respectively, then
	$$
	\cross(l_0,l_1,l_2,l_3) = \cross(Q_0,Q_1,Q_2,Q_3).
	$$
\end{lem}

The convention for the cross ratio used here was suggested by Fock and Goncharov~\cite[pg. 253]{Fock-moduli-2007}. It is motivated by the fact that, if $P_0,P_1,P_2$ are pairwise distinct points on a line $l \subset \RP2$, we would like $\cross(P_0,P_1,P_2,P_3) > 0$ when $P_3$ and $P_1$ lie in different components of $l\setminus \{P_0,P_2\}$. Lemma \ref{lem:cr_lines_points} implies that if $l_0,l_1,l_2$ are pairwise distinct lines passing through a point $P \in \RP2$, then $\cross(l_0,l_1,l_2,l_3) > 0$ when $l_3$ and $l_1$ lie in different regions of $\RP2 \setminus \{l_0,l_2\}$. See Figure~\ref{cross_ratios}.

\begin{figure}[ht]
	\centering
	\includegraphics[width=12.5cm]{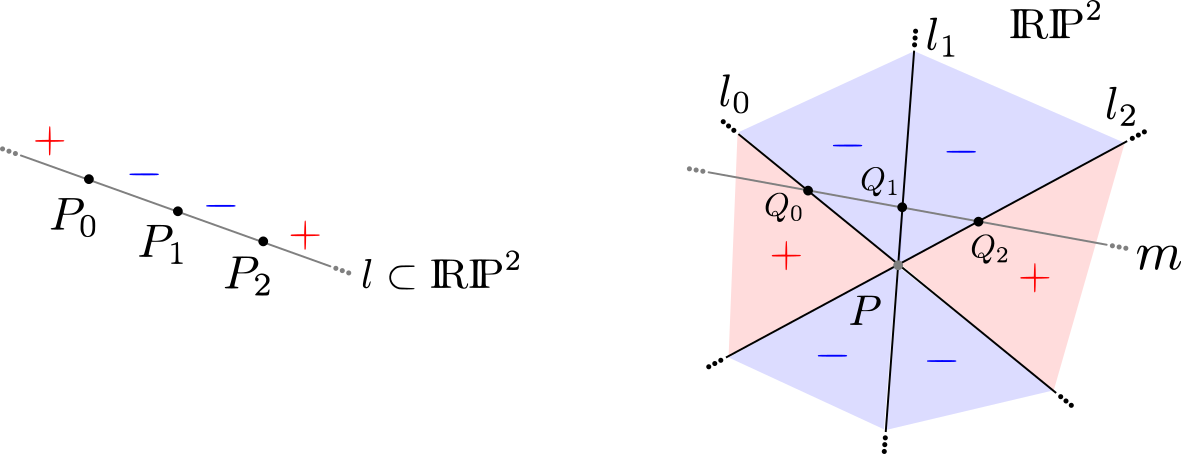}
	\caption{$\cross(P_0,P_1,P_2,P_3) > 0$ when $P_3$ and $P_1$ lie in different components of $l\setminus \{P_0,P_2\}$. Likewise, $\cross(l_0,l_1,l_2,l_3) > 0$ when $l_3$ and $l_1$ lie in different regions of $\RP2 \setminus \{l_0,l_2\}$.}
	\label{cross_ratios}
\end{figure}

Henceforth, if $P,Q \in \RP2$ are two points and $m,l \subset \RP2$ are two lines, we will denote by $PQ$ the line passing through $P$ and $Q$, and by $lm$ the point of intersection between $l$ and $m$.

\begin{thm} \label{thm:cross_triple}
	Let $\mathfrak{F} = ((\F_0,\F_1,\F_2))$ be a cyclically ordered triple of flags in general position where $\F_i = (V_0,\eta_0)$. Then
	$$
	\cross(\eta_0,V_0V_1, V_0(\eta_1\eta_2),V_0 V_2) =\cancel{3}(\mathfrak{F}).
	$$
\end{thm}

\begin{proof}
	Both cross ratio and triple ratio are projectively invariant, and the points $V_0,V_1,V_2, \eta_1 \eta_2$ are in general position so we may assume, without loss of generality, that
	$$
	V_0 = \left[
	\begin{matrix}
	0 \\ 0 \\ 1
	\end{matrix} \right], \quad
	V_1 = \left[
	\begin{matrix}
	1 \\ 0 \\ 1
	\end{matrix} \right], \quad
	V_2 = \left[
	\begin{matrix}
	0 \\ 1 \\ 1
	\end{matrix} \right], \quad
	\eta_1 \eta_2 = \left[
	\begin{matrix}
	1 \\ 1 \\ 1
	\end{matrix} \right].
	$$
	It follows that
	$$
	\eta_1 = \left[
	\begin{matrix}
	1 : 0 : -1
	\end{matrix} \right], \quad
	\eta_2 = \left[
	\begin{matrix}
	0 : 1 : -1
	\end{matrix} \right].
	$$
	The projective line $\eta_0$ passes through $V_0$ but does not pass through $V_2$ or $V_1$, so $\eta_0 = [A:1:0]$ for some $A \in \RR \setminus \{0\}$. By Lemma \ref{lem:cr_lines_points},
	$$
	\cross(\eta_0,V_0V_1, V_0(\eta_1\eta_2), V_0 V_2 ) = \cross( \eta_0 \eta_2 , (V_0V_1)\eta_2,  (V_0(\eta_1\eta_2))\eta_2 , (V_0 V_2)\eta_2 ).
	$$
	A direct calculation shows that
	\[
	\cancel{3}(\mathfrak{F}) = A = \cross( \eta_0 \eta_2 , (V_0V_1)\eta_2 , (V_0(\eta_1\eta_2))\eta_2 , (V_0 V_2)\eta_2 ). 
	\]
	This completes the proof.
\end{proof}
	
	\begin{figure}[t]
		\centering
		\includegraphics[width=5cm]{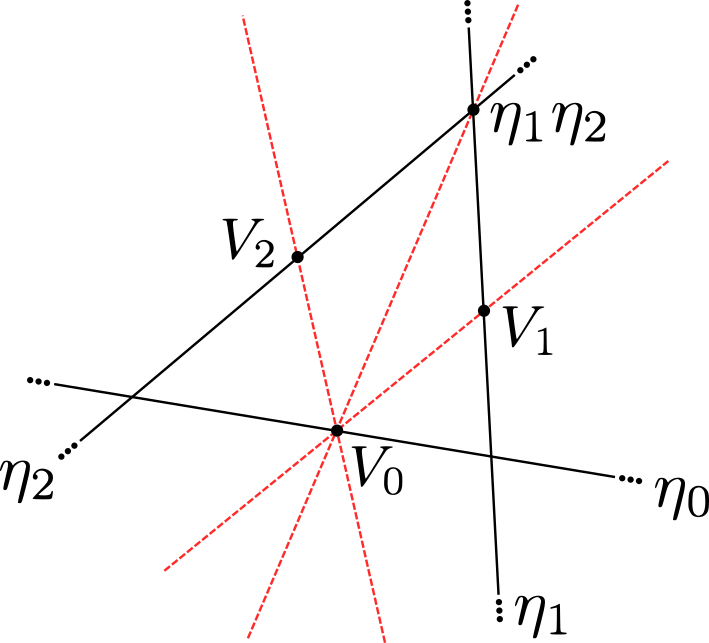}
		\caption{The cross ratio $\cross( V_0 V_2 , V_0(\eta_1\eta_2) , V_0 V_1 , \eta_0 )$ .}
		\label{cross_and_triple}
	\end{figure}


\subsection{Quadruple ratio}
\label{subsubsec:Quadruple ratio}

Let $\mathfrak{F} = (\F_0,\F_1,\F_2,\F_3)$ be an ordered quadruple of flags in general position, $\F_i = ( V_i , \eta_i )$. We define the \emph{quadruple ratio} of $\mathfrak{F}$ as
$$
\cancel{4}(\mathfrak{F}):= \cross( \eta_0 , V_0V_3 , V_0V_2 , V_0V_1 ).
$$
$\cancel{4}(\mathfrak{F})$ is sometimes referred as \emph{edge ratio} with respect to $V_0 V_2$. We give an intuitive description of this definition in $\S~\ref{subsubsec:param_of_flags}$. It follows from Theorem~\ref{thm:cross_triple} that
	\[
	\cancel{4}(\mathfrak{F}) = \cancel{3}\big( \big( ( V_0 , \eta_0 ) , ( V_3 , V_2 V_3 ) , ( V_1 , V_1V_2 ) \big) \big).
	\]
The requirement that $(\F_0,\F_1,\F_2,\F_3)$ are flags in general position may be relaxed, but this is not needed for this paper.

\begin{figure}[t]
	\centering
	\includegraphics[width=12.5cm]{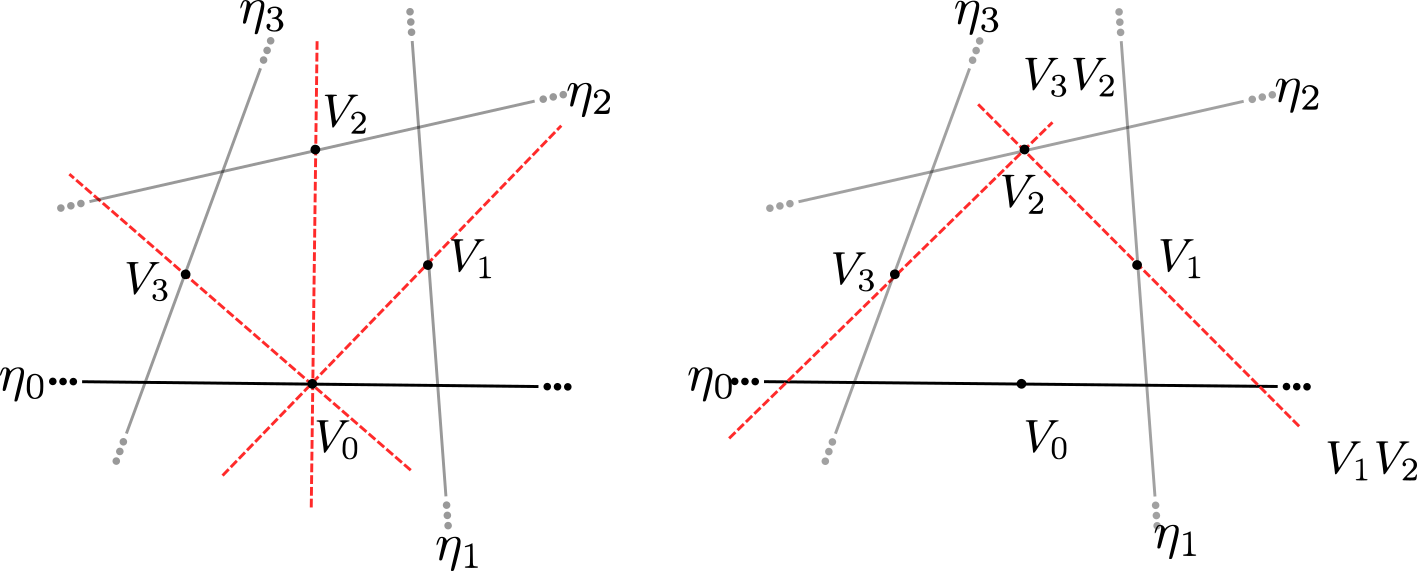}
	\caption{The quadruple ratio $\cancel{4}(\mathfrak{F})$ can be read off both in terms of $\cross( \eta_0 , V_0V_3 , V_0V_2 , V_0V_1 )$ and $\cancel{3}\left( \left( ( V_0 , \eta_0 ) , ( V_3 , V_2V_3 ) , ( V_1 , V_1V_2 ) \right) \right) $.}
	\label{quadruple_ratio}
\end{figure}	


\subsection{Pairs of convex polygons and configuration of flags } \label{subsec_def_toymodel}

When $m$ flags are configured properly, they form two convex polygons, one inscribed into the other. In the remainder of this section we parametrise such configurations for $m =3,4$ using the triple ratio and quadruple ratio. This will be a key ingredient in the proof of Theorem \ref{thm:global-coord}.

We say that a non-degenerate $m$--gon $\sigma_1$ is \emph{strictly inscribed} into another non-degenerate $m$--gon $\sigma_2$ if each edge of $\sigma_2$ contains one vertex of $\sigma_1$ in its interior. A pair of oriented $m$--gons is \emph{oriented} if the two $m$--gons have the same orientation. Furthermore, an oriented pair of strictly inscribed $m$--gons is \emph{marked} by a choice of a preferred vertex of the inscribed polygon (equiv. a preferred edge of the circumscribed polygon) and subsequent ordering of the vertices of the inscribed polygon (equiv. the edges of the circumscribed polygon) starting from the preferred vertex and following the orientation of the pair.

Let $\P_m$ be the space of oriented pairs of strictly inscribed convex $m$--gons in $\RP2$, and $\P^*_m$ be the space of marked oriented pairs of strictly inscribed convex $m$--gons in $\RP2$, both modulo the action of $\SL(3,\RR)$.

Similarly, $\Conf_m$ is the space of cyclically ordered $m$--tuples of flags in general position and $\Conf^*_m$ is the space of ordered $m$--tuples of flags in general position, both modulo the action of $\SL(3,\RR)$. There are natural maps
$$
\xi_m : \P_m \rightarrow \Conf_m \qquad \mbox{ and } \qquad \xi^*_m : \P^*_m \rightarrow \Conf^*_m,
$$
defined by deleting the interior of the edges of the inscribed polygon and extending the edges of the circumscribed polygon to lines. The fact that the polygons are assumed to be strictly inscribed one into the other ensures that the corresponding flags are in general position. Moreover, two pairs of polygons are projectively equivalent if and only if the associated flags are projectively equivalent. Hence $\xi_m$ and $\xi^*_m$ are injective and hence $\P_m$ and $\P^*_m$ can be identified with their images in $\Conf_m$ and $\Conf^*_m$ respectively. On the other hand, these maps are not surjective (see \S\ref{subsubsec:not injective} below). We give a characterisation of $\P_3$ and $\P^*_4$ in $\S\ref{subsubsec:param_of_flags}$.


\subsection{Inscribed triangles and triples of flags}
\label{subsubsec:not injective}\label{rem:example}

	We present an example where we relate $\P_3$, $\Conf_3$ and convex domains in $\RP2$. Let
\[
v_A := \begin{bmatrix} 1 \\ 0 \\ 0 \end{bmatrix}, \qquad
v_B := \begin{bmatrix} 0 \\ 0 \\ 1 \end{bmatrix}, \qquad
v_C := \begin{bmatrix} 1 \\ 1 \\ 1 \end{bmatrix}.
\]
	Now consider the cyclically ordered triple of flags $\mathfrak{F} = ((\F_A, F_B, \F_C))$, where
\[
	\F_A:= ( v_A ,  v_B^{\perp} ), \qquad
	\F_B:=( v_B ,  v_A^{\perp} ), \qquad
	\F_C:= ( v_C ,  \eta_C ), \qquad  \eta_C := [ -t : 1 + t :  -1 ], \quad t \in \RR.
\]
	This triple is in general position if and only if $t \neq 0, -1$, as these are the cases where $ \eta_C $ passes through the points $v_A$ and $v_A^\perp v_B^\perp$ respectively. The case $t = \infty$ would also violate generality as this would ensure that $\eta_C$ passes through $ v_B $ . 

	There are four triangular regions in $\RP2$ with vertices $v_A, \; v_B$ and $v_C$. The triple $\mathfrak{F}$ is in the image of $\xi_3$ if and only if one of these triangles is disjoint from each of the lines $ v_B^\perp $, $ v_A^\perp $ and $ \eta_C $. The lines $ v_B^\perp $ and $ v_A^\perp $ pass through three of the four potential triangles. The only remaining option is the triangle which is strictly contained in the affine patch $ \A := \{ [x: y: 1 ] \ | \  x,y \in \RR \} \subset \RP2$. Therefore, $\mathfrak{F} \in \xi_3(\P_3)$ if and only if $ \eta_C $ intersects the line $ v_A v_B $ at a point of the form $[x : 0 : 1]$, where $x<0$. This is the case if and only if $t > 0$.

	A direct calculation shows that $\cancel{3}(\mathfrak{F}) = t$. So the triple ratio may be used to determine whether $\mathfrak{F} \in \xi_3(\P_3)$. Moreover, there is a unique conic $\mathcal{C}$ passing through $v_A, v_B$ and $v_C$ with supporting projective lines $v_A^{\perp} $ and $ v_B^{\perp} $, namely the set of points
\[
\begin{bmatrix}
 x^2 \\ xz \\ z^2 
\end{bmatrix}, \qquad (x, z) \in \RR^2 \backslash (0,0).
\]
The line defined by $ \eta_C $ is a supporting line to $\mathcal{C}$ at $ v_C $ if and only if $t=1$. As projective transformations preserve the set of conics, this triple ratio test may be used to determine whether an element of $\mathcal{P}_3$ inscribes a conic. We will expand further upon this in \S\ref{subsec:classical-teich}.
	\begin{figure}[ht]
		\centering
		\includegraphics[width=12.5cm]{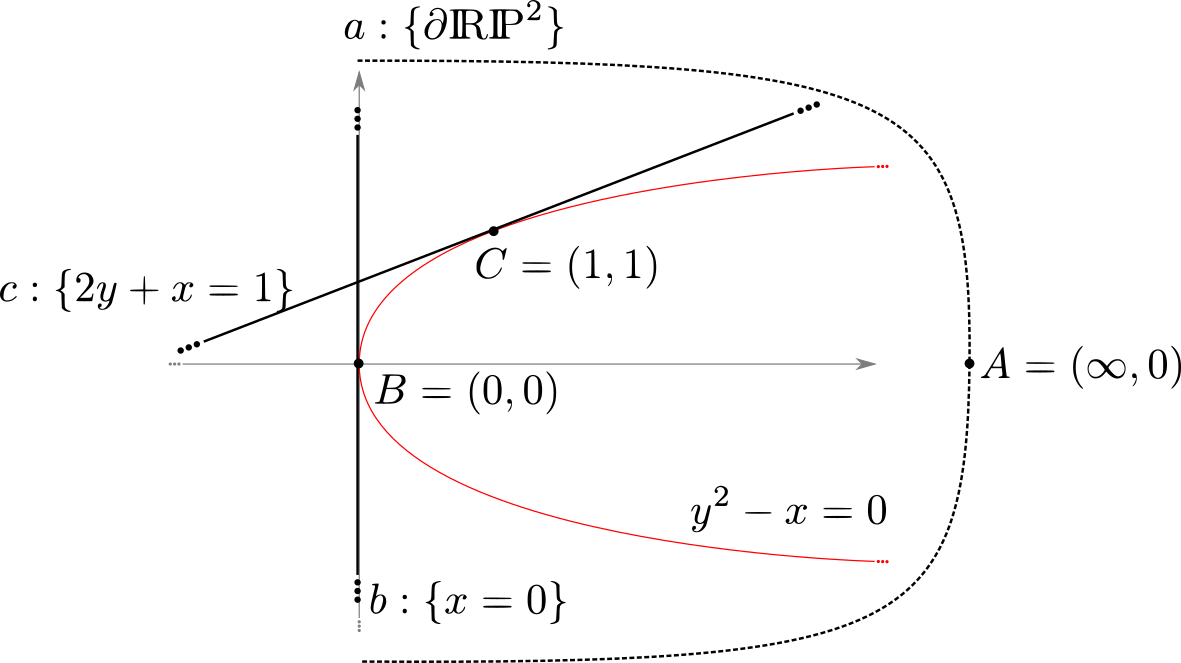}
		\caption{When $t=1$, $((\F_A, \F_B, \F_C))$ corresponds to an element of $\P_3$ and the line $\eta_{C} = [-1:2:-1]$ is a supporting projective line to $\mathcal{C}$ at the point $v_C$.}
		\label{example}
	\end{figure}


\subsection{Parametrisation of the spaces $\P_3$ and $\P^*_4$} \label{subsubsec:param_of_flags}

We give a parametrisation of $\P_3$ and $\P^*_4$ using triple ratios and quadruple ratios. This will used in the proof of Theorem \ref{thm:global-coord}. Define the map $ \cancel{3} \colon \P_3 \rightarrow \RR \PP^1 $ as follows. For $\chi \in \P_3$, choose a representative $\mathfrak{F}$ of $\xi_3(\chi)$. Then $\cancel{3}(\chi) := \cancel{3}(\xi_3(\chi)) = \cancel{3}(\mathfrak{F})$ is well-defined by Lemma \ref{lem:triple_ratio_proj_invariant}.

\begin{thm} \label{thm:param_P_3}
	$\cancel{3}$ establishes a bijection between $\P_3$ and $\RR_{>0}$.
\end{thm}

\begin{proof}
	Let $\chi \in \P_3$ and fix a representative $\mathfrak{F} = ((\F_0,\F_1,\F_2))$ of $\xi_3(\chi)$, where $\F_i = ( V_i , \eta_i )$ for $ i= 0, 1, 2$. As in Theorem \ref{thm:cross_triple}, after a projective transformation, we can assume
	$$
	V_0 = \left[
	\begin{matrix}
	0 \\ 0 \\ 1
	\end{matrix} \right], \quad
	V_1 = \left[
	\begin{matrix}
	1 \\ 0 \\ 1
	\end{matrix} \right], \quad
	V_2 = \left[
	\begin{matrix}
	0 \\ 1 \\ 1
	\end{matrix} \right], \quad
	\eta_1 \eta_2  = \left[
	\begin{matrix}
	1 \\ 1 \\ 1
	\end{matrix} \right],
	$$
	$$
	 \eta_1  = \left[
	\begin{matrix}
	1 : 0 : -1
	\end{matrix} \right], \quad
	 \eta_2 = \left[
	\begin{matrix}
	0 : 1 : -1
	\end{matrix} \right], \qquad
	 \eta_0 = \left[
	\begin{matrix}
	t : 1 : 0
	\end{matrix} \right], \quad t \in \RR.
	$$
	Recall that $\cancel{3}(\mathfrak{F}) \in \RR \setminus \{0\}$ because $ \eta_0 $ is a line through $ V_0 $ which is disjoint from $ V_2 $ or $ V_1 $. In this setting, $\chi$ is uniquely determined by $\cancel{3}(\mathfrak{F})$ so injectivity is immediate. Theorem \ref{thm:cross_triple} and the discussion in \S\ref{subsubsec:positivity_of_cross} imply that $\cancel{3}(\mathfrak{F})$ is positive if and only if $ \eta_0 $ and $ V_0 (\eta_1 \eta_2 )$ lie in different regions of $\RP2 \setminus \{ V_0V_1 , V_0V_2 \}$. It follows that $\cancel{3}(\mathfrak{F})>0$ if and only if $\mathfrak{F}$ represents an element of $\P_3$, proving that $\cancel{3} (\P_3) = \RR_{>0}$.
\end{proof}

Now we are going to show something similar for $\P^*_4$ using both triple ratios and quadruple ratios. We define $g \colon P^*_4 \rightarrow (\RR \PP^1)^4 $ as follows. For $\chi \in \P^*_4$, choose a representative $\mathfrak{F} = (\F_0,\F_1,\F_2,\F_3)$ of $\xi^*_4(\chi)$, with $\F_i = ( V_i , \eta_i )$ for $ i = 0,1,2,3$. Then $g(\chi) := (t_{012},t_{023},e_{02},e_{20})$, where
\begin{align*}
t_{012} :=& \cancel{3}((\F_0,\F_1,\F_2)),\\
t_{023} :=& \cancel{3}((\F_0,\F_2,\F_3)),\\
e_{02} :=& \cancel{4}(\F_0, \F_1, \F_2, \F_3) = \cancel{3}(( ( V_0 , \eta_0 ),( V_3 , V_2V_3 ) , ( V_1 , V_1V_2 ) )), \\
e_{20} :=& \cancel{4}(\F_2, \F_3, \F_0, \F_1) = \cancel{3}(( ( V_2 , \eta_2 ),( V_1 , V_0V_1 ) , ( V_3 , V_3V_0 ) )).
\end{align*}
Once again, $g$ is well-defined due to the projective invariance of triple ratio. One may visualise $\chi$ as in Figure \ref{fig:edge_coordinate}, with an additional edge crossing from $V_0$ to $V_2$. This motivates why $\cancel{4}(\mathfrak{F})$ is also called the \emph{edge ratio} of the oriented edge $V_0 V_2$. 

\begin{figure}[ht]
	\centering
	\includegraphics[width=5cm]{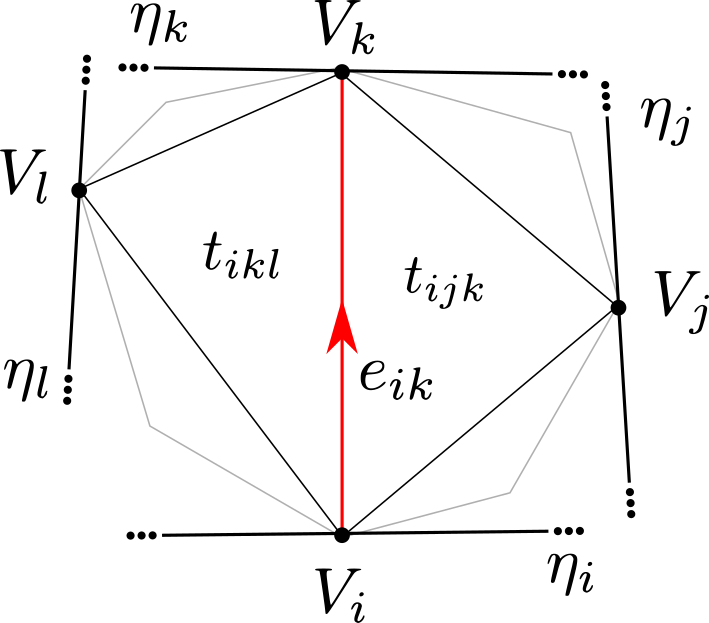}
	\caption{$e_{02}$ is the edge oriented from $V_0$ to $V_2$, adjacent to the two triangles $t_{012},t_{023}$.}
	\label{fig:edge_coordinate}
\end{figure}	

\begin{thm} \label{thm:param_P^*_4}
	$g$ establishes a bijection between $\P^*_4$ and $\RR^4_{>0}$.
\end{thm}

\begin{proof}	
	Let $\chi \in \P^*_4$ and choose a representative $\mathfrak{F} = (\F_0,\F_1,\F_2,\F_3)$ of $\xi^*_4(\chi)$, with $\F_i = ( V_i , \eta_i )$. As in Theorem \ref{thm:param_P_3}, we assume without loss of generality that $V_0$, $V_2$, $V_3$ and $\eta_2 \eta_0$ are fixed at an arbitrary generic quadruple of points in some affine patch $\A$ of $\RP2$.
	
	\begin{figure}[ht]
		\centering
		\includegraphics[width=.99\linewidth]{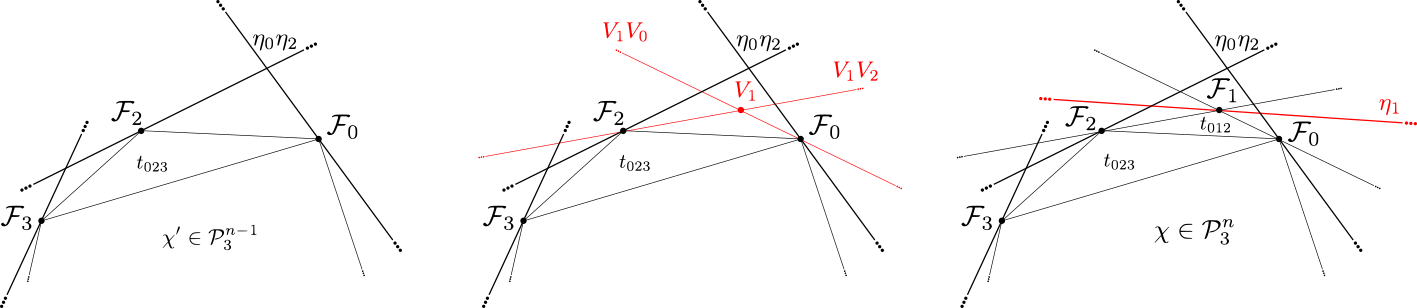}
		\caption{$V_1V_0$ and $V_1V_2$ are uniquely determined by $\cancel{4}(\F_0,\F_1,\F_2,\F_3)$ and $\cancel{4}(\F_2,\F_3,\F_0,\F_1)$, while $\eta_1$ is uniquely determined by $V_1$ and $\cancel{3}((\F_0,\F_1,\F_2))$.}
		\label{fig:triangulation_induction}
	\end{figure}

	Having fixed these points, the lines $\eta_2$ and $\eta_0$ are determined. By Theorem \ref{thm:param_P_3}, the line $\eta_3$ is uniquely determined by $t_{023}$, and any assignment of positive real number to $t_{023}$ gives rise to a unique, well-defined element of $\P_3$. This already ensures the injectivity of $g$.
	
	If necessary, we rearrange the affine patch $\A$ to contain this pair of triangles. In $\A$, a polygon may be referred to as the convex hull of its vertices without ambiguity.
	
	Recall from \S\ref{subsubsec:Quadruple ratio} that the quadruple ratios $e_{02}$ and $e_{20}$ can be expressed as triple ratios
	$$ 
	e_{02} = \cancel{3}( ( V_0 , \eta_0 ) , ( V_3 , V_2V_3 ) , ( V_1 , V_1V_2 ) ) \mbox{ and } e_{20} = \cancel{3}( ( V_2 , \eta_2 ) , ( V_1 , V_1V_0 ) , ( V_3 , V_3V_0 ) ).
	$$ 
	As in Theorem \ref{thm:param_P_3}, the lines $V_0V_1$ and $V_2V_1$ are uniquely determined by $e_{02}$ and $e_{20}$, so $V_1$ is uniquely determined. Furthermore, 
	$V_1$ belongs to the triangle $\langle V_0,V_2,\eta_0 \eta_2 \rangle$ if and only if $e_{02} >0$ and $e_{20}>0$. Equivalently, the quadrilateral $\langle V_0,V_1,V_2,V_3 \rangle$ in $\A$ is convex if and only if $e_{02}, e_{20} \in \RR_{>0}$. This construction is shown in Figure \ref{fig:triangulation_induction}.

	Given that $V_1$, $\F_0$ and $\F_2$ are now fixed, the line $\eta_1$ is uniquely determined by $t_{012}$, once again appealing to Theorem \ref{thm:param_P_3}. Since $\chi$ belongs to $\P^*_4$, $\eta_1$ does not intersect the triangle $\langle V_0,V_1, V_2 \rangle$. That happens if and only if $t_{012} >0$. Together with the previous discussion on $t_{012},t_{023}$ and $e_{02}$, this concludes the proof that $g$ surjects onto $\RR^4_{>0}$.
\end{proof}

The \emph{cyclic group of order four} $C_4 := \langle \alpha \ | \ \alpha^4 = 1 \rangle$, acts on $\P^*_4$ by removing the marking, namely $\P^*_4 / C_4 \cong \P_4$. The corresponding action of $C_4$ on $\RR^4_{>0}$ can be thought of as the change of coordinates:
$$
\alpha \cdot \left( t_{012},t_{023},e_{02},e_{20} \right) := \left(t_{123},t_{013},e_{13},e_{31} \right),
$$
where
\begin{align*}
t_{123}= \frac{t_{012} ( e_{02} t_{023}  e_{20}+ e_{02} t_{023}+ e_{02}+1)}{ e_{02} t_{012}  e_{20}+t_{012}  e_{20}+ e_{20}+1},\quad & \quad
t_{013}= \frac{t_{023} ( e_{02} t_{012}  e_{20}+t_{012}  e_{20}+ e_{20}+1)}{ e_{02} t_{023}  e_{20}+ e_{02} t_{023}+ e_{02}+1},\\
e_{13} = \frac{ e_{20}+1}{( e_{02}+1) t_{012}  e_{20}},\quad & \quad
e_{31} = \frac{ e_{02}+1}{ e_{02} t_{023} ( e_{20}+1)}.
\end{align*}
This change of coordinates will be analysed again in $\S\ref{subsec:change-of-coord}$.

\begin{figure}[h]
	\centering
	\includegraphics[width=.9\linewidth]{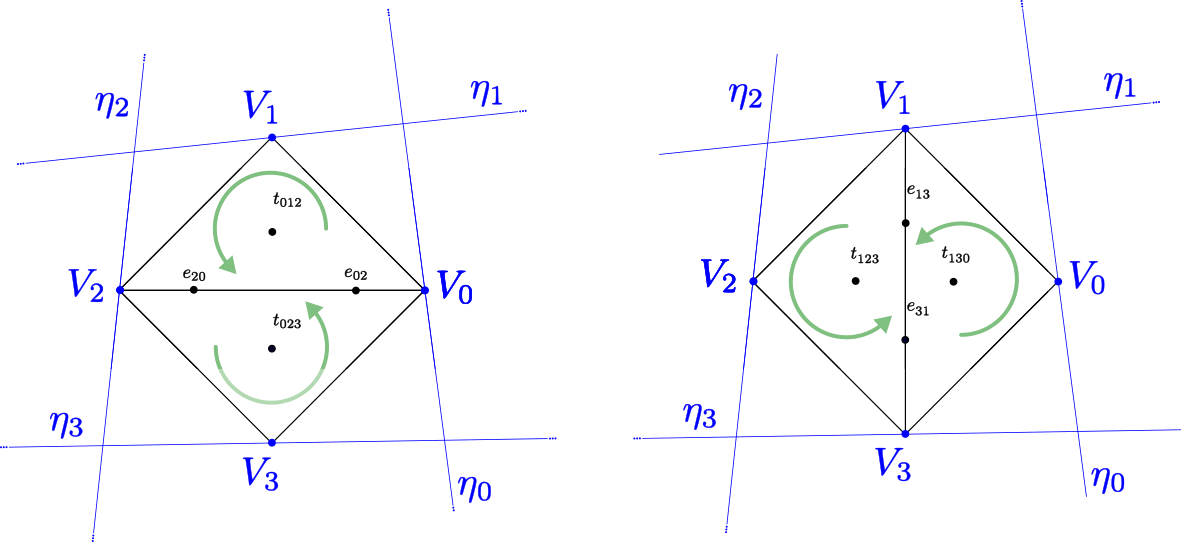}
	\caption{The change of coordinates for the action of $C_4$. The quadrilateral formed by $V_i$, $i=0, \dots, 3$ is depicted in an affine patch such that the triangles are oriented anticlockwise.}
	\label{fig:change_of_coord_C4}
\end{figure}

\newpage
\section{The parameterisation of the moduli space} \label{sec:Moduli-Spaces}

Ratios of flags can be viewed as the algebraic underpinning of Fock and Goncharov's parameterisation of the moduli spaces of interest in this paper. In addition, the parameterisation is based in geometry and topology through framings of the ends and ideal triangulations.
We will now describe these ingredients and then give a self-contained construction of the parameterisation of Fock-Goncharov. Our proofs in this section are different from those found in \cite{Fock-moduli-2007}.


\subsection{Framing the ends}
\label{subsec:classification_of_ends}

In the introduction, we described the spaces $\Ttp(S)$ and $\Ttpd(S)$ corresponding to framed structures where each end is either \emph{minimal} or \emph{maximal}, and the distinction between the spaces arises from the framings that are allowed for a hyperbolic end. In this section, we give a different definition of these spaces and show in the proofs of Theorems~\ref{thm:global-coord} and \ref{thm:global-coord-dual} that these definitions are in fact equivalent.

We recall the notation from the introduction: $S_g$ denotes a closed, orientable surface of finite genus $g$; $D_1, \ldots, D_n$ are open discs on $S_g$ with pairwise disjoint closures; and $p_k \in D_k.$ Then $S_{g,n} = S_g \setminus \{p_1, \ldots, p_n\}$ is a punctured surface and $E_k = \overline{D_k} \setminus \{p_k\}$ is an end of $S_{g,n}.$ A compact core of $S_{g,n}$ is $S_{g,n}^c = S_g \setminus \cup D_k.$ We also write $S_g = S_{g,0}.$ Furthermore, we fix a marked convex projective structure $(\Omega,\Gamma,f)$ on $S= S_g$, with $\Omega = \dev(\widetilde{S})$ and $\Gamma = \hol(\pi_1(S))$. For each end $E_k$, we identify $\pi_1(E_k)$ with its image in $\pi_1(S)$ and call it a \emph{peripheral subgroup} of $\pi_1(S).$ Similarly, $\Gamma_k := \hol(\pi_1(E_k))$ is called a \emph{peripheral subgroup} of $\Gamma$.

Let $\gamma$ be a generator of the peripheral subgroup $\pi_1(E)$ of an end $E$. Since $\hol(\gamma)$ preserves the convex domain $\Omega = \dev(\widetilde{S})$, $\hol(\gamma) \in \SL(3,\RR)$ has three positive real eigenvalues, counted with multiplicity. By a standard argument of convex projective geometry (see \cite[section $2$, pg. 193]{Cooper-convex-2015} ), $\hol(\gamma)$ is conjugate to one of the following three Jordan forms:
$$
\stackrel{
	J_1 = \left(
	\begin{array}{ccc}
	\lambda_1 & 0 & 0 \\
	0 & \lambda_2 & 0 \\
	0 & 0 & \frac{1}{ \lambda_1 \lambda_2} \\
	\end{array}
	\right)
\qquad
	J_2 = \left(
	\begin{array}{ccc}
	\lambda_1^2 & 0 & 0 \\
	0 & \frac{1}{\lambda_1 } & 1 \\
	0 & 0 & \frac{1}{\lambda_1} \\
	\end{array}
	\right)
\qquad
	J_3 = \left(
	\begin{array}{ccc}
	1 & 1 & 0 \\
	0 & 1 & 1 \\
	0 & 0 & 1 \\
	\end{array}
	\right)
}{\lambda_1 > \lambda_2 > \frac{1}{ \lambda_1 \lambda_2} \hspace{3cm} \lambda_1 \neq 1  \hspace{4cm}   }
$$

In the first case, $\hol(\gamma) \sim J_1$ fixes three pairwise distinct points of $\RP2$ and preserves the lines through them. We say that it is \emph{totally hyperbolic}. We refer to $E$ as a \emph{hyperbolic end}. In the second case, $\hol(\gamma) \sim J_2$ fixes only two distinct points of $\RP2$, preserves the line through them and acts as a unipotent transformation on a second line. In this case we will say that $\hol(\gamma)$ is \emph{quasi-hyperbolic} and $E$ is a \emph{special end}. In the last case, $\hol(\gamma) \sim J_3$ is \emph{parabolic}, it fixes a unique point and preserves a unique line through it. Hence $E$ is a \emph{cusp}.

\begin{figure}[h]
	\centering
	\includegraphics[width=0.9\textwidth]{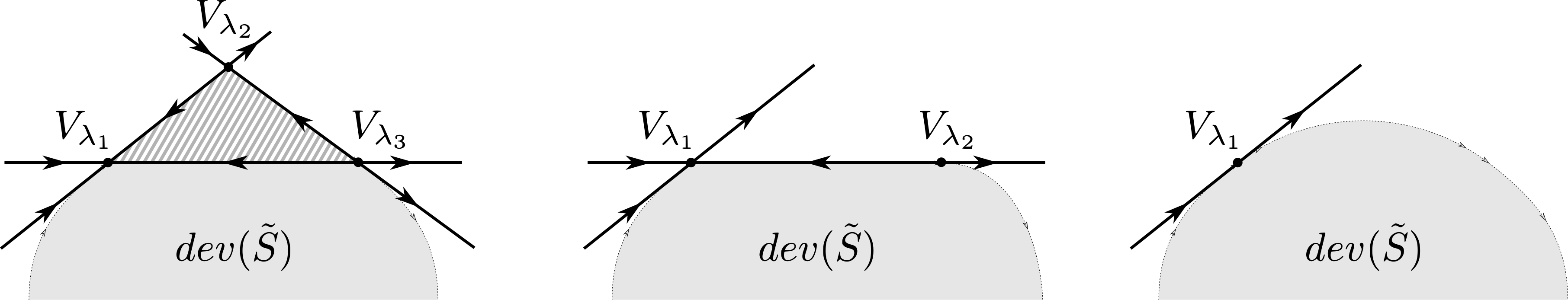}
	\caption{Actions of a totally hyperbolic, quasi hyperbolic and parabolic peripheral element, respectively.}
	\label{peripheral_action}
\end{figure}

The action of $\hol(\pi_1(E))$ preserves $\Omega$, so $\bOmega$ always contains the attracting and repelling fixed points of the peripheral element $\hol(\gamma)$ (whether or not they are distinct). In particular, when $E$ is a hyperbolic or special end, the interior of a segment between the the attracting and repelling fixed points is either contained in $\Omega$ or $\bOmega$. Fix an affine patch $\mathcal{A}$ containing $\cOmega$. A hyperbolic end is \emph{maximal} if $\Omega$ contains the interior (with respect to $\mathcal{A}$) of a triangle spanned by the three fixed points of $\hol(\gamma)$, and it is \emph{minimal} if $\Omega$ is contained in the complement of this triangle. The triangle in question is shaded in the left image of Figure \ref{peripheral_action}. Note that $\Omega$ cannot be equal to the interior of such a triangle since otherwise $S$ would itself be an end, hence an annulus, contradicting our hypothesis on the topology of $S$. We define $\Ttm(S) \subset \Tt(S)$ to be the subset of structures for which \emph{every hyperbolic end is either maximal or minimal}.

Each end may also be endowed with a \emph{framing}. There are two definitions which are related by projective duality.
A \emph{positive framing} of $(\Omega, \Gamma, f) \in \Ttm(S)$ is a choice, for each end $E_k$, of a pair consisting of:
\begin{itemize}
	\item if $E_k$ is a maximal hyperbolic end: the $\Gamma$--orbit of the saddle point $V$ of $\Gamma_k$ and the $\Gamma$--orbit of a supporting line $\eta$ through $V$ that is invariant under $\Gamma_k$;
	\item otherwise: the $\Gamma$--orbit of a fixed point $V$ of $\Gamma_k$ in the frontier of $\Omega$, and the $\Gamma$--orbit of a supporting line $\eta$ through $V$ that is invariant under $\Gamma_k$.
\end{itemize}
Under this definition, there are two possible positive framings for each maximal end, for each special end there are three, and for each minimal end there are four. Each cusp has a unique positive framing. Two positively framed structures are equivalent if and only if they represent the same element of $\Ttframe(S)$. We denote by $\Ttp(S)$ the set of equivalence classes of \emph{positively framed marked properly convex projective structures with maximal or minimal hyperbolic ends} on $S$. We call $\pi^+$ the natural projection of $\Ttp(S)$ onto $\Ttm(S)$ defined by forgetting the framing.

In terms of notation, we denote the structure $(\Omega, \Gamma, f)$ with a fixed positive framing by $(\Omega, \Gamma, f)^+.$

A \emph{negative framing} differs from a positive framing only at hyperbolic ends. In the former, the choice of the pair consisting of a non-saddle point and the line through it and the saddle point was only allowed for minimal ends. In a \emph{negative framing}, this choice is only allowed for maximal ends instead, while keeping everything else the same. This means that for each maximal end there are four possible choices, while for each minimal end there are now two. We denote the structure $(\Omega, \Gamma, f)$ with a fixed negative framing by $(\Omega, \Gamma, f)^-.$
The set $\Ttpd(S)$ is the set of equivalence classes of \emph{negatively framed marked properly convex projective structures with maximal or minimal hyperbolic ends} on $S$, and $\pi^- : \Ttpd(S) \rightarrow \Ttm(S)$ is the associated projection map. 

Projective duality induces natural \emph{duality maps} $\sigma^+ \co \Ttp(S) \to \Ttpd(S)$ and $\sigma^- \co \Ttpd(S) \to \Ttp(S)$ which are mutually inverse. This will be discussed in detail in \S\ref{subsec:dual-proj-st}.

It will be shown in \S\ref{subsec:fin_vol} that finite-volume structures on $S$ are those in which every end is a cusp. In particular, finite-volume structures have a unique framing so there are injective maps $\Ttfin(S) \hookrightarrow{} \Ttp(S)$ and $\Ttfin(S) \hookrightarrow{} \Ttpd(S)$.


\subsection{On cusps and boundary components}

It was already mentioned in the introduction that Fock and Goncharov offer an interpretation of cusps and geodesic boundary components depending on the flags used to decorate the ends; the different cases can be glimpsed from Figure~\ref{peripheral_action}.
In this section, we offer a different interpretation, using \emph{algebraic horospheres}, which results in a different classification for special ends.

Let us first recall some definitions (see \cite[Section $3$ pg. 16]{Cooper-convex-2015} for details). Let $\Omega \subset \RP2$ be a properly convex domain, $p \in \partial \overline{\Omega}$ and $H \subset \RP2$ a supporting hyperplane to $\Omega$ at $p$. Define $\SL(H,p)$ to be the subgroup of $\SL(3,\RR)$ which preserves both $H$ and $p$, and $\mathcal{G} = \mathcal{G}(H,p)$ to be the subgroup of $\SL(H,p)$ of those elements $A \in \SL(H,p)$ which satisfy:
\begin{itemize}
	\item $A$ acts as the identity on $H$,
	\item $A(\eta) = \eta$ for every line $\eta$ in $\RP2$ passing through $p$.
\end{itemize}
If $\eta$ is a line containing $p$ that is not contained in $H$, then $\mathcal{G}$ acts on $\eta$ as the group $\Par(\eta,p)$ of parabolic transformations of $\eta$ fixing $p$, thus there is an isomorphism $(\RR,+) \cong \Par(\eta,p) \cong \mathcal{G}$.

Let $\mathcal{S}_0 \subset \partial \overline{ \Omega}$ be the subset of $\partial \overline{ \Omega}$ obtained by deleting $p$ and all line segments in $\partial \overline{ \Omega}$ with one endpoint at $p$. A \emph{generalized horosphere centred on $(H,p)$} is the image $\mathcal{S}_A := A(\mathcal{S}_0)$ of $\mathcal{S}_0$ under $A \in \mathcal{G}(H,p)$. An \emph{algebraic horosphere}, or simply \emph{horosphere}, is a generalized horosphere contained in $\Omega$.

If $B \in \SL(H,p)$ preserves $\Omega$, it is well-known that $B$ acts on horospheres by
$$
B(\mathcal{S}_A) = B A (\mathcal{S}_0) = B A B^{-1}(B \mathcal{S}_0) = B A B^{-1} (\mathcal{S}_0) = \mathcal{S}_{BAB^{-1}}.
$$
Such action can be expressed in terms of the \emph{exponential horosphere displacement function} $\tau$, defined as follows. Given $B \in \SL(H,p)$, let $\lambda_+(B)$ be the eigenvalue for the eigenvector $p$. If $v \in \RP2 \setminus H$, then $Bv + H = \lambda_-v + H$ and $\lambda_- = \lambda_-(B)$ is another eigenvalue of $B$. This does not depend on the choice of $v$. The \emph{exponential horosphere displacement function} with respect to $(H,p)$ is the homomorphism $\tau : \SL(H,p) \rightarrow (\RR^*, \times)$ given by
$$
\tau(B) = \frac{\lambda_+(B)}{\lambda_-(B)}.
$$
Choose an isomorphism from $(\RR,+)$ to $\Par(\eta,p)$ given by $ t \mapsto A_t$ and denote
$$
\mathcal{S}_t = A_t (\mathcal{S}).
$$

\begin{thm}[\cite{Cooper-convex-2015}] \label{thm_CLT_horospheres}
	If $B \in \SL(H,p)$ preserves $\Omega$, then $B(\mathcal{S}_t) = \mathcal{S}_{\tau(B)t}$.
\end{thm}

Let $x  \in \Ttp(S)$, and let $\gamma$ be a generator of the peripheral subgroup $\pi_1(E)$. Let $\Omega  := \dev (\widetilde{S})$. If $\hol (\gamma)$ is totally hyperbolic, it fixes three pairwise distinct points $\{V_{\lambda_1},V_{\lambda_2},V_{\lambda_3} \}$ of $\RP2$ and preserves the lines $\{H_{12},H_{13},H_{23}\}$ through them (Figure \ref{peripheral_action} on the left). Moreover, $\Omega$ may or may not contain the shaded triangle. Horospheres centred at $(H_{ij},V_{\lambda_i})$ look differently depending on the triple $(\Omega,V_{\lambda_i}, H_{ij})$, (the different horospheres are depicted in Figure \ref{horospheres_totally_hyperbolic}). In all cases $\tau(B) \not= 1$, therefore $\hol (\gamma)$ does not preserve the horospheres.

\begin{figure}[h]
	\centering
	\includegraphics[width=0.9\textwidth]{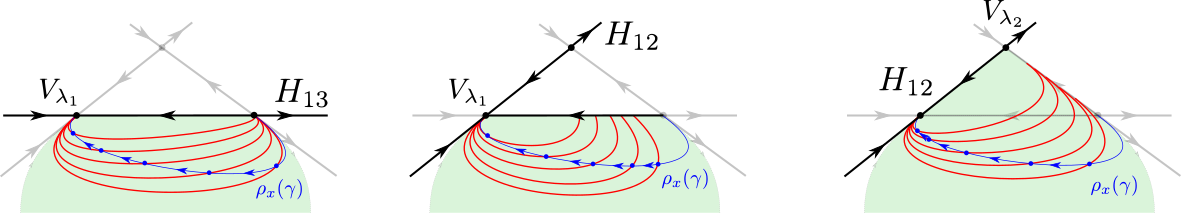}
	\caption{Horospheres are not preserved by a totally hyperbolic peripheral element}
	\label{horospheres_totally_hyperbolic}
\end{figure}	

When $\hol(\gamma)$ is parabolic, it fixes a unique point $V_{\lambda_1} \in \partial \overline{\Omega}$ and preserves a unique supporting hyperplane $H$ through $V_{\lambda_1}$. Since $\tau(B) = 1$, Theorem \ref{thm_CLT_horospheres} implies that $\hol(\gamma)$ preserves horospheres at $(H,p)$, as depicted in Figure \ref{horospheres_parabolic}. Hence horospheres are $\hol(\gamma)$--orbits in $\Omega$.

\begin{figure}[h!]
	\centering
	\includegraphics[width=4cm]{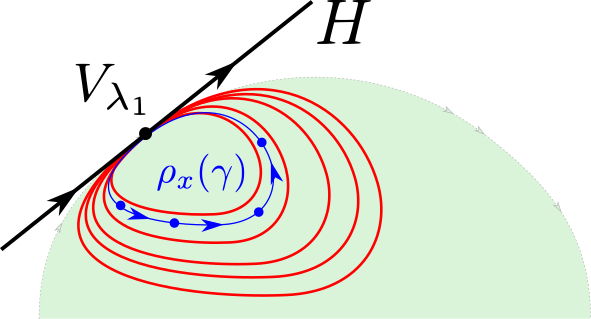}
	\caption{Horospheres are preserved at a cusp}
	\label{horospheres_parabolic}
\end{figure}

When $\hol(\gamma)$ is quasi-hyperbolic, we a have more subtle situation, in between the two previous ones. In this case $\hol(\gamma)$ fixes two points $\{V_{\lambda_1},V_{\lambda_2} \}$ on $\partial \overline{\Omega}$, preserves the line $H_{12}$ through them and acts as a unipotent transformation on a second line $H_1$ containing only $V_{\lambda_1}$. On one hand, the exponential horosphere displacements with respect to $(H_{12},V_{\lambda_2})$ and $(H_1,V_{\lambda_1})$ are non-trivial, behaving like a hyperbolic end. Whereas for $(H_{12},V_{\lambda_1})$, $\tau(\hol(\gamma)) = 1$ and algebraic horospheres are preserved, analogously to the parabolic case. We demonstrate this situation in Figure \ref{horospheres_quasi_hyperbolic}.

\begin{figure}[h!]
	\centering
	\includegraphics[width=0.9\textwidth]{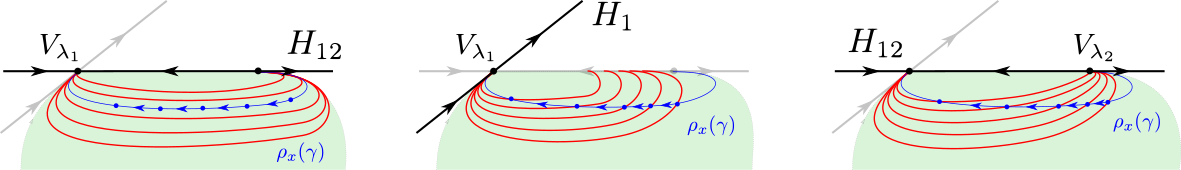}
	\caption{Special boundary components preserve some horospheres and permute others}
	\label{horospheres_quasi_hyperbolic}
\end{figure}


\subsection{Ideal triangulations of surfaces} \label{sec:Ideal_triangulations _of_surfaces}

Here we present some well-known results which we will use extensively. By construction, $S_{g,n} = S_g \setminus \mathcal{P},$ where $\mathcal{P} = \{p_1, \ldots, p_n\}.$ An \emph{essential arc} in $S_{g,n}$ is the intersection with $S_{g,n}$ of a simple arc embedded in $S_{g}$ that has endpoints in $\mathcal{P},$ interior disjoint from $\mathcal{P}$ and is not homotopic (relative to $\mathcal{P}$) to a point in $S_{g}.$ An \emph{ideal triangulation} $\tri$ of $S_{g,n}$ is a union of pairwise disjoint and non-homotopic essential arcs. The components of $S_{g,n}\setminus \tri$ are \emph{ideal triangles}, and we regard two ideal triangulations of $S_{g,n}$ as equivalent if they are isotopic via an isotopy of $S_{g}$ that fixes $\mathcal{P}$.

\begin{lem} \label{ideal_triangulation}
The surface $S_{g,n}$ admits an ideal triangulation. Moreover, every ideal triangulation has $-2 \chi(S_{g,n})$ ideal triangles.
\end{lem}

An \emph{edge flip} on an ideal triangulation consists of picking two distinct ideal triangles sharing an edge, removing that edge and replacing it with the other diagonal of the square thus formed.

For instance, any ideal triangulation of $S_{1,1}$ comprises three essential arcs and divides the surface into two ideal triangles. All of these ideal triangulations are combinatorially equivalent. However, performing an edge flip results in a non-isotopic ideal triangulation. The space of isotopy classes of ideal triangulations of the once-punctured torus naturally inherits the structure of the infinite trivalent tree, where vertices correspond to isotopy classes of ideal triangulations, and there is an edge between two such classes if and only if they are related by an edge flip. A well-known geometric realisation of this was described by Floyd and Hatcher \cite{Floyd-incompressible-1982}. In general, we have:

\begin{lem}[Hatcher~\cite{Hatcher-triangulations-1991}, Mosher~\cite{Mosher-tiling-1988}]\label{lem:finite_sequence_of_elementary_moves}
Any two ideal triangulations of $S$ are related by a finite sequence of edge flips. 
\end{lem}

Now suppose $S$ has a strictly convex real projective structure of finite volume $(\Omega, \Gamma, f)$.  An ideal triangulation of $S$ is \emph{straight} if each ideal edge is the image of the intersection of $\Omega$ with a projective line. We remark that the word \emph{straight} has been chosen instead of \emph{geodesic} so that the terminology can be transferred to the more general case of properly convex projective structures, where geodesics (with respect to the Hilbert metric) generally do not develop into straight lines in the projective plane.

\begin{lem} \label{lem:geodesic_ideal_triangulation}
Every ideal triangulation of $S$ is isotopic to a straight ideal triangulation.
\end{lem}

\begin{proof}
The ideal triangulation $\tri$ of $S$ can be lifted to a $\Gamma$--equivariant topological ideal triangulation $\widetilde{\tri}$ of $\Omega.$ There is a homeomorphism of $\Omega$ that fixes the frontier $\partial \overline{\Omega}$ and takes each topological ideal edge to a segment of a projective line. Since $\overline{\Omega}$ is a closed disc, this homeomorphism is isotopic to the identity. Whence the topological ideal triangulation $\widetilde{\tri}$ of $\Omega$ is isotopic to a straight ideal triangulation. Since the set of ideal endpoints of edges is $\Gamma$--equivariant and any two such endpoints determine a unique segment of a projective line in $\Omega,$ the straight ideal triangulation is $\Gamma$--equivariant. We may therefore choose a $\Gamma$--equivariant isotopy between the ideal triangulations of $\Omega$ and push this down to an isotopy of $S.$
\end{proof}


\subsection{The canonical bijection for positively framed structures}

For a fixed ideal triangulation $\tri$ of $S$, let $\triangle$ be the set of all ideal triangles and $\underline{E}$ be the set of all oriented edges. Given a point in $\Ttp(S),$ to each ideal vertex of the lift of $\tri$ to the universal cover of $S$ there is both an associated point in the projective plane and a line though that point, a flag. Whence to each ideal triangle there is an associated triple of flags, and hence (after choosing a cyclic order) a triple ratio. Similarly, to each oriented edge, there is an associated quadruple ratio of flags. This gives rise to a map $\Ttp(S) \rightarrow \RR^{\triangle \cup \underline{E}}_{>0}.$ The proof of the below theorem will make the association precise and turn it into a well-defined bijective map. The \emph{Fock--Goncharov moduli space} is then the set of all functions $\RR^{\triangle \cup \underline{E}}_{>0} := \{\triangle \cup \underline{E} \rightarrow \RR_{>0} \}.$ We note that $|\triangle \cup \underline{E}| = - 8 \chi(S).$

\begin{thm}[Fock-Goncharov 2007, \cite{Fock-moduli-2007}] \label{thm:global-coord}
Let $S = S_{g,n}$ be a surface of negative Euler characteristic with at least one puncture. For each ideal triangulation $\tri$ of $S$ and each orientation $\nu$ on $S$ there is a canonical bijection
	\[
	\phi^+_{\tri, \nu}: \Ttp(S) \rightarrow \RR^{\triangle \cup \underline{E}}_{>0}.
	\]
\end{thm}

By fixing an orientation of $S$, we may suppress $\nu$ from the notation and simply write $\phi^+_{\tri, \nu} = \phi^+_{\tri}.$ In the introduction we suppressed both the triangulation and the orientation from the notation and used the inverse $\psi^+ = (\phi^+_{\tri, \nu})^{-1}.$

We will call $\phi^+_{\tri}(x)$ the \emph{Fock--Goncharov coordinate} of $x \in \Ttp(S)$.

\begin{proof}
Assume $S$ is oriented using $\nu$. We first define $\phi^+_{\tri, \nu} = \phi^+_{\tri}$.

Let $\widetilde{\tri}$ be the lift of $\tri$ to the universal cover $(\widetilde{S}, \pi)$. Let $x = (\dev_x, \hol_x) \in \Ttp(S)$. The idea is to choose a developing map in its isotopy class that is adapted to the framing and has straight ideal triangles, in order to define triple and quadruple ratios.

Fix an affine patch $\mathcal{A}$ in which $\Omega$ is strictly contained. We determine a developing map in its isotopy class by mapping ideal vertices of $\widetilde{\tri}$ to the points of the framing. This choice is allowed, as the ideal vertices are fixed points of subgroups which are conjugate to the peripheral subgroups, contained in the frontier of $\Omega$. This also determines the images of the edges in $\widetilde{\tri}$ as straight line segments between the vertices.
 
 Recall that each point in the framing is paired up with a supporting line. Hence to each ideal triangle we have an associated triple of flags. More specifically, let $t_i \in \triangle$ be an ideal triangle. Any lift $\widetilde{t_i}$ of $t_i$ to $\widetilde{S}$ is assigned a triple of flags $\mathfrak{F}_i := ((\F_{i,0}, \F_{i,1}, \F_{i,2}))$, cyclically ordered according to the orientation induced on $t_i$ by $\nu$. Since $\Omega$ is properly convex, $\mathfrak{F}_i$ represents an element of $\P_3$. Hence define 
\[
\phi^+_{\tri}(x)(t_i) := \cancel{3}(\mathfrak{F}_i) \in \RR_{>0}.
\]
There is a choice of $\widetilde{t_i}$ in this definition. Any two such choices differ by a deck transformation. Since $\dev_x$ is $\hol_x$--equivariant, any two resulting triples of flags $\mathfrak{F}_i$ and $\mathfrak{F}_i'$ differ by a projective transformation and therefore represent the same element of $\P_3$. Lemma \ref{lem:triple_ratio_proj_invariant} ensures that $\cancel{3}(\mathfrak{F}_i) = \cancel{3}(\mathfrak{F}_i')$.

Let $e_i \in \underline{E}$ be an oriented edge. Choose a lift $\widetilde{e_i}$ of $e_i$. $\widetilde{e_i}$ is shared by two triangles in $\widetilde{\tri}$, each of which is assigned a triple of flags. Thereby we can consider $\widetilde{e_i}$ to be uniquely associated to an ordered quadruple of flags $\mathfrak{F}_i := (\F_{i,0}, \F_{i,1}, \F_{i,2}, \F_{i,3})$, ordered according to the orientation of $\Omega$, where $\F_{i,0}$ and $\F_{i,2}$ represent the vertices at the tail and head of $\dev_x(\widetilde{e_i})$ respectively. As $\Omega$ is properly convex, $\mathfrak{F}_i$ is identified with an element of $\P^*_4$. Thereby we define
\[
\phi^+_{\tri}(x)(e_i) := \cancel{4}(\mathfrak{F}_i) \in \RR_{>0}.
\]
As above, the choice of $\widetilde{e_i}$ only changes $\mathfrak{F}_i$ by a projective transformation and does not change $\cancel{4}(\mathfrak{F}_i)$. The positivity of $\phi^+_{\tri}(x)(e_i)$ follows from Theorem \ref{thm:param_P^*_4}.

The above construction gives rise to a well-defined map
\[
\phi^+_{\tri} : \Ttp(S) \rightarrow \RR^{\triangle \cup \underline{E} }_{>0}.
\]
It remains to show that this map is a bijection.

Let $x,y \in \Ttp(S)$ and $\phi^+_{\tri}(x) = \phi^+_{\tri}(y)$. For any triangle $t \in \triangle$,
\[
\phi^+_{\tri}(x)(t_0) = \phi^+_{\tri}(y)(t_0).
\]
Fix an arbitrary lift $\widetilde{t}$ of $t$ to $\widetilde{S}$. Then $\widetilde{t}$ is assigned triples of flags $\mathfrak{F}^{\widetilde{t}}_x$ and $\mathfrak{F}^{\widetilde{t}}_y$ by the framings of $x$ and $y$ respectively. By Theorem \ref{thm:param_P_3}, there exists $A \in \SL(3, \RR)$ such that
\[
A \cdot \mathfrak{F}^{\widetilde{t}}_x = \mathfrak{F}^{\widetilde{t}}_y.
\]
Following Lemma \ref{lem:geodesic_ideal_triangulation}, we assume that edges of ideal triangles are segments of projective lines, which are thus uniquely determined by $\mathfrak{F}^{\widetilde{t}}_x$ and $ \mathfrak{F}^{\widetilde{t}}_y$. Therefore $A \cdot \dev_x(\widetilde{t}) = \dev_y(\widetilde{t})$. In particular, the restrictions of $A \cdot \dev_x$ and $\dev_y$ to $\widetilde{t}$ are the same up to an isotopy of their image that fixes the vertices..

We claim that $A \cdot \Omega_x = \Omega_y$ and furthermore that $A \cdot \dev_x$ is isotopic to $\dev_y$ over the whole domain. Let $t'$ be a triangle adjacent to $t$ and $\widetilde{t}'$ the lift of $t'$ which is adjacent to $\widetilde{t}$. Let $\widetilde{t}$ and $\widetilde{t}'$ share the edge $e$ and let $v'$ be the ideal vertex of $t'$ not contained in $e$.  By proper convexity of $\Omega_x$ and $\Omega_y$, the flags at the vertices of $\widetilde{t}$ and $\widetilde{t}'$ form two pairs of strictly inscribed convex quadrilaterals $\chi_x$ and $\chi_y$. Cyclically order them according to $\nu$, and endow them with the same marking by choosing an endpoint of $e$. Thus they are elements of $\P^*_4$ having same triple ratios and quadruple ratios. By Theorem \ref{thm:param_P^*_4}, $\chi_x$ and $\chi_y$ are projectively equivalent. Since $A \cdot \dev_x(\widetilde{t}) = \dev_y(\widetilde{t})$, it is also the case that $A \cdot \chi_x = \chi_y$. In particular  the projective lines $\eta'_x$ and $\eta'_y$ at $\dev_x(v')$ and $\dev_y(v')$ respectively, are projectively equivalent via $A$. Lemma \ref{lem:geodesic_ideal_triangulation} ensures once again that $ A \cdot \dev_x(\widetilde{t}') = \dev_y(\widetilde{t}')$, and $A \cdot \dev_x$ and $\dev_y$ are isotopic on $\widetilde{t} \cup \widetilde{t}'$.

Continuing in this manner we see that the image of $\widetilde{\tri}$ is uniquely determined by $\phi^+_{\tri}(x) = \phi^+_{\tri}(y)$. Moreover, since both developing maps are equivariant with respect to the action of their corresponding holonomy groups, we have
\[
\hol_x (\gamma) = A^{-1} \cdot \hol_y (\gamma) \cdot A, \quad \forall \; \gamma \in \pi_1(S).
\]
Therefore $A \cdot \dev_x$ and $\dev_y$ differ by a holonomy-equivariant isotopy, which may be pushed down to $S$. Since the framing is determined by the flags at vertices of the ideal triangles, we can conclude that $x$ and $y$ define equivalent framed convex projective structures on $S.$ This completes the proof that $\phi^+_{\tri}$ is injective.

It remains to show that $\phi^+_{\tri}$ is surjective. Suppose each triangle and each oriented edge in $\tri$ were assigned a positive real number. Lift these assignments to $\widetilde{\tri}$. We will explicitly construct a domain $\Omega^+ \subset \RP2$ and holonomy representation $\hol : \pi_1(S) \rightarrow \SL(3,\RR)$ such that $ \Omega^+ / \hol( \pi_1(S) ) $ is a convex projective surface homeomorphic to $S$. Finally we will show that the original assignment of numbers to $\tri$ determines a canonical choice of framing, so that the convex projective structure on $S$ thus defined is extended to an element of $\Ttp(S)$.

Let $t \in \triangle$. Fix a lift $\widetilde{t}$ of $t$ to $\widetilde{S}$ with orientation inherited from $S$. By Theorem \ref{thm:param_P_3}, the positive real number assigned to $\widetilde{t}$ uniquely determines an element of $\P_3$. Hence fix a cyclically ordered representative $\mathfrak{F}^{\widetilde{t}}$. This triple of flags defines three points of $ \partial \overline{ \Omega^+}$ and a triple of supporting lines to $\Omega^+$ at those points. The rest of the construction of $\Omega^+$ proceeds as in the proof of Theorem \ref{thm:param_P^*_4} where, given a triple of flags, a fourth flag of an adjacent triangle is uniquely determined by two edge ratios and a triple ratio. Iterating this procedure, $\Omega^+$ is defined to be the union of the inscribed triangles determined by all these elements of $\P_3$, and the flags will be the framing for the projective structure.

We describe $\Omega^+$ in more detail. Since the set of triangles of $\widetilde{\tri}$ is countable, we can index them as $\{ \widetilde{T}_i \}_{i \in \NN}$ so that $\cup_{i=0}^m \widetilde{T}_i$ is connected for all $m \in \NN$, and $\cup_{i=0}^\infty \widetilde{T}_i = \widetilde{\tri}$. In the above construction we assigned to each triangle $\widetilde{T}_i \in \widetilde{\tri}$ an element $\chi_i$ of $\P_3$. Let $\mathcal{T}_i$ be the inscribed triangle of $\chi_i$ with the vertices removed, thus define
\[
\Omega^+_m := \bigcup_{i = 0}^m \mathcal{T}_i\qquad \mbox{ and } \qquad \Omega^+ := \bigcup_{i = 0}^\infty \mathcal{T}_i.
\]
We are going to show that $\Omega^+$ is properly convex. Since $\Omega^+_0$ is clearly properly convex, we will proceed by induction on $m$. Suppose $\Omega^+_{m-1}$ is properly convex and fix an affine patch $\A$ strictly containing $\Omega^+_{m-1}$. Let $e$ be the common edge between $\Omega^+_{m-1}$ and $\mathcal{T}_m$ and let $V$ be the vertex of $\mathcal{T}_m$ disjoint from $e$. The edge $e$ is equipped with two edge ratios and a flag $(V_i, \eta_i)$ for each endpoint. We remark that $\eta_i$ is a supporting line for the domain $\Omega^+_{m-1}$. If necessary, we rearrange the patch $\A$ to contain the triangle $\Delta$ defined by $\eta_0,\eta_1$ and $e$, disjoint from $\Omega^+_m$. Since the edge ratios are strictly positive, Theorem \ref{thm:param_P^*_4} implies that $V$ lies strictly within the triangle $\Delta$. It follows that $\Omega^+_m$ is also convex and properly contained in the affine patch $\A$. To conclude the induction argument it is enough to observe that the positivity of the triple ratio associated to $\mathcal{T}_m$ implies that $\eta_V$, the line attached to $V$, is also a supporting line for $\Omega^+_m$. Hence $\Omega^+$ is properly convex.

The set $\Omega^+$ is defined as a union of ideal triangles, thus it has a natural combinatorial structure. By construction, we have a combinatorial isomorphism from the universal cover $\widetilde{S}$ to $\Omega^+$, and we now define an action of $\pi_1(S)$ by projective maps on $\Omega^+$ that has this combinatorial isomorphism as a developing map.

We now define the holonomy representation $\hol: \pi_1(S) \rightarrow \Gamma < \SL(3, \RR)$. Fix an arbitrary lift $\widetilde{t}$ of some $t \in \triangle$. Let $\mathfrak{F}^{\widetilde{t}}$ be the cyclically ordered triple of flags associated to $\widetilde{t}$. For each $\gamma \in \pi_1(S)$, $\widetilde{t}' := \gamma (\widetilde{t})$ is another lift of $t$, which is assigned another triple of flags $\mathfrak{F}^{\widetilde{t}'}$ with the same triple ratio as $\mathfrak{F}^{\widetilde{t}}$. Once again, Theorem \ref{thm:param_P_3} applies to provide a unique projective transformation $T_{\gamma} \in \SL(3, \RR)$ such that
\[
T_{\gamma} \cdot \mathfrak{F}^{\widetilde{t}} = \mathfrak{F}^{\widetilde{t}'}.
\]
Now $T_{\gamma}$ is an element of $\SL(\Omega^+)$, namely it preserves $\Omega^+$. Indeed, $\Omega^+$ is defined exclusively in terms of triangle parameters and edge parameters, which are invariant under projective transformations. More precisely, given two triangles $\widetilde{t},\widetilde{t}' \subset \widetilde{\tri}$,
\[ 
\widetilde{t} = \gamma (\widetilde{t}') \iff \dev(\widetilde{t}) = T_{\gamma} \cdot \dev(\widetilde{t}').
\]
Hence we have
\[
\hol(\gamma) := T_{\gamma}, \quad \forall \; \gamma \in \pi_1(S).
\]
The developing map is equivariant with respect to $\hol \colon \gamma \mapsto T_\gamma$, which is an isomorphism onto its image $\Gamma := \hol(\pi_1(S))$. In particular, $\dev$ induces an isomorphism $f$ of CW--complexes from $S$ to $\Omega^+ / \Gamma$. The triple $(\Omega^+, \Gamma, f)$ thereby defines an element of $\Tt(S)$.

The next step is to show that $(\Omega^+, \Gamma, f) \in \Ttm(S)$ and that it has a natural framing given by the flags at each ideal vertex. Before we analyse maximal and minimal ends, we characterise flags of $\Omega^+$ as fixed elements of conjugates of peripheral subgroups of $\Gamma$.

Recall that every vertex $V$ of $\Omega^+$ has a line $\eta$ assigned to it, forming the flag $(V,\eta)$. As $\dev$ is a combinatorial equivalence, $V$ corresponds to an ideal vertex $\widetilde{E}$ of $\widetilde{\tri}$, which is a lift of an end $E$ of $S$. Let $\gamma_0$ be a generator of $\pi_1(E)$, and $\gamma$ the element conjugate to $\gamma_0$ fixing $\widetilde{E}$. Then $V$ is fixed by $\hol(\gamma)$. Furthermore, if $\widetilde{t}$ is a triangle in $\Omega^+$ with vertex $V$, then $\hol(\gamma)(\widetilde{t})$ is, by construction, another triangle of $\Omega^+$ sharing the same flag $(V,\eta)$. Therefore $\eta$ is also preserved by $\hol(\gamma)$. Summarising, every flag of $\Omega^+$ is a pair consisting of a fixed point and an invariant line of some (conjugate of a) peripheral subgroup of $\Gamma$. Thus when $E$ is a special end or a cusp, the flag $(V,\eta)$ is an admissible choice of framing at $E$. 

Now suppose that $E$ is hyperbolic, namely $\hol(\gamma)$ fixes three distinct points of $\RP2$, one attracting $V_+$, one repelling $V_-$, and a saddle $V_0$. By the previous discussion there are six possibilities for $(V,\eta)$, given that $V \in \{V_+,V_-,V_0\}$, $\eta \in \{ V_+V_-, V_+V_0, V_-V_0\}$ and $V \in \eta$. For convenience, we fix an affine patch $\A$ strictly containing $\Omega^+$. Recall from \S \ref{subsec:classification_of_ends} that the interior of the segment between $V_+$ and $V_-$ in $\A$, say $s$, is either contained in $\Omega^+$ or in the frontier of $\Omega^+$, and $\Omega^+$ can not be equal to the interior of a triangle bounded by $s$ and $V_0$.

If $V = V_0$, then both $V_0$ and $s$ are subsets of $\overline{ \Omega^+}$. By proper convexity, $\Omega^+$ must contain the interior of the triangle spanned by $V_+,V_-,V_0$, and therefore $E$ must be maximal.

Suppose $V \not= V_0$. Let $t$ be an ideal triangle of $\Omega^+$ with vertex $V$. If $s$ was contained in $\Omega^+$, then by the construction $s$ would be covered by some triangles of $\Omega^+$. However, the orbit of $t$ under $\langle \hol(\gamma) \rangle$ accumulates at $s$, so $s$ can not intersect the interior of any triangle. At the same time, $s$ can not be the edge of a triangle, as $\hol(\gamma)$ acts on $s$ by translation. It follows that $s$ must be in the frontier of $\Omega^+$, and $E$ is minimal.

In conclusion, every hyperbolic end $E$ is either maximal or minimal, and $(V,\eta)$ is an admissible choice of positive framing also at $E$. Framing each end in this way gives $(\Omega^+, \Gamma, f)^+ \in \Ttp(S)$.
\end{proof}


\subsection{A natural action of the symmetric group} \label{subsubsec:sym_gp}

The \emph{symmetric group} on three letters, $\Sym(3)$, acts by permuting the set of framings of a given end. This induces a $\Sym(3)^n$--action on $\Ttp(S)$. 

\begin{thm} \label{sym_group_action}
	Let $S = S_{g,n}$ be a surface of negative Euler characteristic with $n>0$ punctures. There is a faithful but not free action $\Sym(3)^n \curvearrowright \Ttp(S)$, whose orbits are framed convex projective structures having conjugate holonomy representations. In particular, the map $\mu^+\co\Ttp(S) \to \X^\times(S)$ is an isomorphism and the quotient space $\Ttp(S) / \Sym(3)^n$ is naturally identified with $\X(S)$.
\end{thm}

\begin{proof}
	First we construct a $\Sym(3)$--action on the moduli space of a once punctured surface $S=S_{g,1}$. Denote the end of $S$ by $E$ and let $\gamma \in \pi_1(S)$ be a generator of the peripheral subgroup. Every point $x = (\Omega^x,\Gamma^x,f^x)^+ \in \Ttp(S)$ comes with a holonomy representation $\hol_x$ and a $\Gamma^x$--orbit of flags constituting the framing of $E$. Let $(V, \eta)$ be the unique element in that orbit such that $V$ is fixed by $\hol_x(\gamma)$. As a projective transformation, $\hol_x(\gamma)$ fixes three points (counted with multiplicity) $V_1,V_2$ and $V_3$ in $\RP2$. We order them so that:
	\begin{enumerate}
		\item $V_1 = V$;
		\item $V_2 \in \eta$,
		\item if $E$ is special, $V_1 = V_2$ if and only if $V_3 \notin \eta$.
	\end{enumerate}
	Thereby, $x$ is assigned the $\Gamma^x$--orbit of an ordered triple $(V_1, V_2, V_3)$. For $\alpha \in \Sym(3)$, we define $y := \alpha \cdot x$ to be that structure whose assigned $\Gamma^x$--orbit of ordered triple is $(V_{\alpha(1)}, V_{\alpha(2)}, V_{\alpha(3)} )$. If $E$ is not hyperbolic, $y$ has equivalent Hilbert geometry to $x= (\Omega^x,\Gamma^x,f^x)$, but it is framed according to the conditions $1) - 3)$ above, applied to the triple $(V_{\alpha(1)}, V_{\alpha(2)}, V_{\alpha(3)} )$. If $E$ is hyperbolic, one may need to retract or expand $\Omega^x$ to match the correct framing. In this case, the flag $(V_{\alpha(1)}, V_{\alpha(1)}V_{\alpha(2)} )$ uniquely determines whether $E$ is maximal or minimal, hence define $\Omega^y$ accordingly. Then $\Omega^x / \Gamma^x$ differs from $\Omega^y / \Gamma^x$ by at most an annulus, so there is a natural homeomorphism $g : \Omega^x / \Gamma^x \rightarrow \Omega^y / \Gamma^x$ which is the identity outside a sufficiently small regular neighbourhood of $E$. We define $f^y := f^x \circ g$ and $y$ is the structure $(\Omega^y,\Gamma^x,f^y)$, framed as per $(V_{\alpha(1)}, V_{\alpha(2)}, V_{\alpha(3)} )$. We remark that $x$ and $y$ have the same holonomy representations.
	
	The above construction extends to the case $n \geq 1$, by acting on a different boundary component with each distinct copy of $\Sym(3)$. The stabilizer of a structure with all ends hyperbolic is trivial, therefore the action is faithful. 	
	Structures with only cuspidal ends are global fixed points, hence the action is not free. The existence of both types of structures follows by a standard construction in hyperbolic geometry.
	
	It remains to show that the $\Sym(3)^n$--orbit of a point $x$ is precisely the set of those structures in $\Ttp(S)$ with holonomy conjugate to $x$. One inclusion is trivial as the $\Sym(3)$--action preserves the holonomy. For the other, suppose $x = (\Omega^x,\Gamma^x,f^x)$ and $y = (\Omega^y,\Gamma^y,f^y)$ have conjugate holonomies. By acting with the appropriate element of $\SL(3,\RR)$, we can fix a representative of $y$ which has the same holonomy as $x$. It follows that $\Gamma^x = \Gamma^y$ and they share the same peripheral subgroups, hence have the same fixed points in $\RP2$. In particular, for each end $E_k$ of $S$ there is an element $\alpha_k \in \Sym(3)$ mapping the ordered triple associated to the framing of $x$ at $E_k$ to the ordered triple associated to the framing of $y$ at $E_k$. This is enough to prove that $\Omega^x$ and $\Omega^y$ may only differ at the hyperbolic ends, which can be maximal or minimal. 

In particular, analogous to the construction of the developing map in the proof of Theorem~\ref{thm:global-coord}, the straightening procedure of ideal triangulation can now be applied to convex cores of the structures $x$ and $y$ respectively, to give a homeomorphism $g : \Omega^x / \Gamma^x \rightarrow \Omega^y / \Gamma^y$ which is the identity outside a small neighbourhood of $\cup_{k} E_k$, such that $f^y := f^x \circ g$. This concludes the proof that $y$ is indeed $\alpha \cdot x$, for some $\alpha = (\alpha_1,\dots,\alpha_n) \in \Sym(3)^n$.

Associating flags to the holonomies gives the claim regarding the map to $\X^\times(S).$
\end{proof}

Global fixed points of the action are those structures whose ends are all cusps. In \S \ref{subsec:fin_vol} we show that this set is in fact $\Ttfin(S)$. In \S \ref{subsubsec:max-min-coordinates} we give a parametrisation of a fundamental domain for the action of $\Sym(3)^n$.


\subsection{Duality} \label{subsec:dual-proj-st}

Every convex projective structure has a dual structure, constructed using the self-duality of the projective plane. This induces an isomorphism $\sigma^+ \co \Ttp(S) \rightarrow \Ttpd(S)$. In this section we recall the main points of this construction, and refer the reader to \cite{Goldman-notes-1988, Vinberg-theory-1963, Cooper-marked-2010} for details and proofs. 

Let $\Omega$ be a properly convex domain of $\RP2$. Its \emph{dual domain} is defined to be the set
$$
\Omega^* := \{ P \in \RP2 \ | \ P^\perp \cap \overline{\Omega} = \emptyset \}.
$$
Then $\Omega^*$ is also a properly convex domain and $\Omega = (\Omega^*)^*$. Supporting lines to $\Omega$ correspond to points on the frontier of $\Omega^*$, and vice-versa. Furthermore, there is a diffeomorphism $\Phi: \Omega \rightarrow \Omega^*$, called the \emph{dual map}. The map $\Phi$ has an elementary interpretation as follows. For $P \in \Omega$, $P^\perp$ is a line disjoint from $\overline{ \Omega^*}$, thus $\A_P = \RP2 \setminus P^\perp$ is an affine patch strictly containing $\Omega^*$. Then $\Phi(P)$ is the centre of mass of $\Omega^*$ in $\A_P$.

In general, $\Phi$ is not a projective transformation, unless $\Omega$ is a conic. However, for $A \in \PGL(\Omega)$ and $P \in \Omega$, $\Phi(A(P)) = (A^t)^{-1} (\Phi(P)).$ Hence we define the \emph{dual group} of a subgroup $\Gamma \leq \PGL(3,\RR)$ to be the group
$$
\Gamma^* := \{ (A^t)^{-1} \ | \ A \in \Gamma \}.
$$
Clearly $\Gamma$ is isomorphic to  $\Gamma^*$. The upshot of this summary of duality is the following result.
\begin{thm} \label{thm:involution}
	Let $(\Omega,\Gamma,f) \in \Tt(S)$. Then $\Phi$ induces a diffeomorphism $\overline{\Phi} : \Omega / \Gamma \rightarrow \Omega^* / \Gamma^*$. In particular, $(\Omega^*,\Gamma^*,f^*)$ is an element of $\Tt(S)$, where $f^* := \overline{ \Phi} \circ f$.
\end{thm}
Theorem \ref{thm:involution} defines an involution $\sigma \co \Tt(S) \rightarrow \Tt(S)$, taking a structure to its dual structure. Duality inverts the eigenvalues of the holonomy, therefore $\overline{ \Phi}$ preserves cusps and maps hyperbolic (resp. quasi-hyperbolic) ends to hyperbolic (resp. quasi-hyperbolic) ends. However it reverses inclusions,
$$
\Omega_1 \subset \Omega_2 \iff \Omega_1^* \supset \Omega_2^*.
$$
Hence a hyperbolic end $E$ is maximal for a structure $x$ if and only if it is minimal for $\sigma(x)$. It follows that $\sigma$ restricts to $\Ttm(S)$. When $x \in \Ttm(S)$ is enriched with a positive framing, its dual structure $\sigma(x)$ has a natural negative framing. Specifically, suppose $(V,\eta)$ is a flag in the $\Gamma$--orbit of the framing of an end $E$, fixed by an element $\hol(\gamma)$. Then $(\eta^\perp,V^\perp)$ is fixed by $(\hol(\gamma)^t)^{-1}$ and its $\Gamma^*$--orbit is an admissible dual framing of the same end $E$. In particular, if $V_0,V_+,V_-$ and $W_0,W_+,W_-$ are the fixed points of $\hol(\gamma)$ and $(\hol(\gamma)^t)^{-1}$, the dual relation between them is
\begin{align*}
V_+^\perp = W_-W_0, \qquad &V_-^\perp = W_+W_0, \qquad V_0^\perp = W_+W_-,\\
(V_+V_-)^\perp = W_0, \qquad &(V_-V_0)^\perp = W_+, \qquad (V_+V_0)^\perp = W_-.
\end{align*}
Therefore $\sigma$ lifts to an isomorphism $\sigma^+\co \Ttp(S) \rightarrow \Ttpd(S)$, that we call the \emph{duality map}. Moreover, the above argument gives a natural isomorphism $\sigma^-\co \Ttpd(S) \rightarrow \Ttp(S)$ which is inverse to $\sigma^+.$

It follows from this discussion, that the composition 
$\phi^+_{\tri, \nu}\circ\sigma^- \co \Ttpd(S) \to \RR^{\triangle \cup \underline{E}}_{>0}$ gives a natural bijection. However, the associated conjugacy classes of framed holonomies for the same point in 
$\RR^{\triangle \cup \underline{E}}_{>0}$ will be projectively dual to each other. This can be remedied by \emph{defining} the map $\phi^-_{\tri, \nu}$ as
\[\phi^+_{\tri, \nu} \circ (\mu^+)^{-1} \circ \sigma  \circ \mu^+ \circ \sigma^{-1},\] 
where $\sigma \co \X^\times(S) \to \X^\times(S)$ takes the character of the representation $\rho$ to the character of the inverse transpose of the representation, and each flag to the dual flag. 

Instead, we \emph{construct} the map $\phi^-_{\tri, \nu}$ in the next section in such a way that the identity 
\[ \phi^-_{\tri, \nu} = \phi^+_{\tri, \nu} \circ(\mu^+)^{-1} \circ \sigma  \circ\mu^+ \circ \sigma^{-1} \]
holds.


\subsection{The canonical bijection for negatively framed structures}

In \S\ref{subsec:classification_of_ends} we defined the negative framing of a convex projective structure and the moduli space $\Ttpd(S)$. The proof of the following result shows how to parametrise $\Ttpd(S)$ in a manner analogous to that of $\Ttp(S).$

\begin{thm} \label{thm:global-coord-dual}
	Let $S = S_{g,n}$ be a surface of negative Euler characteristic with at least one puncture. For each ideal triangulation $\tri$ of $S$ and each orientation $\nu$ on $S$ there is a canonical isomorphism
	\[
	\phi_{\tri, \nu}^*: \Ttpd(S) \rightarrow \RR^{\triangle \cup \underline{E}}_{>0}.
	\]
\end{thm}

Once again, we fix an orientation of $S$, and suppress $\nu$ from the notation to simply write $\phi_{\tri, \nu}^* = \phi_{\tri}^*.$ The map in the introduction is $\psi^- = 
(\phi_{\tri, \nu}^*)^{-1}.$

\begin{proof}
One can follow the proof of Theorem \ref{thm:global-coord} verbatim until the end of the proof of injectivity. We recall that $\phi^+_{\tri}$ was defined by mapping the lift to $\widetilde{S}$ of the ideal triangulation of $S$ into $\Omega$ with vertices at the peripheral fixed points in the frontier of $\Omega,$ and then taking triple ratios and quadruple ratios of the appropriate flags associated to edges and vertices in the triangulation. In the proof of the injectivity, it turned out that $\Omega = \Omega^+,$ the domain obtained as the image of the ideal triangulation of $\widetilde{S}.$
	
The same construction is now applied to define $\phi^-_{\tri}$ and hence produce a point in $\RR^{\triangle \cup \underline{E}}_{>0}$ associated to an element of $\Ttpd(S).$ However, in certain cases, the negatively framed domain $\Omega$ strictly contains the image $\Omega^+$ of $\widetilde{S}.$ This happens precisely when we frame a maximal hyperbolic end with one of the two non-saddle points, and the line through the saddle point. Then $\Omega\setminus \Omega^+$ consists of the orbits of maximal cusps framed in this way. One can now produce an element of $\Ttpd(S)$ by thickening the structure at these ends. Alternatively, one may appeal to duality as follows, which in particular proves the claim that the new domain is the intersection of the circumscribing triangles.
	
Referring to the setting and the notation in the proof of Theorem~\ref{thm:global-coord}, we define $\mathcal{T}_i$ to be the circumscribed triangle in $\chi_i$ (containing the inscribed triangle), with the vertices removed. Then
	\[
	\Omega^- := \bigcap_{i = 0}^\infty \mathcal{T}_i.
	\]
	In this case, $\Omega^-$ is clearly properly convex as it is the countable intersection of properly convex sets. We have the action of $\Gamma$ on $\Omega^-,$ and are required to describe a homeomorphism $f\co S \to \Omega^-/\Gamma.$ In the proof of Theorem~\ref{thm:global-coord} this followed from the combinatorial map between the universal cover of $S$ with the induced triangulation to the triangulated domain. However, here we do not have such a triangulation of $\Omega^-.$
	
By duality, $(\Omega^-)^*$ is a properly convex domain, union of the triangles $\mathcal{T}_i^*$. Following the discussion in the proof of Theorem \ref{thm:global-coord}, $(\Omega^-)^*$ is combinatorially isomorphic to $\widetilde{\tri}$, and there is the dual group $\Gamma^*$ of $\PGL((\Omega^-)^*)$ acting as a simplicial isomorphism on $(\Omega^-)^*$. In particular, $(\Omega^-)^* / \Gamma^* \cong S$  and Theorem~\ref{thm:involution} therefore gives a homeomorphism $f\co S \to (\Omega^-)^* / \Gamma^* \to \Omega / \Gamma.$ Whence $(\Omega,\Gamma,f)$ with the given framing is an element of $\Ttpd(S)$.
\end{proof}

The same Fock-Goncharov coordinate (for a fixed triangulation and orientation of $S$) gives structures in both $\Ttp(S)$ and $\Ttpd(S).$ The respective domains $\Omega^+$ and $\Omega^-$ for these structures only differ (up to projective equivalence) at those hyperbolic ends whose flags comprise one of the two non-saddle points and the line through the saddle point. In fact, in these cases, $\Omega^-$ contains the triangle spanned by the three peripheral fixed points, hence the end is maximal, while $\Omega^+$ does not.\footnote{We note that this distinction does not arise in \cite{Fock-moduli-2007}. Indeed, the proof of Theorem 2.5 in \cite{Fock-moduli-2007} claims that
$\Omega^+$ and $\Omega^-$ are always equal domains.}
However, the proofs of Theorems~\ref{thm:global-coord} and \ref{thm:global-coord-dual} show that the structures have equivalent framed holonomies, as the developing maps and framed domains are constructed from equivalent sets of inscribed and circumscribed triangles. We record this in the next result.

\begin{cor}
The following diagram commutes:
\begin{center}
\begin{tikzpicture}
  \matrix (m) [matrix of math nodes,row sep=3em,column sep=4em,minimum width=2em] {
  & \Ttp(S) &    \\
    \RR_{>0}^{\triangle \cup \underline{E}}  & & \X^\times(S)  \\
    & \Ttpd(S) &   \\};
  \path[-stealth]
    (m-2-1) edge node [above] {$(\phi^+_{\tri, \nu})^{-1} \hspace{1cm} $} (m-1-2)
    (m-2-1) edge node [below] {$(\phi_{\tri, \nu}^-)^{-1} \hspace{1cm} $} (m-3-2)
    (m-1-2) edge node [above]{$\mu^+$} (m-2-3)
    (m-3-2) edge node [below]{$\mu^-$} (m-2-3);
\end{tikzpicture}
\end{center}
\end{cor}


\section{Properties of the parameterisation}
\label{sec:Properties of the parameterisation}


\subsection{Change of coordinates} \label{subsec:change-of-coord}

Let $S$ be a surface as in Theorem \ref{thm:global-coord}. Having fixed a triangulation $\tri$ and orientation $\nu$, $\Ttp(S)$ can be canonically parametrised by positive real numbers. A different choice of $\nu$ or $\tri$ may be interpreted as a change of coordinates. Similarly, the duality map $\sigma : \Ttp(S) \rightarrow \Ttpd(S)$, defined by taking a structure to its dual, induces an involution on Fock-Goncharov moduli space. We explicitly construct these transition maps.

We remark that also the action of $\Sym(3)^n$ on $\Ttp(S)$ descends to a change of coordinates. However, even in the simplest cases, the transition functions are quite convoluted and we do not have a local description of them.

\subsubsection{Transition maps for a different orientation} \label{subsubsec:change-of-orientation}

The transition map associated to a switch in the orientation of $S$ is simple to describe. Denote by $-\nu$ the opposite orientation of $\nu$. Then for all $q \in \triangle \cup \underline{E}$,
$$
\phi_{\tri, -\nu}(x)(q) = \frac{1}{\phi_{\tri, \nu}(x)(q)}.
$$
Indeed triple ratios are computed with respect to flags with the opposite cyclical order, and edge ratios are computed after permuting the second and final arguments.


\subsubsection{Duality map in Fock-Goncharov coordinates}\label{subsubsec:change-of-coord-involution}

Fix a triangulation $\tri$ and an orientation of $S$.  Two dual structures $x \in \Ttp(S)$ and $y \in \Ttpd(S)$, are generally represented by different points of $\RR^{\triangle \cup \underline{E}}_{>0}$. That is, the composition map
$$
\phi^-_{\tri} \circ \sigma^+ \circ (\phi^+_{\tri})^{-1} = \phi^+_{\tri} \circ \sigma^- \circ (\phi^-_{\tri})^{-1} 
 \colon \RR^{\triangle \cup \underline{E}}_{>0} \rightarrow \RR^{\triangle \cup \underline{E}}_{>0}
$$
is not trivial. Recall from \S\ref{subsec:dual-proj-st} that duality maps a flag $(V,\eta)$ to its dual $(\eta^\perp,V^\perp)$. Referring to the left hand side of Figure~\ref{fig:ARS_flip}, a straightforward calculation shows that the change of coordinates is locally:
$$
e_{20} \mapsto \frac{ e_{02} t_{012} ( t_{023}+1 ) } { t_{012}+1 },
\qquad 
e_{02} \mapsto \frac{ e_{20} t_{023} ( t_{012}+1 ) } { t_{023}+1 }
\qquad 
\text{and}
\qquad 
t_{012} \mapsto \frac{1}{t_{012}}.
$$
This transformation manifestly has order two. Its set of fixed points is the algebraic variety $\mathcal{V}$ defined by the polynomials 
\[\{e_{ij} - e_{ji}, \ t_{ijk} - 1 \ | \ e_{ij} \in \underline{E}, \ t_{ijk} \in \triangle \}.\]
The duality of framings (see also the classification in \S\ref{subsubsec:max-min-coordinates}) 
implies that structures in $(\phi^+_{\tri})^{-1}(\mathcal{V}) \subset \Ttp(S)$ may only have cusps or minimal hyperbolic ends, where the line in the framing of each minimal hyperbolic end must pass through the saddle point. In particular, $(\phi^+_{\tri})^{-1}(\mathcal{V})$ strictly contains the classical Teichm\"uller space, namely the set of finite-area hyperbolic structures (see Lemma~ \ref{lem:coordinates-of-conic} in \S\ref{subsec:classical-teich}). On the other hand, one can think of the rest of the structures in $(\phi^+_{\tri})^{-1}(\mathcal{V})$ as infinite-area hyperbolic structures with a preferred framing (and dual framing).

There is a rational re-parametrisation of the edge ratios due to Parreau~\cite{Parreau-invariant-2015}, which reduces the change of coordinates $\phi^-_{\tri} \circ \sigma^+ \circ (\phi^+_{\tri})^{-1}$ into an even simpler form. If $e_{ik}$ is the coordinate of the edge oriented from $V_i$ to $V_k$, and $V_j$ is the third vertex of the triangle with ordered triple of vertices $(V_i,V_j,V_k)$, then the new edge parameter is
$$
s_{ik} := \frac{e_{ik}t_{ijk}}{1 + t_{ijk}}.
$$
The duality map on the edge ratios is then reduced to
$$
s_{ik} \mapsto s_{ki}.
$$
We will not make use of this reparametrisation of the edge ratios since this will results in a more complicated monodromy map, which is discusses in \S\ref{subsect_monodromy_operator}.


\subsubsection{Transition maps for a different triangulation} \label{subsubsec:change-of-coord}

The transition map induced by a change of triangulation is slightly more complicated. Henceforth we fix an orientation $\nu$ on $S$ and simplify the notation to $\phi_{\tri} = \phi_{\tri, \nu}$. Recall that any two ideal triangulations $\tri$ and $\tri'$ of $S$ differ by a finite sequence of edge flips (cf. Lemma \ref{lem:finite_sequence_of_elementary_moves}), where a flip along an edge $e$ is the removal of $e$ and insertion of the other diagonal into the arising quadrilateral (see Figure \ref{fig:ARS_flip}). In particular, $\tri$ and $\tri'$ have the same number of vertices, edges and triangles. Let $\triangle \cup \underline{E}$ and $\triangle' \cup \underline{E}'$ denote triangles and oriented edges of the triangulations $\tri$ and $\tri'$ respectively. Let
$$
\Phi_e \colon \RR^{\triangle \cup \underline{E}}_{>0} \rightarrow \RR^{\triangle' \cup \underline{E}'}_{>0}
$$
be the coordinate change induced by a flip along $e$. If $(e_0,\dots,e_k)$ is the sequence of edges along which we flip in order to get from $\tri$ to $\tri'$, then the coordinate change between $\tri$ and $\tri'$ is the composition map $\Phi_{e_k} \circ \dots \circ \Phi_{e_0}$. We explicitly give $\Phi_e$ below according to Figure \ref{fig:ARS_flip}. To simplify the notation, we denote $\phi_\tri(x)(q),\phi_{\tri'}(x)(q)$ by $q,q'$, for all $q \in \triangle \cup \underline{E}$ and $q' \in \triangle' \cup \underline{E}'$. Hence:

\begin{align*}
e_{01}' = \frac{e_{01}  e_{02}}{ e_{02}+1}, \quad & \quad
e_{10}' = \frac{e_{10} ( e_{02}+1) t_{012}  e_{20}}{ e_{02} t_{012}  e_{20}+t_{012}  e_{20}+ e_{20}+1},\\
e_{12}' = \frac{e_{12} ( e_{02} t_{012}  e_{20}+t_{012}  e_{20}+ e_{20}+1)}{ e_{20}+1},\quad & \quad
e_{21}' = e_{21} ( e_{20}+1),\\
e_{23}' = \frac{e_{23}  e_{20}}{ e_{20}+1},\quad & \quad
e_{32}' = \frac{e_{32}  e_{02} t_{023} ( e_{20}+1)}{ e_{02} t_{023}  e_{20}+ e_{02} t_{023}+ e_{02}+1},\\
e_{30}' = \frac{e_{30} ( e_{02} t_{023}  e_{20}+ e_{02} t_{023}+ e_{02}+1)}{ e_{02}+1},\quad & \quad
e_{03}' = e_{03} ( e_{02}+1),\\
e_{13}' = \frac{ e_{20}+1}{( e_{02}+1) t_{012}  e_{20}},\quad & \quad
e_{31}' = \frac{ e_{02}+1}{ e_{02} t_{023} ( e_{20}+1)},\\
t_{130}'= \frac{t_{023} ( e_{02} t_{012}  e_{20}+t_{012}  e_{20}+ e_{20}+1)}{ e_{02} t_{023}  e_{20}+ e_{02} t_{023}+ e_{02}+1},\quad & \quad
t_{123}'= \frac{t_{012} ( e_{02} t_{023}  e_{20}+ e_{02} t_{023}+ e_{02}+1)}{ e_{02} t_{012}  e_{20}+t_{012}  e_{20}+ e_{20}+1}.
\end{align*}

\begin{figure}[ht]
	\centering
	\includegraphics[width=14cm]{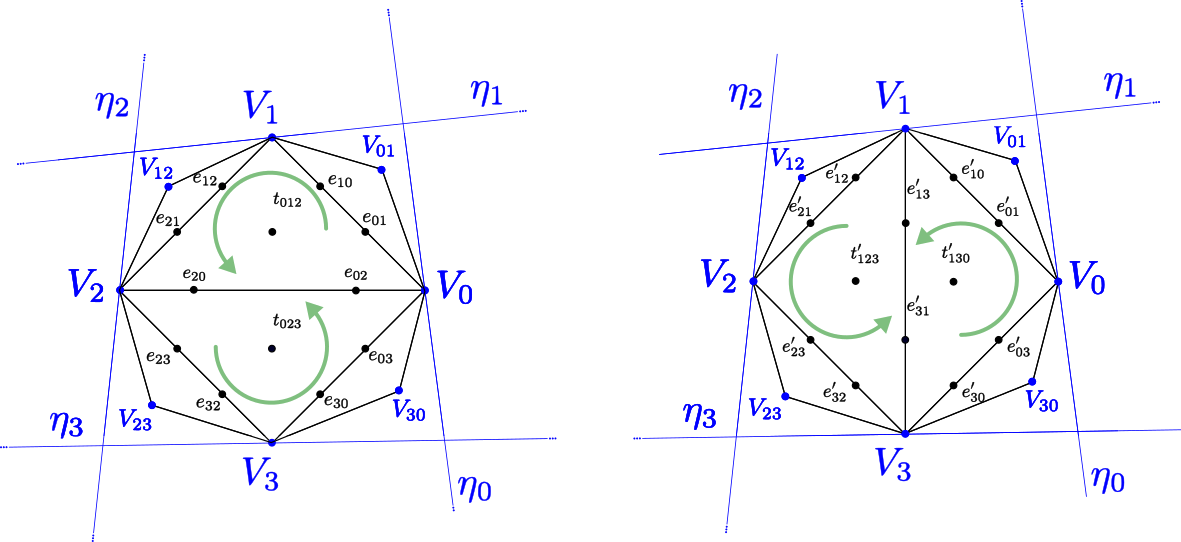}
	\caption{The change of coordinates for a flip. The quadrilateral formed by $V_i$, $i=0, \dots, 3$ is assumed to be in an affine patch such that the triangles are oriented anticlockwise.}
	\label{fig:ARS_flip}
\end{figure}


\subsection{Monodromy map} \label{subsect_monodromy_operator}

	In this paragraph we develop an efficient way to compute the monodromy map $\mu: \Ttp(S) \rightarrow \X_3(S)$. This has an immediate application in \S\ref{subsubsec:max-min-coordinates} and plays a key role in the proofs of \S\ref{subsec:fin_vol} and \S\ref{subsec:compatibility_brackets}.

	Henceforth, $S$ is assumed to be as in Theorem~\ref{thm:global-coord}, endowed with a framed convex projective structure $x \in \Ttp(S)$. We fix an orientation $\nu$ and an ideal triangulation $\tri$, with triangles and oriented edges $\triangle$ and $\underline{E}$. The canonical isomorphism of Theorem \ref{thm:global-coord} is given by $\phi = \phi_{\tri, \nu}$. As earlier, we simplify the notation by reducing $\phi(x)(q), \phi(x)(q)$ to $q, q'$, for all $q \in \triangle \cup \underline{E}$ and $q' \in \triangle' \cup \underline{E}'$. It will always be clear from context whether $q$ stands for a triangle, edge or coordinate.
	
	
\subsubsection{The monodromy graph}	\label{subsubsect:monodromy-graph}	
	
	We construct a sub-triangulation $\tri'$ of $\tri$, by connecting the vertices of every triangle with its barycentre. The \emph{monodromy graph} $\graph$ is the dual spine of $\tri'$. Let $\widetilde{\graph}$ and $\widetilde{\tri}$ be lifts of $\graph$ and $\tri$ to $\Omega$, via the developing map of $x$ (Figure \ref{monodromy_graph}).

	Let $ \underline{E}(\widetilde{\graph})$ be the set of oriented edges of $\widetilde{\graph}$. We now define the map
\[
\Mon \co  \underline{E}(\widetilde{\graph}) \rightarrow \SL(3,\RR).
\]
	Using the construction of Theorem \ref{thm:global-coord}, each triangle $t$ of $\widetilde{\tri}$ corresponds to a cyclically ordered triple of flags $\F^t := (\F^t_i, \F^t_j, \F^t_k)$ in $\RP2$.
\begin{itemize}
	\item[$(i)$] Suppose $e \in \underline{E}(\widetilde{\graph})$ is strictly contained in $t$. Then $e$ forms part of the boundary of a triangle in $\widetilde{\graph}$, which itself is contained in $t$. It makes sense to compare the orientation of $e$ with that of $t$.
	$\Mon(e)$ is defined as follows.
\begin{itemize}
	\item
	 If the orientation of $e$ agrees with that of $t$ then
	\[
		\Mon(e)(\F_i,\F_j,\F_k) := (\F_j,\F_k,\F_i),
	\]
	\item
	If the orientation of $e$ does not agree with that of $t$ then
	\[
		\Mon(e)(\F_i,\F_j,\F_k) := (\F_k,\F_i,\F_j).
	\]
\end{itemize}
	Such an edge is referred to as a \emph{$\triangle$--edge} and is depicted as $e_0$ in Figure \ref{example_monodromy}.

	\item[$(ii)$] Otherwise, $e \in \underline{E}(\widetilde{\graph})$ intersects an edge of $\widetilde{\graph}$, say between $V_j$ and $V_k$. Let $t'$ be the triangle associated with the triple $\F^{t'} := (\F^{t'}_k, \F^{t'}_j, \F^{t'}_l)$, and suppose $e$ is oriented from $t$ to $t'$. Then $\Mon(e)$ is the unique projective transformation defined by
\[
	\Mon(e)(V_i,\F_j,\F_k) = (V_l,\F_k,\F_j).
\]
	In this case $e$ is called an \emph{$\underline{E}$--edge}. Such an edge is depicted as $e_1$ in Figure \ref{example_monodromy}.
\end{itemize}

\begin{minipage}[t]{0.45\textwidth}
	\centering
	\includegraphics[height=5cm]{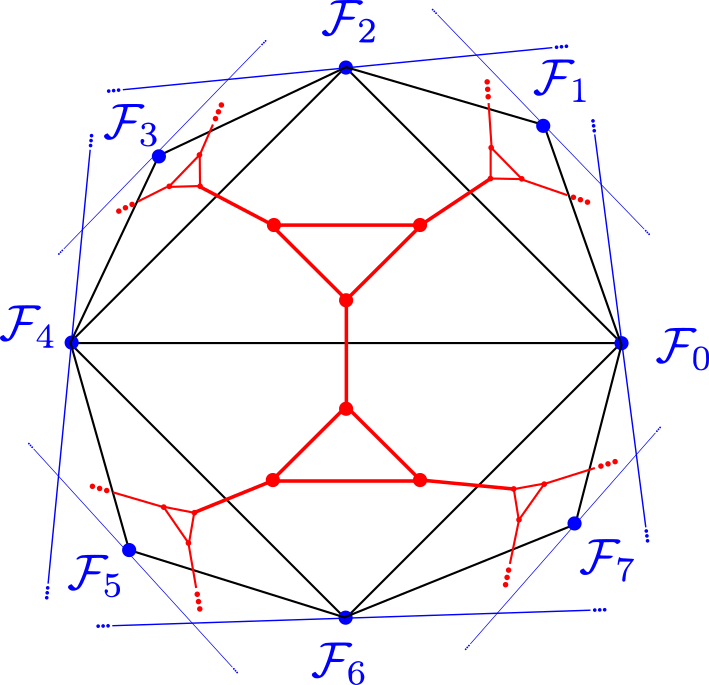}
	\captionof{figure}{The development of the monodromy graph is the dual spine of the development of the triangulation $\tri'$.}
	\label{monodromy_graph}
\end{minipage}
\begin{minipage}[t]{0.55\textwidth}
	\centering
	\includegraphics[height=5.5cm]{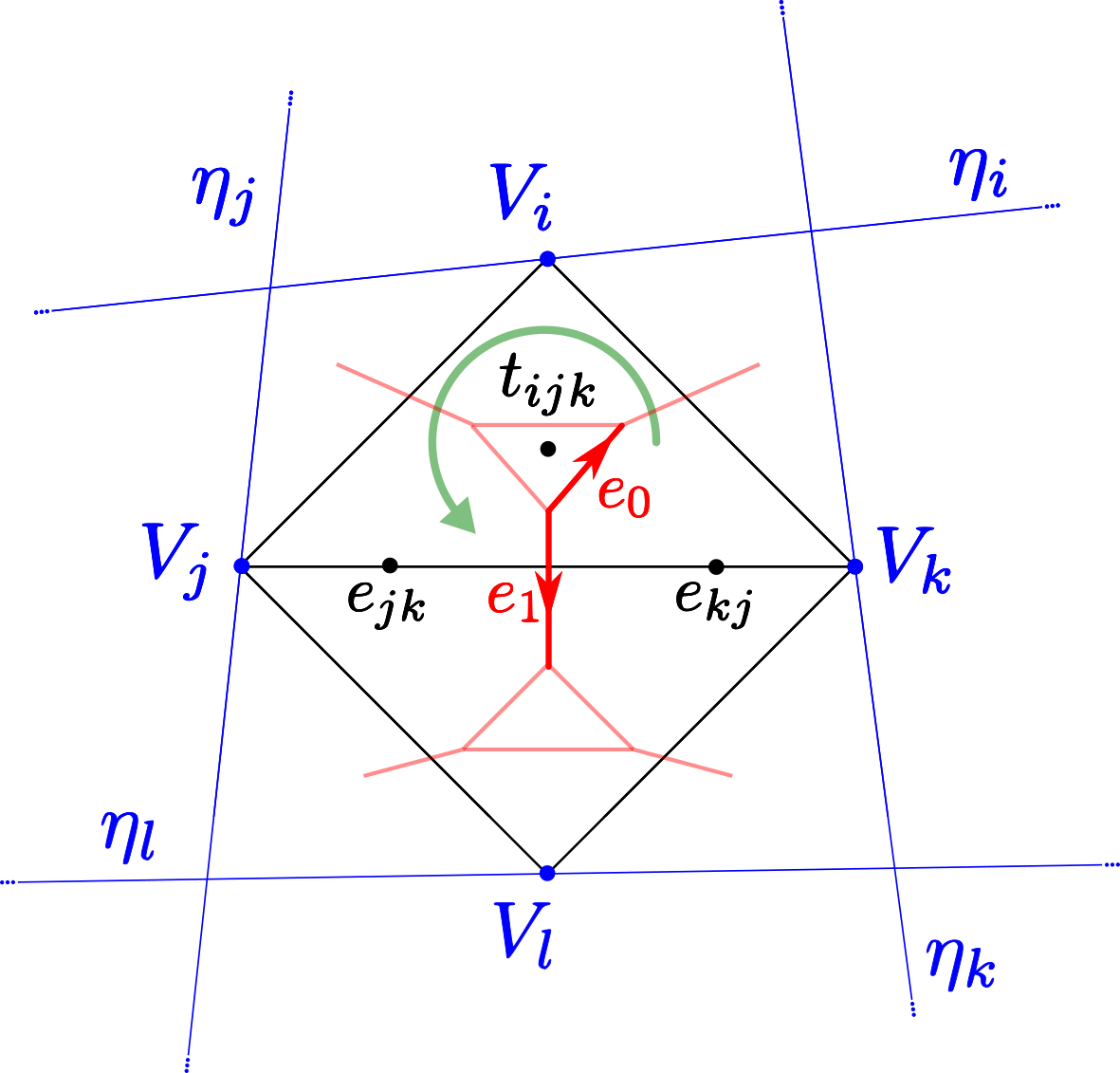}
	\captionof{figure}{Assuming the upper triangle is oriented anticlockwise with respect to the chosen affine patch, then $\Mon(e_0)(\F_i,\F_j,\F_k)=(\F_j,\F_k,\F_i)$ and $\Mon(e_1)(V_i,V_j,V_k,\eta_j\eta_k)=(V_l,V_k,V_j,\eta_j\eta_k)$.}
	\label{example_monodromy}
\end{minipage}

	We claim that $\Mon$ is well-defined. In case $(i)$, such a projective transformation always exists by Theorem~ \ref{thm:param_P_3} and the fact that $\cancel{3}(\F_i,\F_j,\F_k) = \cancel{3}(\F_j,\F_k,\F_i)$. Meanwhile, the operator defined in case $(ii)$ can be restated as the unique projective transformation such that
	\[
	\Mon(e)(V_i,V_j,V_k,\eta_j\eta_k) = (V_l,V_k,V_j,\eta_j\eta_k),
	\]
	from which existence and uniqueness are clear by the fact that both quadruples form a projective basis.

\begin{lem} \label{lem_compute_monodromy}
	Let $\gamma \in \pi_1(S)$. There exists a finite sequence $(e_0 ,  \dots , e_m)$ of oriented edges $e_i \in \underline{E}(\widetilde{\graph})$ such that the path $e_0 \cdot e_{1} \cdots e_m \subset \widetilde{\graph}$ is a lift of a loop $\widehat{\gamma} \subset \graph$, freely homotopic to $\gamma$. Furthermore,
	$$
	\mu(x)(\gamma) =  \left[ \prod_{i=0}^m \Mon(e_{m-i}) \right].
	$$
\end{lem}

\begin{proof}
	The graph $\graph$ contracts in $S$ to the dual spine of $\tri$ by collapsing all trivial loops. This implies that the natural homomorphism $\pi_1(\graph) \rightarrow \pi_1(S)$ is surjective.
	
	Let $\gamma \in \pi_1(S)$ and choose a loop in $\graph$ which is freely homotopic to $\gamma$. Lift this loop to $\widetilde{\graph}$ and let $w := e_0 \cdot e_{1} \cdots  e_m$ be the sequence of edges defined by the chosen lift. Denote by $(\F_0,\F_1,\F_2)$ the flags at the vertices of the triangle $t_{012}$ of $\widetilde{\tri}$ from which $w$ starts, and $(\F_3,\F_4,\F_5)$ the flags at the vertices of the triangle $t_{345}$ where $w$ ends. Note that $t_{012}$ and $t_{345}$ are lifts of the same triangle in $T$ as $w$ is a lift of a closed loop. 

	It is clear from the definition of $\Mon$ that $\prod_{i=0}^m \Mon(e_{m-i})$ is a projective transformation which maps the vertices and two flags of $t_{012}$ to the vertices and two flags of $t_{345}$. However $\cancel{3}(\F_0,\F_1,\F_2) = \cancel{3}(\F_3,\F_4,\F_5)$ so invoking Theorem~\ref{thm:param_P_3},
	\[
	\prod_{i=0}^m \Mon(e_{m-i}) \cdot (\F_0,\F_1,\F_2) = (\F_3,\F_4,\F_5).
	\]
	It follows that  $\prod_{i=0}^m \Mon(e_{m-i})$ agrees on $(\F_0,\F_1,\F_2)$ with a representative in $\mu(x)(\gamma)$. The claim follows from the fact that $\mu(x)(\gamma)$ is (the conjugacy class of) a projective transformation mapping $(\F_0, \F_1, \F_2)$ to $(\F_3, \F_4, \F_5)$ and that such a projective transformation, hence its conjugacy class, is unique.
\end{proof}

We remark that there was a choice of lift in the above construction. If another lift $w'$ is chosen, giving rise to a product  $\prod_{i=0}^m \Mon(e'_{m-i})$, and $A$ is the projective transformation taking the starting triangle of $w'$ to the starting triangle of $w$, then
\[
\prod_{i=0}^m \Mon(e'_{m-i}) = A^{-1} \cdot  \prod_{i=0}^m \Mon(e_{m-i}) \cdot A.
\]
Hence the product $\prod_{i=0}^m \Mon(e_{m-i})$ is defined up to conjugation in $\SL(3, \RR)$. However, this is the best we can hope for as the monodromy is defined only up to conjugation. Moreover, the triangulation $\tri$ is only defined up to a projective transformation.

\begin{rem} \label{rem_sequence_edges}
	The sequence in Lemma \ref{lem_compute_monodromy} can be chosen to alternate between $\triangle$--edges and $\underline{E}$--edges. More precisely, we can choose the $e_i$'s so that $(e_0 \cdot  \dots \cdot e_m)$ is a sequence of pairs of the form ($\underline{E}$--edge, $\triangle$--edge).
\end{rem}


\subsubsection{How to compute the monodromy} \label{subsubsec:computing-monodromy}

Consider the following special case.\\
Let $\{ \F_0,\F_1,\F_2,\F_3 \}$ be the following generic quadruple of flags:
\begin{align*}
	\F_0 = (V_0,\eta_0) = \left( \left[ \begin{array}{c} 0 \\ 0 \\ 1 \end{array} \right] , \left[ 1:0:0 \right] \right), \quad & \quad \F_1 = (V_1,\eta_1) = \left( \left[ \begin{array}{c} 1 \\ -1 \\ 1 \end{array} \right] , \left[t_{012}:t_{012}+1:1 \right] \right),\\
	\F_2 = (V_2,\eta_2) = \left( \left[ \begin{array}{c} 1 \\ 0 \\ 0 \end{array} \right] , \left[ 0:0:1 \right] \right),  \quad & \quad \F_3 = (V_3,\eta_3) = \left( \left[ \begin{array}{c} e_{02} e_{20} \\ e_{20} \\ 1  \end{array} \right] , \eta_3 \right).
\end{align*}
Then $\{V_0,V_1,V_2,\eta_0 \eta_2\}$ is a projective basis of $\RP2$. Moreover, $\eta_1$ and $V_3$ are uniquely determined by the parameters:
$$
\cancel{3}((\F_0,\F_1,\F_2)) = t_{012},
$$
$$
\cancel{4}(\F_0,\F_1,\F_2,\F_3) = e_{02} \qquad \mbox{ and } \qquad \cancel{4}(\F_2,\F_3,\F_0,\F_1) = e_{20}.
$$
Suppose $\{\F_0,\F_1,\F_2,\F_3\}$ are flags assigned to the vertices of some triangulation $\tri$, with monodromy graph $\graph$. Let $e_0$ be a $\triangle$--edges in $t_{012}$ and $e_1$ the $\underline{E}$--edge crossing the interval from $V_0$ to $V_2$. Assume $e_0$ is oriented according to the orientation of $t_{012}$, and $e_1$ is oriented towards $t_{023}$. See Figure \ref{example_monodromy} by setting $(k,i,j,l)=(0,1,2,3)$. Then one easily computes:
$$
\Mon(e_0) = \frac{1}{\sqrt[3]{t_{012}}}
\left(
\begin{array}{ccc}
0 & 0 & 1 \\
0 & -1 & -1\\
t_{012} & t_{012}+1 & 1
\end{array}
\right), \qquad \qquad
\Mon(e_1) = \sqrt[3]{\frac{e_{20}}{e_{02}}}
\left(
\begin{array}{ccc}
0 & 0 & e_{02}  \\
0 & -1 & 0\\
\frac{1}{e_{20}} & 0 & 0
\end{array}
\right).
$$

If we denote by $e_0^{-1},e_1^{-1}$ the same edges with opposite orientation, we have that $\Mon(e_0^{-1}) = \Mon(e_0)^{-1}$ and $\Mon(e_1^{-1}) = \Mon(e_1)^{-1}$.

We remark that the appearance of cubic roots can be avoided by replacing triangle invariants and edge invariants by their cubes; this was done in \cite{Haraway-tessellating-2018}.

\begin{lem} \label{lem_example_monodromy}
	Let $e$ be an oriented edge of $\widetilde{\graph}$ with tail in some triangle $t_{ijk}$ of $\widetilde{\tri}$, having flags $\F_i,\F_j$ and $\F_k$ assigned to its vertices. Let $P \in \SL(3,\RR)$ be the projective transformation mapping $(V_i,V_j,V_k,\eta_k\eta_j)$ to the projective basis $(V_1,V_2,V_0,\eta_0\eta_2)$ as above, and let
	$$
	T(z) := \frac{1}{\sqrt[3]{z}}
	\left( 
	\begin{array}{ccc}
	0 & 0 & 1 \\
	0 & -1 & -1\\
	z & z+1 & 1
	\end{array}
	\right), \qquad \qquad
	E(x,y) := \sqrt[3]{\frac{x}{y}}
	\left(
	\begin{array}{ccc}
	0 & 0 & y  \\
	0 & -1 & 0\\
	\frac{1}{x} & 0 & 0
	\end{array}
	\right).
	$$
	\begin{itemize}
		\item[$i)$] If $e$ is a $\triangle$--edge oriented as $t_{ijk}$, and $\cancel{3}((\F_i,\F_j,\F_k)) = t_{ijk}$, then
		$$
		\Mon(e) = P^{-1} \cdot T(t_{ijk}) \cdot P.
		$$
		\item[$ii)$] If $e$ is an $\underline{E}$--edge crossing the segment of endpoints $(\F_i,\F_k)$ oriented away from $t_{ijk}$, and $e_{ki},e_{ik}$ are the corresponding edge ratios, then
		$$
		\Mon(e) = P^{-1} \cdot E(e_{ki},e_{ik}) \cdot P.
		$$
	\end{itemize}
\end{lem}

\begin{proof}
	In both cases, right hand side and left hand side agree on the four points $\{V_i,V_j,V_k,\eta_k\eta_j\}$, so the statement follows.
\end{proof}

We may now compute $\mu(x)$ using Theorem~\ref{thm:global-coord}.\\
Let $\gamma$ be an oriented closed curve on $S$, in general position with respect to $\tri$ (and itself). Let $(t_0,\dots,t_m)$ be the sequence of triangles of $\triangle$ crossed by $\gamma$, cyclically ordered with respect to the orientation of $\gamma$ (possibly with repetitions). For each $i \in \{1,\dots, m\}$, suppose $t_i$ has the orientation induced by $S$ and that $\gamma$ enters $t_i$ through the edge $e_i' \in \underline{E}$ and leaves $t_i$ through $e_i \in \underline{E}$, where $e_i, e'_i$ are oriented according to $t_i$. A diagram of this situation is shown in Figure~\ref{epsilon_rightleft_turn}. We define the following quantities:
\[
E_i := E\left( \phi(x)(e_i) , \phi(x)(e_i^{-1} )\right) , \quad T_i := T\left( \phi(x)(t_i)   \right),
\]
\[
\epsilon_i :=
\begin{cases}
+1 & \text{ if } e_i \cap e_i' \text{ is the head of } e_i', \\
-1 & \text{ if } e_i \cap e_i' \text{ is the tail of } e_i'.
\end{cases}
\]

\begin{figure}[h]
	\centering
	\includegraphics[height=4cm]{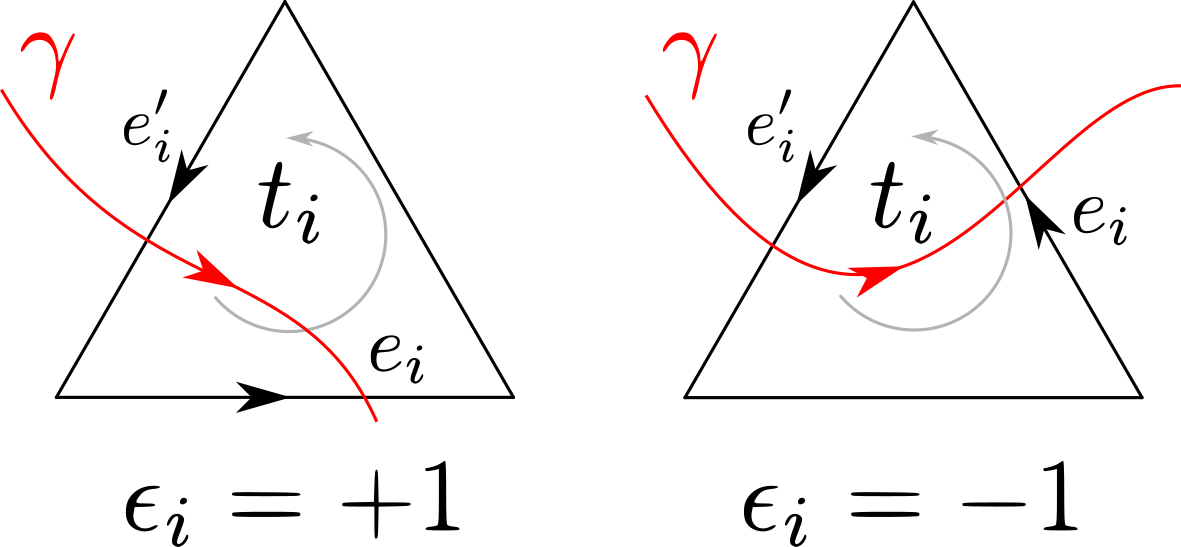}
	\captionof{figure}{	If we fix an affine patch in which $t_i$ is oriented anticlockwise, $\epsilon_i = +1$ or $-1$ respectively when $\gamma$ ``turns right" or ``turns left".}
	\label{epsilon_rightleft_turn}
\end{figure}

\begin{thm}\label{thm_monodromy}\label{thm:monodromy}
	$\mu(x)(\gamma)$ is the conjugacy class $[T_0^{\epsilon_0}\cdot E_0 \cdots T_m^{\epsilon_m}\cdot E_m]$. Furthermore, if $\gamma$ represents a peripheral element of $\pi_1(S)$ and $\epsilon_0 = 1$ (resp. $-1$), then $T_0^{\epsilon_0}\cdot E_0 \cdots T_m^{\epsilon_m}\cdot E_m$ is lower triangular (resp. upper triangular).
\end{thm}

\begin{proof}
	It follows from Lemma~\ref{lem_compute_monodromy}, Remark~\ref{rem_sequence_edges} and Lemma~ \ref{lem_example_monodromy} that $\mu(x)(\gamma)$ is the conjugacy class of
\[
Q := \prod_{i=0}^m P_{m-i}^{-1} \cdot E_{m-i} \cdot T_{m-i}^{\epsilon_{m-i}} \cdot P_{m-i},
\]
where $P_i \in \SL(3,\RR)$ is the projective transformation corresponding to the triangle $t_i$, as defined in Lemma~ \ref{lem_example_monodromy}.\\
	Observe that $P_j \cdot P_{j-1}^{-1} = T_j^{-\epsilon_j} \cdot E_{j-1}^{-1}$ for all $j \in \{1,\dots,m\}$, as both transformations agree on the four points $\{V_0,V_2,E_{j-1}(V_1),\eta_0 \eta_2\}$. Hence
	\begin{align*}
		Q &= \prod_{i=0}^m P_{m-i}^{-1} \cdot E_{m-i} \cdot T_{m-i}^{\epsilon_{m-i}} \cdot P_{m-i}\\
		&= \left( \prod_{i=0}^{m-1} P_{m-i}^{-1} \cdot E_{m-i} \cdot T_{m-i}^{\epsilon_{m-i}} \cdot T_{m-i}^{-\epsilon_{m-i}}\cdot E_{m-(i+1)}^{-1} \cdot P_{m-(i+1)} \right) \cdot P_{0}^{-1} \cdot E_{0} \cdot T_{0}^{\epsilon_{0}} \cdot P_{0}\\
		&= P_{m}^{-1} \cdot E_{m} \cdot T_{0}^{\epsilon_{0}} \cdot P_{0}.
	\end{align*}
	A simple inductive argument shows that
	$P_{m}^{-1} = P_{0}^{-1} \cdot \left( \prod_{i=0}^{m-1} E_{i} \cdot T_{i+1}^{\epsilon_{i+1}} \right)$. It follows that
	$$
	Q = P_{0}^{-1} \cdot E_{0} \cdot \left( \prod_{i=1}^{m} T_{i}^{\epsilon_{i}} \cdot E_{i} \right) \cdot T_{0}^{\epsilon_{0}} \cdot P_{0}.
	$$
	With regard to the second part of the statement, we remark that $\gamma$ represents a peripheral element of $\pi_1(S)$ if and only if it circuits some ideal vertex of $\tri$. In this case $\epsilon_i = \epsilon_j$ for all $i,j \in \{0,\dots,m\}$. Since $T_i \cdot E_i$ and $T_i^{-1} \cdot E_i$ are lower- and upper-triangular respectively, the product $T_0^{\epsilon_0}\cdot E_0 \dots T_m^{\epsilon_m}\cdot E_m$ is triangular as claimed.
\end{proof}


\subsection{Maximal and minimal ends in coordinates} 
\label{subsubsec:max-min-coordinates}

In this paragraph, we show a straightforward application of Theorem~\ref{thm_monodromy}. We provide a criterion to determine the geometry of an end using only the Fock-Goncharov coordinates of the structure. More specifically, not only may one distinguish among hyperbolic ends, special ends and cusps, but also between maximality and minimality as well as the different framings.

Let $E_i$ be an end of $S$. Let $V_i$ be a vertex of $\widetilde{\tri}$ developing $E_i$, and $\gamma_i \in \pi_1(S)$ a peripheral element whose holonomy fixes $V_i$. Let $\F_i = (V_i,\eta_i)$ be the framing at $V_i$. Denote by $e_{ik}$ for $k=1, \dots, m$ the edges of $\tri$ oriented away from $E_i$, and $e_{ki}$ the edges oriented toward $E_i$. Let $t_k^i$ be the triangle in $\tri$ appearing after the edge $e_{ik}$ according to the orientation of $\gamma_i$.

In Theorem~\ref{thm_monodromy} we show that the conjugacy class of $\hol(\gamma_i)$ has a representative of the form
\[
\prod_{k=1}^m T^{-1}(t_k^i) E(e_{ik}, e_{ki}) = \prod_{k=1}^m \sqrt[3]{ \frac{t_k^i e_{ik}}{e_{ki}}}
\begin{pmatrix} \frac{1}{t_k^i e_{ik}} & \frac{t_k^i+1}{t_k^i} & e_{ki} \\ 0 & 1 & -e_{ki} \\ 0 & 0 & e_{ki} \end{pmatrix}.
\]
The eigenvalues are
$$
\lambda_1 = \prod_{k=1}^m(e_{ki} (t_k^i)^2 e_{ik}^2)^{-\frac{1}{3}}, \qquad \lambda_2 = \prod_{k=1}^m(e_{ki}^{-1} t_k^i e_{ik})^{\frac{1}{3}}, \qquad \lambda_3 = \prod_{k=1}^m(e_{ki}^{2} t_k^i e_{ik})^{\frac{1}{3}}.
$$
As an artefact of the construction in Lemma~\ref{lem_example_monodromy}, $\lambda_1$ is always the eigenvalue corresponding to the eigenvector $V_i$, and $\lambda_2$ corresponds to the other eigenvector contained in $\eta$. One easily deduces the following classification:
\begin{enumerate}
	\item $E_i$ is a hyperbolic end:
	\begin{enumerate}
		\item $E_i$ is maximal, i.e. $V_i$ is the saddle eigenvector:
		\begin{itemize}
			\item $\eta_i$ passes through the attracting eigenvector: 
				\[\prod_{k=1}^m t_k^i e_{ik}  > \; 1 \ \text{and} \ \prod_{k=1}^m t_k^i e_{ik}  e_{ki} < \; 1 \]
			\item $\eta_i$ passes through the repelling eigenvector: 
				\[\prod_{k=1}^m  t_k^i e_{ik} < \; 1 \ \text{and} \ \prod_{k=1}^m  t_k^i e_{ik} e_{ki} > \; 1\]
		\end{itemize}
		\item $E_i$ is minimal, i.e. $V_i$ is not the saddle eigenvector:
		\begin{itemize}
			\item $V_i$ is the attracting eigenvector and $\eta_i$ passes through the saddle eigenvector:
				\[\prod_{k=1}^m t_k^i e_{ik}  < \; 1 \ \text{and}\ \prod_{k=1}^m e_{ki} < \; 1\]
			\item $V_i$ is the attracting eigenvector and $\eta_i$ passes through the repelling eigenvector: 
				\[\prod_{k=1}^m t_k^i e_{ik}  e_{ki} < \; 1 \ \text{and} \ \prod_{k=1}^m e_{ki} > \; 1 \]
			\item $V_i$ is the repelling eigenvector and $\eta_i$ passes through the saddle eigenvector: 
				\[\prod_{k=1}^m  t_k^i e_{ik}  > \; 1 \ \text{and}\ \prod_{k=1}^m  e_{ki} > \; 1\]
			\item $V_i$ is the repelling eigenvector and $\eta_i$ passes through the attracting eigenvector: 
				\[\prod_{k=1}^m t_k^i e_{ik} e_{ki} > \; 1 \ \text{and} \ \prod_{k=1}^m  e_{ki} < \; 1 \]
		\end{itemize}
	\end{enumerate}
	\item $E_i$ is a special end:
	\begin{enumerate}
		\item $V_i$ is the attracting eigenvector of multiplicity two:
		\begin{itemize}
			\item $\eta_i$ does not pass through the repelling eigenvector: 
				$$\prod_{k=1}^m t_k^i e_{ik}  = \; 1 \ \text{and} \ \prod_{k=1}^m e_{ki} < \; 1 $$
			\item $\eta_i$ passes through the repelling eigenvector: 
			$$\prod_{k=1}^m  t_k^i e_{ik} e_{ki} = \; 1 \ \text{and} \ \prod_{k=1}^m e_{ki} > \; 1$$
		\end{itemize}
		\item $V_i$ is the repelling eigenvector of multiplicity two:
		\begin{itemize}
			\item $\eta_i$ does not pass through the attracting eigenvector: 
			$$\prod_{k=1}^m t_k^i e_{ik}  = \; 1 \ \text{and} \ \prod_{k=1}^m e_{ki} > \; 1 $$
			\item $\eta_i$ passes through the attracting eigenvector: 
			$$\prod_{k=1}^m  t_k^i e_{ik} e_{ki} = \; 1 \ \text{and} \ \prod_{k=1}^m e_{ki} < \; 1$$
		\end{itemize}
		\item $V_i$ is the attracting eigenvector of multiplicity one: 
		$$\prod_{k=1}^m t_k^i e_{ik}  < \; 1 \ \text{and}\ \prod_{k=1}^m e_{ki} = \; 1$$
		\item $V_i$ is the repelling eigenvector of multiplicity one: 
		$$\prod_{k=1}^m t_k^i e_{ik}  > \; 1 \ \text{and}\ \prod_{k=1}^m e_{ki} = \; 1$$
	\end{enumerate}
	\item $E_i$ is a cusp:
	$$\prod_{k=1}^m t_k^i e_{ik}  = \; 1 \ \text{and}\ \prod_{k=1}^m e_{ki} = \; 1$$
\end{enumerate}



\section{Applications of the parameterisation}
\label{sec:Applications of the parameterisation}


\subsection{Structures of finite area}\label{subsec:fin_vol}

	We seek now to determine necessary and sufficient conditions for a framed convex projective structure to have finite area, given its representative in $\RR^{\triangle \cup \underline{E}}_{>0}$. This leads to a new proof of a result from Marquis~ \cite{Marquis-surface-2012}.

	Throughout this section, $S = S_{g,n}$ will be a surface of negative Euler characteristic, with at least one puncture and orientation $\nu$. Fix a framed convex projective structure on $S$ with holonomy $\hol$ and developing map $\dev$. Furthermore, fix an ideal triangulation $\tri$ of $S$, let $\Omega := \dev(\widetilde{S}) \subset \RP2$ and denote by $\widetilde{\tri}$ the lift of $\tri$ to $\Omega$. The map $\phi = \phi_{\tri, \nu}$ is the canonical isomorphism defined in Theorem~\ref{thm:global-coord}. We identify $e_{ij}, t_{ijk} \in \tri$ with $\phi(x)(e_{ij})$ and $\phi(x)(t_{ijk})$ respectively, so as to simplify notation. 

A proof of the following lemma was given by Marquis~\cite{Marquis-surface-2012}. In order to keep this part of the notes self-contained, we give an adaptation of that proof here.

\begin{lem}\label{lem:para-fin}
The area of $S$ is finite if and only if every end of $S$ is a cusp.
\end{lem}
\begin{proof}
	The surface $S$ has finite area if and only if each ideal triangle in $\widetilde{\tri}$ has finite area. Vertices of those ideal triangles correspond to ends of $S$, therefore it is enough to show that an ideal triangle has finite area if and only if its vertices correspond to cusps.
	
	Fix an ideal triangle $t$ of $\widetilde{\tri}$ and denote its vertices by $V_0, V_1$ and $V_2$. Let $E$ be the end corresponding to $V_0$ and let $\gamma \in \pi_1(S)$ be a generator of the subgroup conjugate to $\pi_1(E)$ such that $\hol(\gamma)$ fixes $V_0$. Firstly, we show that if $E$ is not a cusp, the area of $t$ is unbounded. There are two cases, depending on the regularity of $\partial \overline{ \Omega}$.
	
	Suppose $E$ is either hyperbolic, or special where the eigenvalue of $V_0$ has multiplicity two. These are the cases where $\partial \overline{ \Omega}$ is $C^0$ but not $C^1$ at $V_0$. Let $\eta_1$ and $\eta_2$ be distinct supporting lines to $\Omega$ at $V_0$. One may construct a triangle $\Delta$ with vertex $V_0$, strictly contained in a common affine patch with $\Omega$, having two edges on $\eta_1$ and $\eta_2$, and strictly containing $\Omega$. Observe that $\Delta$ and $t$ share the vertex $V_0$. We are going to show that $\mu_{\Delta}(t) = \infty$, from which it follows that $\mu_{\Omega}(t) = \infty$. There is a projective transformation of $\RP2$ such that
	$$
	V_0 = \left[
	\begin{array}{c}
	0 \\ 0 \\ 1
	\end{array}
	\right]
	\qquad
	\text{ and }
	\qquad
	\Delta = \left\{ \left[
	\begin{array}{c}
	x \\ y \\ 1
	\end{array}
	\right] \ | \ x, y > 0 \right\}.
	$$
	In the affine patch $\A : \{ z = 1 \}$, $V$ is the origin and $\Delta$ is the positive quadrant. Let $\epsilon>0$ and $0 < a < b$ such that
	$$
	t_0 = \{ (x,y) \in \A \ | \ 0<x<\epsilon \text{ and } ax < y < bx \} \subset t.
	$$
	Let $l_a : \{y=ax\}$ and $l_b : \{y=bx\}$.	A direct calculation shows that
	$$
	d_\Delta(l_a,P) = \text{constant} \qquad \text{and} \qquad d_\Delta(V_0,P) = \infty, \quad \forall \; P \in l_b \cap \{0<x< \epsilon\}, 
	$$
	therefore $\mu_{\Delta}(t_0) = \infty$ and $\mu_{\Delta}(t) = \infty$.
	
	Now suppose $E$ is special, but $V_0$ has eigenvalue of multiplicity one. In this case $\partial \overline{\Omega}$ is $C^1$ at $V_0$, but not $C^2$, therefore there is a unique supporting line $\eta$ to $\partial \overline{\Omega}$ at $V_0$. We are going to apply the same principle as in the previous case, but with respect to a different set $\Delta$. Let $\lambda$ be the eigenvalue of $\hol(\gamma)$ with multiplicity two. Then there is a projective transformation of $\RP2$ such that
	$$
	\hol(\gamma) =
	\begin{pmatrix} 
	\lambda & 0 & 0 \\
	\lambda & \lambda & 0 \\
	0 & 0 & 1/\lambda^2
	\end{pmatrix}.
	$$
	In this situation, $V_0 = \left[0 : 0 : 1\right]^t$. For all points $P = \left[ x_0 : y_0 : 1 \right]^t \in \RP2$ with $x_0>0$, there is a curve $\C_P \subset \RP2$ starting at $V_0$ passing through $P$ and preserved by $\hol(\gamma)$. Namely
	$$
	\C_P(t) = \left\{ \left.
	\left[
	\begin{array}{c}
	 t \\ t ( \frac{y_0}{x_0}+ \frac{\log(t) - \log(x_0) }{3 \log(\lambda) } ) \\ 1
	 \end{array}
	 \right]
	 \  \right|  \ t > 0 \right\}.
	$$
	If $P \notin \overline{\Omega}$, then $\C_P$ is disjoint $\Omega$. Hence choose such a $P$ and fix $\C_P$. After the affine transformation
	$$
	\left(x,y\right) \mapsto \left( \frac{ e^{\frac{y_0}{x_0}} \lambda^3 }{x_0} x , \frac{ \ln(\lambda^3) e^{\frac{y_0}{x_0}} \lambda^3 }{x_0}y \right),
	$$
	$\C_P$ is mapped to $ \C'_P(t) := \left\{ \left. \left[ t : t \log(t) : 1 \right]^t	\  \right|  \ t >0 \right\}$. Thus $\Omega$ is strictly contained in the domain
	$$
	\Delta :=  \left\{ \left[
	\begin{array}{c}
	x \\ y \\ 1
	\end{array}
	\right] \ | \ 0 < x, \ x \log(x) < y \right\}.
	$$
	Once again, $\Delta$ and $t$ share the vertex $V_0$. By working in the affine patch $\A : \{ z = 1 \}$, one can apply a similar argument as in the previous case to show that $\mu_{\Delta}(t) = \infty$. The salient difference is that the ``width" of $t$ in $\Delta$ is not constant, but decreases sufficiently slowly.
	
	It remains to show that if each ideal vertex of $t$ corresponds to a cusp, then $t$ has finite area. This is the case where $\partial \Omega$ is smooth at the vertices. Making use of an elegant proof from Colbois, Vernicos and Verovic \cite{Colbois-aire-2004}, the fact that $\partial \Omega$ is $C^2$ at $V_i$, ensures the existence of an open subset $B_i$, strictly contained in $\Omega$, whose frontier is a conic sharing the supporting line at $V_i$ with $\Omega$. The conic $\partial \overline{B_i}$ intersects $\partial \overline{T}$ in three points, one of which is $V_i$, and the others will be denoted $t_{i,1}$ and $t_{i,2}$. Let $t_i$ be the triangle in the interior of $t$ with vertices $V_i, \; t_{i,1}$ and $t_{i,2}$. Then 
	\[
	4 \pi = \mu_{B_i} (t_i) \;  > \; \mu_{\Omega}(t_i).
	\]
	The set $C := \overline{t \; \backslash \; ( t_0 \cup t_1 \cup t_2)} \subset \Omega $ is compact so $\mu_{\Omega}(C) < \infty$. It follows that $\mu_{\Omega}(C \cup t_0 \cup t_1 \cup t_2) <  \infty$ and therefore $\mu_{\Omega }(t) <  \infty$.
\end{proof}

\begin{thm}[Marquis 2010, \cite{Marquis-espace-2010}]\label{thm:Marquis-main}
The space of finite-area structures $\Ttfin(S)$ is an open cell of dimension $16g-16+6n = -8 \chi(S) - 2n$.
\end{thm}

\begin{proof}
	We have seen in Theorem~\ref{thm:global-coord} that $\Ttp(S)$ is an open cell of dimension $- 8 \chi(S) = 16g-16+8n.$ Moreover, it is shown in Lemma~\ref{lem:para-fin} that the finite-area condition is equivalent to the requirement that every end of $S$ is cuspidal.
	
	Fix an ideal vertex $V_i$ of $\tri$ and let $\gamma_i \in \pi_1(S)$ be a peripheral element around $V_i$ representing a simple loop. Denote by $e_{ik}$ for $k \in \{1, \dots, m_i\}$ the edges of $\tri$ oriented away from $V_i$ and $e_{ki}$ the edges oriented toward $V_i$. Let $t_k^i$ be the triangle in $\tri$ appearing after the edge $e_{ik}$ according to the orientation of $\gamma_i$. From the discussion in \S\ref{subsubsec:max-min-coordinates}, $\hol(\gamma)$ is parabolic if and only if
\[
	\prod_{k=1}^{m_i} e_{ki} =: \mathcal{X}_i = 1 \quad \text{ and} \quad \prod_{k=1}^{m_i} e_{ik} t_k^i =: \mathcal{Y}_i = 1.
\]
	As conjugation preserves eigenvalues, the parabolicity conditions reduce to two such equations for each of the $n$ vertices in $\tri$. It remains to show that the $2n$ equations thus obtained are independent. Suppose the converse, then there is some equation of the form
\begin{equation}
\mathcal{X}_1^{c_1} \mathcal{Y}_1^{d_1} \dots \mathcal{X}_n^{c_n} \mathcal{Y}_n^{d_n } =1 \label{eq:Vi} 
\end{equation}
which is non-trivial in the sense that it holds when $c_k$ and $d_j$ are non-zero for some values of $k$ and $j$.

	Each oriented edge $e_{ij}$ has endpoints in only two vertices. Therefore this edge appears in $\mathcal{X}_i$ for exactly one value of $i$ and in $\mathcal{Y}_j$ for exactly one value of $j$. Suppose the edge $e_{ij}$ is oriented from a vertex $V_i$ to a vertex $V_j$. It follows that for (\ref{eq:Vi}) to hold, we must have $c_i = - d_j $. 

	Now consider the triangle in $\tri$ with vertices $V_1$, $V_2$ and $V_3$. Using the above consideration, it follows that
\[
c_1 = - d_2 = c_3 =  - d_1 = c_2 = - d_3.
\]
	Completing this for all triangles in $\tri$, it follows that $c_i = \pm d_j \quad \forall \;  i ,j \in \{ 1, \ldots n\}$.  Now consider the degree of each triangle parameter term in (\ref{eq:Vi}). Each such term appears exactly $3$ times in (\ref{eq:Vi}). The degree of the term $t_i$ in the left hand side of (\ref{eq:Vi}) must be in the set $\{\pm c_1, \pm 3 c_1 \}$. Hence $c_i = d_j =0 \quad \forall \; i,j$ and equation (\ref{eq:Vi}) may only hold trivially.

	It follows that the $2n$ equations generated by the parabolicity conditions are independent and
\[
\text{dim}_{\RR}(\Ttfin(S)) = \text{dim}_{\RR}(\Ttp(S)) -2n = 16g-16+6n.
\]
	To show that $\Ttfin(S)$ is a topological cell, it suffices to note that $\Ttp(S)$ is a cell and that imposing the parabolicity conditions amounts to retracting $\Ttp(S)$ onto an algebraic subvariety.
\end{proof}


\subsection{The classical Teichm\"uller space} \label{subsec:classical-teich}


Having obtained a subvariety of $\Ttp(S)$ comprising structures of finite area there is another obvious subvariety that we wish to interrogate -- the set of points which are invariant under the projective duality $\phi^-_{\tri} \circ \sigma^+ \circ (\phi^+_{\tri})^{-1}$ defined in  \S\ref{subsubsec:change-of-coord-involution}. This is of interest because, for instance, results dating back to Kay~\cite{Kay-ptolemaic-1967} indicate that these structures should be identified with hyperbolic structures on $S$. To that end we use this section to characterise the Fock-Goncharov coordinates of these structures and in so doing reproduce a classical result of Fricke and Klein~\cite{Fricke-vorlesungen-1965} regarding the toplogy of the Teichm\"uller space $\Teich(S)$ of $S$.

As usual $S = S_{g,n}$ is a surface of negative Euler characteristic, with at least one puncture and orientation $\nu$. We fix an ideal triangulation $\tri$ of $S$, and let $\phi = \phi_{\tri, \nu}$ be the canonical isomorphism constructed in Theorem~\ref{thm:global-coord}. 

A quick calculation shows that the Fock-Goncharov coordinates which are invariant under projective duality are
\begin{align}
\begin{split}
		\phi(x)(e) &= \phi(x)(e^{-1}), \quad \forall \; e \in \underline{E},\label{equ:FG duality} \\
		\phi(x)(t) &= 1, \hspace{2cm} \forall \; t \in \triangle,  
\end{split}
\end{align}

One can easily deduce from the classification in \S\ref{subsubsec:max-min-coordinates} that this can only happen when each end is either a cusp or a minimal hyperbolic end, framed by the line through the saddle point. If we restrict our attention to the former, the following result elucidates how we identify the convex projective structure in question with a hyperbolic structure.

\begin{lem} \label{lem:coordinates-of-conic}
	Let $x = (\Omega, \Gamma, f) \in \Ttp(S)$. For every vertex $V_i$ of $\tri$, let $e_{ik} \in \underline{E}, \ k = 1,\dots,m_i$, be the oriented edges starting from $V_i$. Then $\partial \overline{ \Omega}$ is a conic of $\RP2$ if and only if
	\begin{align} 
	\begin{split}
		\phi(x)(e) &= \phi(x)(e^{-1}), \quad \forall \; e \in \underline{E}, \label{eq:finite-area-hyperbolic-equations}\\
		\phi(x)(t) &= 1, \hspace{2cm} \forall \; t \in \triangle,  \\
		\prod_{k=1}^{m_i} \phi(e_{ik}) &= 1, \hspace{2cm} \forall \; i =1,\dots,n.
		\end{split}
	\end{align}
\end{lem}

\begin{proof}
	We begin with a preliminary observation. Suppose $\{\F_0,\F_1,\F_2,\F_3\}$ is a generic quadruple of flags of $\RP2$, where $\F_i = (V_i,\eta_i)$. Let $\mathcal{C}$ be the unique conic passing through $V_0,V_1,V_2$ and tangent to $\eta_0$ and $\eta_1$. A straightforward computation shows that $\mathcal{C}$ is tangent to $\eta_2$ if and only if $\cancel{3}((\F_0,\F_1,\F_2)) = 1$ (see Remark~\ref{rem:example}). Furthermore, $V_4\in \mathcal{C}$ if and only if $\cancel{4}(\F_0,\F_1,\F_2,\F_3) = \cancel{4}(\F_2,\F_3,\F_0,\F_1)$.
	
	Let $\widetilde{\tri}$ be the lift of $\tri$ to $\Omega$. By Lemma~\ref{lem:para-fin}, the projective structure $x$ has finite volume if and only if all ends are cusps. In that case, the set of vertices of $\widetilde{\tri}$ is dense in $\partial \overline{\Omega}$, therefore $\partial \overline{\Omega}$ is a conic if and only if there is a conic passing through all the vertices of $\widetilde{\tri}$. We can conclude that $\partial \overline{ \Omega}$ is a conic if and only if
	\begin{itemize}
		\item[$a)$] for each vertex $V_i$,
				\begin{equation}
				\prod_{k=1}^{m_i} t_k^i e_{ik}  = \; 1 \qquad \text{and} \qquad \prod_{k=1}^{m_i} e_{ki}^{-1} = \; 1, \label{eq:first-equation}
				\end{equation}
				where $\{t_k^i\}_{k = 1,\dots,m_i}$ is the set of triangles in $\tri$ with a vertex at $V_i$. These conditions force the ends to be cusps, as underlined in the classification in \S\ref{subsubsec:max-min-coordinates}.
		\item[$b)$] for all edges $e \in \underline{E}$ and triangles $t \in \triangle$,
				\begin{equation}
				\phi(x)(e) = \phi(x)(e^{-1}) \qquad \text{and} \qquad \phi(x)(t) = 1. \label{eq:second-equation}
				\end{equation}
				This follows from the preliminary discussion. In fact, assuming all ends are cusps, one can construct a conic through all the vertices of $\tri$, tangent to the supporting lines defined by the flags, if and only if every triple ratio is equal to $1$ and edge ratios of opposite oriented edges are equal.
	\end{itemize}
	Combining equations (\ref{eq:first-equation}) and (\ref{eq:second-equation}) gives precisely the conditions of the lemma.
\end{proof}

The Klein disc embeds conformally in $\RP2$ as the interior of a conic and the group of orientation-preserving hyperbolic isometries, $\PSL(2, \RR)$, is isomorphic to $PO^+(1,2) \leq \SL(3, \RR)$. Hence every hyperbolic structure on $S$ is naturally a convex projective structure on $S$. Conversely if $\Omega$ is a conic then, up to conjugation, the corresponding holonomy group is a subgroup of $ \text{Aut}(\Omega) \cong PO^+(1,2)$. Moreover the Hilbert metric on a conic is twice the natural hyperbolic metric in the Klein model. Two hyperbolic structures are equivalent as hyperbolic structures if and only if they are equivalent as projective structures.

We have seen that finite-area structures admit a unique natural framing (resp. dual framing) so one can canonically identify $\Teich(S)$ with a subvariety of $\Ttp(S)$ despite the fact that $\Teich(S)$ contains no information regarding a framing. Thus we deduce the following result.

\begin{thm}[Fricke-Klein 1926, \cite{Fricke-vorlesungen-1965}]
The Teichm\"uller space of $S=S_{g,n}$ is homeomorphic to $\RR^{6g-6 + 2n}$.
\end{thm}

\begin{proof}
The parameter space $\Teich(S)$ is the set of those structures $(\Omega, \Gamma, f) \in \Ttp(S)$ such that $\partial \overline{ \Omega}$ is a conic of $\RP2$. Recall from Theorem \ref{thm:global-coord} that $\Ttp(S)$ is an open cell of dimension $- 8 \chi(S) = 16g-16+8n$. By Lemma~\ref{lem:coordinates-of-conic}, $\Teich(S)$ is a subvariety of $\Ttp(S)$ defined by the $(-5 \chi(S) + n)$ equations in (\ref{eq:finite-area-hyperbolic-equations}). This set of equations is independent, thus
$$
\dim_{\RR}(\Teich(S)) = -3 \chi(S) - n = 6g-6+3n - n =  6g-6+2n.
$$
Furthermore, there is a natural retraction of $\Ttp(S)$ onto $\Teich(S)$, therefore $\Teich(S)$ is also an open cell, and the claim follows.
\end{proof}


\subsection{Closed surfaces}\label{subsec:closed_surfaces}

A result of Goldman~\cite{Goldman-convex-1990} determines that if $S$ is an orientable, closed surface then $\Tt(S)$ is an open cell of dimension $- 8 \chi(S)$. We are going to give a different proof of the same result, using Fock-Goncharov coordinates. This requires a notion of how to glue two ends of a convex projective surface so as to induce a convex projective structure on the resulting glued surface. The more technical details of this gluing are omitted, for which we refer the reader to the original result.

Bonahon and Dreyer~\cite{Bonahon-parameterizing-2014} also use Fock and Goncharov's coordinates to prove this result employing geodesic laminations in place of ideal triangulations. We present this version here to motivate the usefulness of having a simple expression for the monodromy map.

\begin{thm}[Goldman 1990, \cite{Goldman-convex-1990}]
	If $S$ is a closed, orientable surface of negative Euler characteristic, then $\Tt(S)$ is an open cell of dimension $- 8 \chi(S)$.
\end{thm}
\begin{proof}
	Suppose $S$ is of genus $g$, endowed with a convex projective structure $(\Omega, \Gamma, f)$. That is, $S$ is the quotient of a properly convex domain $\Omega \subset \RP2$ by the action of a discrete subgroup $\Gamma < \SL(3, \RR)$ marked by $f$. Denote the holonomy representation of $S$ by $\hol$ and developing map by $\dev$. 
	
	Choose an essential, simple, closed, non-separating geodesic on $S$ and denote it by $\gamma$. Let $F$ be the interior of the surface obtained from $S$ by cutting along $\gamma$, and let $f_0$ be the restriction of $f$ to $F$. Then $F$ is a topologically finite surface, of the same Euler characteristic as $S$, with two ends, say $E_+$ and $E_-$. Let $\gamma_\pm$ be the generator of the peripheral subgroup $\pi_1(E_\pm)$, homotopic to $\gamma$ in $S$. The restriction of $\hol$ to curves in $F$ endows $F$ with a holonomy representation which we denote by $\hol_0$. Moreover, $F$ inherits a developing map from $S$, as the universal cover $\widetilde{S}$ of $S$ is tiled by copies of the universal cover of $F$. Choose one such copy of the universal cover of $F$ and name it $\widetilde{F}$. Let $\dev_0$ be the developing map obtained by restricting $\dev$ to $\widetilde{F}$. It is clear that $\Omega_0 := \dev_0(\widetilde{F})   \subset \RP2$ is a convex domain preserved by the action of $\Gamma_0 := \hol_0(\pi_1(F))$, thus $(\Omega_0,\Gamma_0,f_0)$ is a convex projective structure on $F$. 

	The surface $F$ has minimal hyperbolic ends at both $E_+$ and $E_-$ as the peripheral elements at those ends are inherited from the geometry of $S$. We frame $E_+$ and $E_-$ by choosing the $\Gamma_0$--orbit of the attracting fixed point and the line through the saddle point. This choice is arbitrary, but it gives a well-defined map
	\[
	\Psi : \Tt(S) \rightarrow \Ttp(F).
	\]
	By construction, $\Psi(\Tt(S))$ is a subset of the set of structures $\mathcal{V} \subset \Ttp(F)$, wherein:
	\begin{itemize}
		\item[$(a)$] All ends are minimal hyperbolic ends, framed by the attracting fixed points and the lines through the saddle points,
		\item[$(b)$] The transformation $\hol_0(\gamma_+)$ is conjugate to $\hol_0(\gamma_-)$.
	\end{itemize}
	From the classification in \S\ref{subsubsec:max-min-coordinates}, we know that the set of structures satisfying $(a)$ is an open cell of $\Ttp(F)$, hence of full dimension. On the other hand, condition $(b)$ amounts to requiring that the eigenvalues of $\hol_0(\gamma_+)$ and $\hol_0(\gamma_-)$ are equal. This amounts to two equations easily deducible from the standard form of a peripheral element, as shown in \S\ref{subsubsec:max-min-coordinates}. Therefore $\mathcal{V}$ is an open cell of dimension $-8 \chi(F) - 2$.

	We are going to show that $\Psi(\Tt(S)) = \mathcal{V}$. Let $(\Omega_0, \Gamma_0, f_0) \in \mathcal{V}$, with holonomy $\hol_0$ and developing map $\dev_0$. Let $\gamma_\pm$ be the generator of $\pi(E_\pm)$ such that the attracting fixed point of $\hol_0(\gamma_\pm)$ is one of the vertices of the framing of $E_\pm$. By definition, $\hol_0(\gamma_+)$ and $\hol_0(\gamma_-)$ are conjugate, so let $A \in \SL(3,\RR)$ be a conjugating matrix
	$$
	\hol_0(\gamma_+) = A^{-1} \cdot \hol_0(\gamma_-) \cdot A
	$$
with the property that $A$ maps the attracting (resp. repelling, saddle) fixed point of $\hol_0(\gamma_+)$ to the attracting (resp. repelling, saddle) fixed point of $\hol_0(\gamma_-).$ This is done with respect to a lift of $\Omega$ to $\Sph^2.$
	
	Define the group $\Gamma := \langle \Gamma_0, A \rangle < \SL(3,\RR)$. Let $s$ be the open segment between the attracting and repelling fixed points of $\hol_0(\gamma_-)$. Hence define $\Omega$ to be the orbit of $\Omega_0 \cup s$ under $\Gamma$. 
	
	We claim that $\Omega$ is properly convex, tesselated by copies of $\Omega_0,$ and that $\Omega / \Gamma \cong S$. By definition, $\Omega_0$ and $A \cdot \Omega_0$ share supporting projective lines at the endpoints of $s$. This ensures that 
	$A \cdot \Omega_0$ is contained in the interior of the closed triangle $\delta$ defined by the fixed points of $\hol_0(\gamma_-)$ which contains the interval $s$ and which does not contain $\Omega_0$.  Both $\Omega_0$ and $A \cdot \Omega_0$ are convex, and so 
$\Omega_0 \cup A \cdot \Omega_0$ is convex. Iterating this argument at all lifts of $E_+$ and $E_-$ shows that $\Omega$ is convex and tesselated by copies of $\Omega_0.$ This also implies that $\Omega / \Gamma \cong S.$
	
	Let $f$ be the unique homeomorphism extending $f_0$ to $S$. Then $(\Omega, \Gamma, f)$ is a convex projective structure of $S$, which maps onto $(\Omega_0, \Gamma_0, f_0)$ via $\Psi$. It follows that $\Psi$ surjects onto $\mathcal{V}$.
	
	To conclude the proof, it is sufficient to show that the fibers of $\Psi$ are $2$--dimensional discs. This follows from the fact that the centraliser of $ \hol_0(\gamma_\pm)$ in $\SL(3, \RR)$ is a $2$--dimensional. Since $A$ is conjugate to a diagonal matrix with positive eigenvalues, this implies there is a  $2$--disc worth of choices for $A$.
\end{proof}

In fact, more is known about the moduli space for a closed surface.
Denote $\X_3(S) = \Hom(\pi_1(S), \SL(3, \RR))\sslash \SL(3, \RR)$ the \emph{character variety} of $S.$ We have a well-defined map $\Tt(S) \to \X_3(S)$ and denote its image $\X(S).$ In the case of a closed surface, this is known to give a Zariski component:

\begin{thm}[Choi-Goldman 1997, \cite{Choi-classification-1997}]\label{thm:Goldman-Choi-main}
If $S$ is a closed, orientable surface of negative Euler characteristic, then $\Tt(S)$ is homeomorphic to a connected component of $\X_3(S).$
\end{thm}

\subsection{A Poisson Structure on the Moduli Space} \label{sec:poisson-structure}

This section is devoted to the construction and analysis of a Poisson structure on the moduli space $\Ttp(S_{g,n})$. We start with a brief summary of basic notions on symplectic structures and Poisson structures. In \S~\ref{subsec:goldman-symplectic-structure} we remind the reader about Goldman's results on character varieties, with particular emphasis on Goldman's formula (Theorem~ \ref{thm:Gol-formula}). Next, we shift our attention to the moduli space $\Ttp(S_{g,n})$, and explicitly construct a Poisson bracket via Theorem~\ref{thm:global-coord} and Fock-Goncharov moduli space. Finally, we concentrate on the relation between the Poisson structures on $\Ttp(S_{g,n})$ and on $\X_3(S_{g,n})$. The material in \S~\ref{subsec:compatibility_brackets} is devoted to developing the necessary tools to show that the structures are homomorphic.

\subsubsection{Symplectic structures and Poisson structures} \label{subsec:symplectic_bascis}

Over Riemannian manifolds, a smooth scalar function $f$ differentiates to a vector field $\nabla f$ via the metric. Symplectic manifolds are similarly rich in structure, they relate scalar functions to Hamiltonian vector fields.\\

A symplectic form $\omega$ on a smooth manifold $M$ is a closed differential $2$--form which is not degenerate in that the kernel of the operation of contracting vector fields is trivial. The pair $(M,\omega)$ is a \emph{symplectic manifold}. Because $\omega$ is non-degenerate, each smooth scalar function $f$ corresponds to a unique \emph{Hamiltonian vector field} $X_f$, the vector field satisfying
$$
\omega(X_f(x),v_x) = -D_xf(v_x) \qquad \mbox{for all } x \in M \mbox{ and } v_x \in T_xM.
$$
A \emph{Poisson manifold} is a smooth manifold $M$ endowed with a \emph{Poisson bracket}, $i.e.$ a Lie bracket 
$$
\{-,-\} : C^{\infty}(M) \times C^{\infty}(M)  \rightarrow C^{\infty}(M) 
$$
on the algebra $C^{\infty}(M)$ of smooth functions, satisfying the Leibniz Rule.

\begin{lem}
	Let $(M,\omega)$ be a symplectic manifold. For all $f,g \in C^{\infty}(M)$, the bracket $\{f,g\}_\omega := \omega( X_f, X_g)$ is a Poisson bracket.
\end{lem}

Symplectic manifolds are Poisson manifolds, but the converse is not true. A classical example of a Poisson manifold which may not give rise to a compatible symplectic structure is the linear Poisson structure on the dual of a Lie algebra. The fact that a symplectic form must be non-degenerate ensures that a symplectic manifold must have even dimension which clearly does not have to be the case for a Lie algebra.

If the Poisson bracket of a Poisson manifold $M$ restricts to a submanifold $N \subset M$, then $N$ is said to be a \emph{Poisson sub-manifold}. If moreover the restricted Poisson bracket on $N$ can be defined by a symplectic form, then $N$ is a \emph{symplectic submanifold}. Any Poisson manifold admits a \emph{symplectic foliation}, a possibly singular foliation whose leaves are the maximal symplectic submanifolds.


\subsubsection{Goldman's Poisson structure and formula} \label{subsec:goldman-symplectic-structure}

Let $S_{g,n}$ be a compact oriented surface with fundamental group $\pi_1$. Let $G$ be a Lie group preserving a non-degenerate bilinear form on its Lie algebra and $\X_G$ the  $G$--character variety $ \Hom(\pi_1,G) \sslash G$. We begin by recalling the following:

\begin{thm}[Goldman 1984, \cite{Goldman-symplectic-1984}] \label{thm:Goldman-symplectic}
	If $S_{g,n}$ is closed $(n=0)$, then $\X_G$ has a natural symplectic structure.
\end{thm}

Theorem \ref{thm:Goldman-symplectic} is valid for any reductive group $G$, in particular $\SL(m,\RR)$. Goldman's result was generalised to compact oriented surfaces with non-empty boundary in the following sense.

\begin{thm}[Guruprasad-Huebschmann-Jeffrey-Weinstein 1997, \cite{Guruprasad-group-1997}] \label{thm:GHJW}
	If $S_{g,n}$ has $n>0$ boundary components, then $\X_G$ has a natural Poisson structure.
\end{thm}

On each leaf of its symplectic foliation, the symplectic form is that defined by Goldman. For the purpose of describing this foliation, we fix a presentation of $\pi_1(S_{g,n})$ based at $p \in S_{g,n}$:
$$
\pi_1 := \pi_1(S_{g,n},p) = \langle x_1,y_1,\dots,x_g,y_g,b_1,\dots,b_n \ | \ \prod_{i=1}^g [x_i,y_i] \prod_{j=1}^n b_j =1  \rangle.
$$

For $n>0$, $\pi_1$ is a free group of rank $r = 2g + n - 1 = 1 - \chi(S_{g,n})$. The \emph{$i^{th}$ boundary map} $\mathbf{b_i} : [\rho] \mapsto [\rho_{| b_i}]$ sends the extended class of a representation to the class corresponding to the restriction of $\rho$ to the boundary loop $b_i$. The \emph{boundary map} 
$$
\mathbf{b} = (\mathbf{b_1},\dots,\mathbf{b_s}) : \X_G \rightarrow (G \sslash G)^{\times n}
$$
is a submersion on a Zariski-open subset $\X_G^{\mathbf{b}} \subset \X_G$ (cf. Hartshorne \cite{Hartshorne-algebraic-1977}, page 271), and $\X_G^{\mathbf{b}}$ is foliated by non-singular submanifolds $\mathbf{F}_y : = \X^{\mathbf{b}} \cap \mathbf{b}^{-1}(y)$. The leaves $\mathbf{F}_y$ are symplectic submanifolds (cf. Lawton \cite{Lawton-poisson-2009}), which extend to make the complement of the singular locus in $\X_G$ a Poisson manifold. Furthermore, the Poisson structure extends continuously over singularities in $\X_G$ (cf. Goldman \cite{Goldman-symplectic-1984}), making it into a \emph{Poisson variety}, an affine variety whose coordinate ring is a Poisson algebra.

As further motivation for investigating structures on the $G$--character variety, we recall an application of Theorem~\ref{thm:GHJW} to the special case $G = \SL(m,\RR)$, where the Poisson structure of $\X_{\SL(m,\RR)}$ distinguishes the topology of the surface, whereas the algebraic structure alone does not.

\begin{thm}[Lawton 2009, \cite{Lawton-poisson-2009}] \label{thm:poisson_distinguish_varieties}
	Let $S$ and $S'$ be compact connected orientable surfaces with boundary, and let $\X_n(S)$ and $\X_n(S')$ be their respective $\SL(m,\RR)$--character varieties. Then:
	\begin{itemize}
		\item[$(1)$]  $\X(S)$ and $\X(S')$ are isomorphic as varieties if and only if $\chi(S) = \chi(S')$;
		\item[$(2)$] $\X(S)$ and $\X(S')$ are isomorphic as Poisson varieties if and only if $ S \cong S'$.
	\end{itemize}
\end{thm}

When $G$ is a classic matrix group (for example $\GL$, $\SL$ or $\SU$), the Poisson structure of $\X_G$ has a deep geometric characterisation (see Goldman \cite{Goldman-symplectic-1984} and Lawton \cite{Lawton-poisson-2009}). Henceforth we restrict our attention to the case $G = \SL(m,\RR)$ which is of main interest here ($m = 3$), and abbreviate the notation of the $\SL(m,\RR)$--character variety to $\X_m$.

Let $\tr : \SL(m,\RR) \rightarrow \RR$ be the trace function. For all $\alpha \in \pi_1$, define 
\begin{align*}
\tr_\alpha \colon \X_m \rightarrow \RR \\
[\rho] \mapsto \tr(\rho(\alpha))
\end{align*}
The subalgebra of $C^\infty(\X_m)$ generated by functions of this form is the \emph{trace algebra} $\Tr(\X_m)$.

Let $\alpha$ and $\beta$ be immersed curves in $S$ representing elements in $\pi_1$ (in a general position). Denote the set of (transverse) double point intersections of $\alpha$ and $\beta$ by $\alpha \cap \beta$, and the oriented intersection number of $\alpha$ and $\beta$ at $p \in \alpha \cap \beta$ by $\epsilon(p,\alpha,\beta) \in \{-1,1\}$. Additionally, let $\alpha_p \in (\pi_1,p)$ represent the curve $\alpha$ based at $p$.

\begin{thm}[Goldman 1986 ~\cite{Goldman-invariant-1986}, Lawton 2009 ~\cite{Lawton-poisson-2009}] \label{thm:Gol-formula}
	Let $\{-,-\}_{\Gol}$ be the Poisson structure on $\X_m$. Then
	\begin{equation} \label{eq:Gol-formula}
	\{\tr_\alpha, \tr_\beta\}_{Gol} = \sum_{p \in \alpha \cap \beta} \epsilon(p,\alpha,\beta) (\tr_{\alpha_p \beta_p} - \frac{1}{n} \tr_\alpha \tr_\beta).	
	\end{equation}
	defines a Poisson bracket on $\Tr(\X_m)$.
\end{thm}

It is worth noticing that one may rewrite (\ref{eq:Gol-formula}) as
\begin{equation}
\{\tr_\alpha, \tr_\beta\}_{Gol} = \left(\sum_{p \in \alpha \cap \beta} \epsilon(p,\alpha,\beta) \tr_{\alpha_p \beta_p} \right) - \frac{\iota(\alpha,\beta)}{n} \tr_\alpha \tr_\beta,
\end{equation}
for the algebraic intersection number $\iota(\alpha,\beta)$. However we can not simplify it any further as $\tr_{\alpha_p \beta_p}$ heavily depends on $p$, as shown in Figure \ref{fig:curves_example}.
	
\begin{figure}[h]
	\centering
	\includegraphics[width=\textwidth]{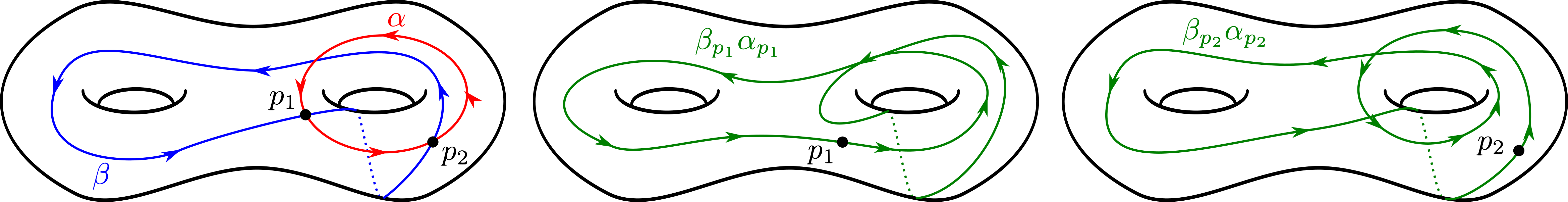}
	\captionof{figure}{$\alpha_{p_1} \beta_{p_1}$ and $\alpha_{p_2} \beta_{p_2}$ are very different curves, in particular they are not homotopic.}
	\label{fig:curves_example}
\end{figure}

It is known for $\SL(m,\CC)$ that the coordinate ring $\CC[\X_{\SL(m,\CC)}]$ of the $\SL(m,\CC)$--character variety is generated by trace functions, hence Theorem \ref{thm:Gol-formula} is true for $\Tr(\X_m) = \CC[\X_{\SL(m,\CC)}]$ (cf. Procesi~\cite{Procesi-invariant-1976}). However, it is still an open problem for $\SL(m,\RR)$.


\subsubsection{Fock and Goncharov's Poisson structure} \label{subsec:FG-poisson-structure}

We recall the setting and notation of \S~\ref{sec:Moduli-Spaces}. As usual, $\tri$ denotes a fixed ideal triangulation of a surface $S = S_{g,n}$ ($2g + n > 2$) and $\triangle,\underline{E}$ the sets of its triangles and oriented edges respectively. To frame our subsequent definition, we will make the following assumption:
\begin{itemize}
	\item[(I)] \emph{every edge in $\tri$ is contained in two distinct triangles, or equivalently every triangle has three distinct edges}
\end{itemize}
Such a triangulation always exists, and can be constructed algorithmically from any given triangulation. When $t$ is a triangle with two edges identified, we perform an edge flip on the third edge of $t$. We underline that this assumption is not necessary, but its inclusion avoids a lot of technicalities.

The Fock--Goncharov moduli space is the set of functions $\RR^{\triangle \cup \underline{E}}_{>0} := \{ \triangle \cup \underline{E} \rightarrow \RR_{>0}\}$, and it is isomorphic to $\Ttp(S)$ by Theorem \ref{thm:global-coord}.

For $q \in \triangle \cup \underline{E}$, the \emph{coordinate function of $q$} is the map
\begin{align*}
q^* : \RR^{ \triangle \cup \underline{E}}_{>0} &\rightarrow \RR_{>0}\\
v &\mapsto v(q)
\end{align*}
and the \emph{$q$--coordinate} is the element $\overline{q} \in \RR^{ \triangle \cup \underline{E}}_{>0}$, such that
$$
\overline{q}(q') = 
\begin{cases}
1 &\mbox{ if } q = q',\\
0 & \mbox{ otherwise }.
\end{cases}
$$

It is clear from the definitions that $\frac{\partial q_i^*}{\partial \overline{q_j} } = \delta_{ij}$. For all $f,g \in C^{\infty}(\RR^{\triangle \cup \underline{E}}_{>0})$, we define
\begin{equation} \label{eq:triangulation-bracket}
\{ f,g \}_{\tri} := \sum_{q_1,q_2 \in \triangle \cup \underline{E} } 2 \epsilon(q_1,q_2) q_1^* q_2^* \frac{\partial f}{\partial \overline{q}_1} \frac{\partial g}{\partial \overline{q}_2},
\end{equation}
where $\epsilon : \triangle \cup \underline{E} \times \triangle \cup \underline{E} \rightarrow \{-1,0,1\}$ is the skew-symmetric integral valued function defined as follows. Given $e \in \underline{E}$, we will say that
\begin{itemize}
	\item $e$ and $t \in \triangle$ are \emph{in special position} if $e \subset t$ and $e,t$ are oriented coherently (where $t$ has orientation induced from $S$).
	\item $e$ and $e' \in \underline{E}$ are \emph{in special position with respect to $e$} if
	\begin{itemize}
		\item There is a unique $t \in \triangle$ such that $e,e' \subset t$ and $e,e'$ share the same starting endpoint;
		\item $e$ and $t$ are in special position.
	\end{itemize}
\end{itemize}
Hence we define:
\begin{align*}
\epsilon(e_1,e_2) &= 
\begin{cases}
1 & \mbox{ if } e_1,e_2 \mbox{ are in special position with respect to } e_1;\\
-1 & \mbox{ if } e_1,e_2 \mbox{ are in special position with respect to } e_2;\\
0 & \mbox{ otherwise }.
\end{cases}\\
\epsilon(e,t) &= 
\begin{cases}
1 & \mbox{ if } e^{-1},t \mbox{ are in special position};\\
-1 & \mbox{ if } e,t \mbox{ are in special position};\\
0 & \mbox{ otherwise }.
\end{cases}\\
\epsilon(t,e) &= - \epsilon(e,t) \qquad \mbox{ and } \qquad  \epsilon(t_1,t_2) = 0,
\end{align*}
for all $e,e_1,e_2 \in \underline{E}$ and $t,t_1,t_2 \in \triangle$. We remark that this is the only place where we make use of assumption $(I)$. As mentioned before, one does not need such condition, but disposing of it makes the above definitions much more involved.

\begin{rem}
	The function $\epsilon$ can also be understood graphically by considering the diagram shown in Figure~\ref{epsilon_diagram}. Here we symbolically represent a triangle by a point in its interior and an oriented edge by a point on the edge, closer to the starting endpoint. Then $\epsilon(q_1,q_2)$ is the difference between the number of arrows from $q_1$ to $q_2$ and the number of arrows from $q_2$ to $q_1$.
\end{rem}

\begin{figure}[ht]
	\centering
	\includegraphics[height=5cm]{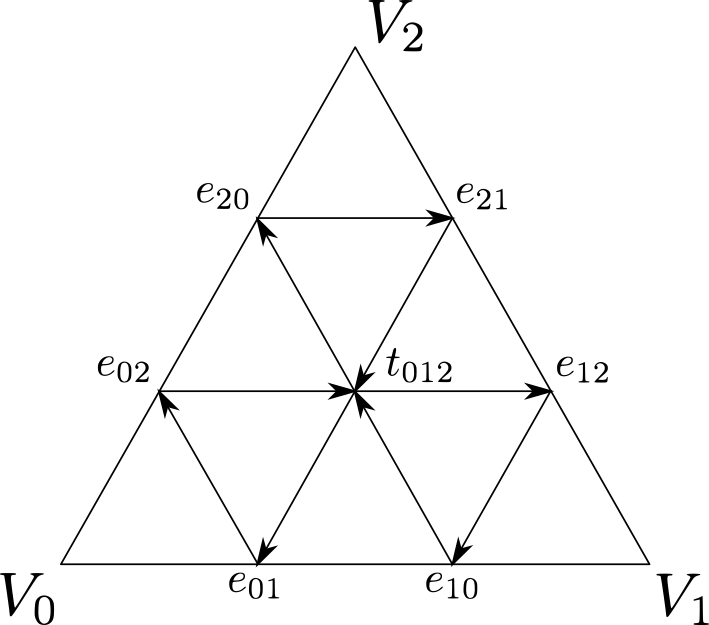}
	\caption{$\epsilon$ can be read off of a graph. Here we have shown edges connecting marked points belonging to one triangle only. Points of other triangles are connected by arrows in the same ways. Points connected by edges without arrows are not taken into account.}
	\label{epsilon_diagram}
\end{figure}

\begin{pro}
	$\{ -,- \}_{\tri}$ is a Poisson bracket.
\end{pro}

\begin{proof}
	$\RR$--bilinearity, skew-symmetry and Leibniz's product rule are simple to check. We only show that $\{ -,- \}_{\tri}$ satisfies the Jacobi identity, that is
	$$
	\{f,\{g,h\}_{\tri} \}_{\tri} + \{g,\{h,f\}_{\tri} \}_{\tri} + \{h,\{f,g\}_{\tri} \}_{\tri} = 0,
	$$
	 for all $f,g,h \in C^{\infty}(\RR^{\triangle \cup \underline{E}}_{>0})$. To make the computations more readable, henceforth we write $\Bf,\Bg,\Bh$ and simplify the notation as follows
\begin{align*}
\{ \Bf,\Bg \}_{\tri} :=& \sum_{q_1,q_2 \in \triangle \cup \underline{E} } 2 \epsilon(q_1,q_2) q_1^* q_2^* \frac{\partial \Bf}{\partial \overline{q}_1} \frac{\partial \Bg}{\partial \overline{q}_2} \\
=& \sum_{q_i,q_j \in \triangle \cup \underline{E} } \left[2 \epsilon_{ij} q_i^* q_j^*\right] \frac{\partial \Bf}{\partial \overline{q}_i} \frac{\partial \Bg}{\partial \overline{q}_j} \\
=& \;\sum_{i,j} \omega_{ij} \partial_i \Bf \partial_j \Bg.
\end{align*}
Observe that $\omega_{ij} = - \omega_{ji}$. It follows that
\begin{align*}
\{\Bf,\{\Bg,\Bh\}_{\tri} \}_{\tri} &= \sum_{i,j,k,l} \omega_{kl} \partial_k \Bf \left( \partial_l \omega_{ij} \partial_i \Bg \partial_j \Bh + \omega_{ij} \partial_l \partial_i \Bg \partial_j \Bh + \omega_{ij} \partial_i \Bg \partial_l \partial_j \Bh \right),\\
\{\Bg,\{\Bh,\Bf\}_{\tri} \}_{\tri} &= \sum_{i,j,k,l} \omega_{kl} \partial_k \Bg \left( \partial_l \omega_{ij} \partial_i \Bh \partial_j \Bf + \omega_{ij} \partial_l \partial_i \Bh \partial_j \Bf + \omega_{ij} \partial_i \Bh \partial_l \partial_j \Bf \right),\\
\{\Bh,\{\Bf,\Bg\}_{\tri} \}_{\tri} &= \sum_{i,j,k,l} \omega_{kl} \partial_k \Bh \left( \partial_l \omega_{ij} \partial_i \Bf \partial_j \Bg + \omega_{ij} \partial_l \partial_i \Bf \partial_j \Bg + \omega_{ij} \partial_i \Bf \partial_l \partial_j \Bg \right).
\end{align*}
We factorise in terms of the partial derivatives of $\Bf,\Bg,\Bh$ to get the coefficients
\begin{align}
\partial_k \Bf \partial_i \Bg \partial_j \Bh &: \qquad \sum_l \left(\omega_{kl} \partial_l \omega_{ij} + \omega_{il} \partial_l \omega_{jk} + \omega_{jl} \partial_l \omega_{ki} \right), \hspace{2.5cm} \forall i,j,k, \label{eq:Jacobi-id-1} \\
\partial_k \Bf \partial_l \partial_i \Bg \partial_j \Bh &: \qquad \omega_{kl} \omega_{ij} + \omega_{jl} \omega_{ki}, \hspace{6cm} \forall i,j,k,l,\label{eq:Jacobi-id-2} \\
\partial_k \Bf \partial_i \Bg \partial_l \partial_j \Bh &: \qquad \omega_{kl} \omega_{ij} + \omega_{il} \omega_{jk}, \hspace{6cm} \forall i,j,k,l,\label{eq:Jacobi-id-3} \\
\partial_k \Bg \partial_i \Bh \partial_l \partial_j \Bf &: \qquad \omega_{kl} \omega_{ij} + \omega_{il} \omega_{jk}, \hspace{6cm} \forall i,j,k,l \label{eq:Jacobi-id-4}.
\end{align}
When $i = l$, the coefficient in $(\ref{eq:Jacobi-id-2})$ is $\omega_{ki} \omega_{ij} + \omega_{ji} \omega_{ki} = 0$.\\
When $i \not= l$, $\partial_k \Bf \partial_i \partial_l \Bg \partial_j \Bh = \partial_k \Bf \partial_l \partial_i \Bg \partial_j \Bh$ and we can sum their coefficients to get
$$
\omega_{kl} \omega_{ij} + \omega_{jl} \omega_{ki} + \omega_{ki} \omega_{lj} + \omega_{ji} \omega_{kl} = \omega_{kl}(\omega_{ij} + \omega_{ji}) + \omega_{ki}(\omega_{jl} + \omega_{lj}) = 0.
$$
A similar argument applies to $(\ref{eq:Jacobi-id-3})$ and $(\ref{eq:Jacobi-id-4})$. Our claim follows from showing that the coefficients in $(\ref{eq:Jacobi-id-1})$ are also zero. Since $\partial_l \omega_{ij} = 2 \epsilon(q_i,q_j)\partial_l(q_i^* q_j^*) = 2 \epsilon_{ij} \left( \delta_{il} q_j^* + \delta_{jl} q_i^* \right)$, by substituting in $(\ref{eq:Jacobi-id-1})$, we conclude that
\begin{align*}
&\sum_l \left(\omega_{kl} \partial_l \omega_{ij} + \omega_{il} \partial_l \omega_{jk} + \omega_{jl} \partial_l \omega_{ki} \right)\\
=&\sum_l 2 q_l^* \left[ \delta_{jl} q_i^*q_k^* \left( \epsilon_{kl}\epsilon_{ij} + \epsilon_{il} \epsilon_{jk} \right) + \delta_{il} q_j^*q_k^* \left( \epsilon_{kl}\epsilon_{ij} + \epsilon_{jl} \epsilon_{ki} \right) + \delta_{kl} q_i^*q_j^* \left( \epsilon_{il}\epsilon_{jk} + \epsilon_{jl} \epsilon_{ki} \right) \right]\\
=&\sum_l 2 q_l^* \left[ q_i^*q_k^* \left( \epsilon_{kj}\epsilon_{ij} + \epsilon_{ij} \epsilon_{jk} \right) + q_j^*q_k^* \left( \epsilon_{ki}\epsilon_{ij} + \epsilon_{ji} \epsilon_{ki} \right) + q_i^*q_j^* \left( \epsilon_{ik}\epsilon_{jk} + \epsilon_{jk} \epsilon_{ki} \right) \right]= 0.
\end{align*}
\end{proof}

The next result shows that $\{ -,- \}_{\tri}$ does not depend on $\tri$.

\begin{thm}\label{thm:bracket-tri-independent}
	If $\tri$ and $\tri'$ are ideal triangulations of $S$, then $(\RR^{\triangle \cup \underline{E}}_{>0},\{-,-\}_\tri)$ and $(\RR^{\triangle' \cup \underline{E}'}_{>0},\{-,-\}_{\tri'})$ are naturally isomorphic Poisson manifolds.
\end{thm}

\begin{proof}
	Let $\phi_e : \RR^{\triangle \cup \underline{E}}_{>0} \rightarrow \RR^{\triangle' \cup \underline{E}'}_{>0}$ be the change of coordinates induced by flipping along an edge $e$ of $\tri$ (see \S~\ref{subsec:change-of-coord} for details). Since $\phi_e$ is a bijection and any two ideal triangulations of $S$ differ by a finite sequence of edge flips, it is enough to prove that for all $f,g \in C^{\infty}(\RR^{\triangle' \cup \underline{E}'}_{>0})$,
	$$
	\{f,g\}_{\tri'} \circ \phi_e = \{ f \circ \phi_e, g \circ \phi_e \}_{\tri},
	$$
	or equivalently
	\begin{equation} \label{eq:bracket_tri_indep}
	\left( \sum_{q_1',q_2' \in \triangle' \cup \underline{E}' } \epsilon(q_1',q_2') q_1'^* q_2'^* \frac{\partial f}{\partial \overline{q}_1'} \frac{\partial g}{\partial \overline{q}_2'} \right) \circ \phi_e = \sum_{q_1,q_2 \in \triangle \cup \underline{E} } \epsilon(q_1,q_2) q_1^* q_2^* \frac{\partial f \circ \phi_e}{\partial \overline{q}_1} \frac{\partial g \circ \phi_e}{\partial \overline{q}_2}.
	\end{equation}
	Denote by $(\phi_e)_i := q_i'^* \circ \phi_e$ the $q_i'^*$ component of $\phi_e$. Then
	$$
	\text{LHS} = \sum_{q_1',q_2' \in \triangle' \cup \underline{E}' } \epsilon(q_1',q_2') (\phi_e)_1 (\phi_e)_2 \left(\frac{\partial f}{\partial \overline{q}_1'}\circ \phi_e\right) \left( \frac{\partial g}{\partial \overline{q}_2'} \circ \phi_e\right),
	$$
	\begin{align*}
	\text{RHS} &= \sum_{q_1,q_2 \in \triangle \cup \underline{E} } \epsilon(q_1,q_2) q_1^* q_2^* \left[  \sum_{q_1',q_2' \in \triangle' \cup \underline{E}' } \left(\frac{\partial f}{\partial \overline{q}_1'} \circ \phi_e \frac{ \partial (\phi_e)_1}{\partial \overline{q}_1}\right)\left(\frac{\partial g}{\partial \overline{q}_2'} \circ \phi_e \frac{ \partial(\phi_e)_2}{ \partial \overline{q}_2}\right) \right] \\
	&= \sum_{q_1',q_2' \in \triangle' \cup \underline{E}' } \sum_{q_1,q_2 \in \triangle \cup \underline{E} } \epsilon(q_1,q_2) q_1^* q_2^* \left(\frac{\partial f}{\partial \overline{q}_1'} \circ \phi_e \frac{ \partial (\phi_e)_1}{\partial \overline{q}_1}\right)\left(\frac{\partial g}{\partial \overline{q}_2'} \circ \phi_e \frac{ \partial (\phi_e)_2}{\partial \overline{q}_2}\right).
	\end{align*}
	Hence we rewrite equation (\ref{eq:bracket_tri_indep}) as
	$$
	\sum_{q_1',q_2' \in \triangle' \cup \underline{E}' } \left(\frac{\partial f}{\partial \overline{q}_1'} \circ \phi_e \right)\left(\frac{\partial g}{\partial \overline{q}_2'} \circ \phi_e\right) \left[ \sum_{q_1,q_2 \in \triangle \cup \underline{E} } \epsilon(q_1,q_2) q_1^* q_2^* \frac{ \partial (\phi_e)_1}{ \partial \overline{q}_1} \frac{ \partial (\phi_e)_2}{ \partial \overline{q}_2} - \epsilon(q_1',q_2') (\phi_e)_1 (\phi_e)_2 \right] = 0
	$$
	By differentiating the functions in section \ref{subsec:change-of-coord}, one can easily check that for all $i,j$
	$$
	\sum_{q_1,q_2 \in \triangle \cup \underline{E} } \epsilon(q_1,q_2) q_1^* q_2^* \frac{ \partial (\phi_e)_1}{ \partial \overline{q}_1} \frac{ \partial (\phi_e)_2}{ \partial \overline{q}_2} - \epsilon(q_i',q_j') (\phi_e)_i (\phi_e)_j = 0.
	$$
\end{proof}

Let $\eta$ be the change of coordinates induced by the sequence of edge flips between $\tri$ and $\tri'$. The following diagram commutes
\begin{center}
\begin{tikzpicture}
  \matrix (m) [matrix of math nodes,row sep=3em,column sep=4em,minimum width=2em] { 
\Ttp(S)  &    \RR^{\triangle \cup \underline{E}}_{>0} \\
 	 &  \RR^{\triangle' \cup \underline{E}'}_{>0} \\ };
  \path[-stealth]
    (m-1-1) edge node [above] {$\phi_\tri$} (m-1-2)
    (m-1-2) edge node [right] {$\eta$} (m-2-2)
    (m-1-1) edge node [below]{$\phi_{\tri'}$} (m-2-2);
\end{tikzpicture}
\end{center}
and the pull-back of $\{-,-\}_\tri$ via $\phi_\tri$ does not depend on the coordinate system. Therefore for all $f,g \in C^{\infty}(\Ttp(S) )$ and any ideal triangulation $\tri$ of $S$, we define
$$
\{f,g\}_{FG} : = \{f \circ \phi^{-1}_\tri, g \circ \phi^{-1}_\tri \}_{\tri},
$$
which makes $(\Ttp(S),\{-,-\}_{FG})$ a Poisson manifold.


\subsubsection{Compatibility of Poisson brackets} \label{subsec:compatibility_brackets}

Let $\Tr_3 := \Tr(\X_3(S)) \subset C^{\infty}(\X_3(S))$ be the trace algebra with respect to $\SL(3,\RR)$ and let 
\begin{align*}
\mu^* \colon C^{\infty}(\X_3(S)) &\rightarrow C^{\infty}(\Ttp(S))\\
f &\mapsto f \circ \mu
\end{align*}
be the injective homomorphism induced by the monodromy operator $\mu \colon \Ttp(S) \rightarrow \X_3(S)$. This section is devoted to the proof of the following result.

\begin{thm} \label{thm:compatibility-poisson-bracket}
	The restriction of $\mu^*$ to $\Tr_3$ is a Poisson homomorphism, $i.e.$ for all $f,g \in \Tr_3$,
	\begin{equation} \label{eq:compatibility-poisson-bracket}
		2 \ \mu^*  \circ \{f,g \}_{Gol} = \{\mu^* \circ f,\mu^* \circ g \}_{FG}.	
	\end{equation}
\end{thm}

The strategy we adopt consists of the following $5$ steps. First we apply Theorem~\ref{thm:global-coord} to reword the problem in terms of $\{-,-\}_{\tri}$ for a fixed triangulation $\tri$ of $S$. That allows us to compare the definition of $\{-,-\}_{\tri}$ in (\ref{eq:triangulation-bracket}) with the formulation of $\{-,-\}_{Gol}$ in (\ref{eq:Gol-formula}). Next, we break up $\alpha$ and $\beta$ into pieces and focus our attention on \emph{tunnels and galleries}, where curves travel next to each other (see \S~\ref{subsubsec:proof-part2} for details). We then show in \S~\ref{subsubsec:proof-part3} how the global quantity $\{ \tr_\alpha \circ \mu, \tr_\beta  \circ \mu\}_{\tri}$ translates into a local behaviour $\{ \tr_\alpha, \tr_\beta \}_{Gol}$ at the intersection points of $\alpha$ and $\beta$. A special case when $\alpha$ and $\beta$ are \emph{parallel} is treated in \S~\ref{subsubsec:proof-part4}. All of the pieces are put together to conclude the proof in \S~\ref{subsubsec:proof-part5}.

\subsubsection{Proof part 1: set up} \label{subsubsec:proof-part1}

Fix an ideal triangulation $\tri$ of $S$ and the corresponding isomorphism $\phi_{\tri} : \RR^{\triangle \cup \underline{E}}_{>0} \rightarrow \Ttp(S)$. Let $(\mu \circ \phi_{\tri})^* \colon C^{\infty}(\X_3(S)) \rightarrow C^{\infty}(\RR^{\triangle \cup \underline{E}}_{>0} )$ be the map $(\mu \circ \phi_{\tri})^*(f) = f \circ \mu \circ \phi_{\tri}$, then $\mu^*$ is a Poisson homomorphism on $\Tr$ if and only if $(\mu \circ \phi_{\tri})^*$ is a Poisson homomorphism on $\Tr$. Henceforth we identify $\Ttp(S)$ with $\RR^{\triangle \cup \underline{E}}_{>0}$ and denote by $\mu,\mu^*$ the maps $\mu \circ \phi_{\tri}, (\mu \circ \phi_{\tri})^*$.

Recall that $\Tr_3$ is the sub-algebra of $C^\infty(\X_3(S))$ generated by elements of the form $\tr_\alpha$ for $\alpha \in \pi_1(S)$. By Theorem~\ref{thm:Gol-formula} and the defining properties of Poisson brackets, it is enough to show that for all representatives $\alpha,\beta \subset S$ of immersed oriented closed curves in general position,
\begin{equation} \label{eq:poisson-equality}
\sum_{p \in \alpha \cap \beta} \epsilon(p,\alpha,\beta) \left( \tr_{\alpha_p \beta_p}\circ \mu - \frac{1}{3} (\tr_\alpha \circ \mu )(\tr_\beta \circ \mu ) \right) = \sum_{q_1,q_2 \in \triangle \cup \underline{E} } \epsilon(q_1,q_2) q_1^* q_2^* \frac{\partial (\tr_\alpha \circ \mu)}{\partial \overline{q}_1} \frac{\partial (\tr_\beta \circ \mu)}{\partial \overline{q}_2}.
\end{equation}

Here $\alpha$ and $\beta$ are assumed to be in general position both with respect to each other and with respect to $\tri$.

We start by recalling some notation from \S~\ref{subsect_monodromy_operator}. For $\gamma \in \pi_1(S)$, we showed in Lemma~\ref{lem_compute_monodromy} that there exists a finite sequence $(e^\gamma_0 \cdot  \dots \cdot e^\gamma_n)$ of $n+1$ oriented edges $e^\gamma_i \in \underline{E}^\graph$ such that $e^\gamma_0 \cdot e^\gamma_{1} \cdot \dots \cdot e^\gamma_n \subset \graph$ is the development of a loop in $\graph_0$ homotopic equivalent to $\gamma$. Let $(s^\gamma_0 \cdot  \dots \cdot s^\gamma_n)$ be such a loop, where $s^\gamma_i \in \underline{E}^{\graph_0}$ corresponds to $e^\gamma_i \in \underline{E}^\graph$. Then we showed in Theorem~\ref{thm:monodromy} that $\mu(x)(\gamma) $ is the conjugacy class of a product $S^\gamma_0\cdot S^\gamma_1 \cdots S^\gamma_n$, where $S^\gamma_i$'s are matrices of the type $T,T^{-1}$ or $E$. More precisely, $S^\gamma_i$ is of type $T,T^{-1}$ (respectively type $E$) if and only if $s^\gamma_i$ is a $\triangle$--edge (resp. $\underline{E}$--edge), as described at the beginning of \S~\ref{subsect_monodromy_operator}. Furthermore, we can assume that $s^\gamma_{i}$ and $s^\gamma_{i+1}$ are of different type for all $i \in \ZZ_{n+1}$.

For the remainder of this chapter we will work with fixed sequences $e_i^\alpha,e_i^\beta$, loops $s_i^\alpha,s_i^\beta$ and matrices $S_i^\alpha,S_i^\beta \in \SL(3,\RR)$ for $\alpha$ and $\beta$.

For convenience, we provide here a simple result which will be used later on.

\begin{lem} \label{lemma_basic_derivatives}
	Let $T(z),E(x,y) \in \SL(3,\RR)$ be the matrices defined in Lemma~\ref{lem_example_monodromy}, and let
	$$
	M_1 := \left(
	\begin{array}{ccc}
	-2 & 0 & 0\\
	0 & 1 & 0\\
	0 & 0 & 1\\
	\end{array} \right),
	M_2 := \left(
	\begin{array}{ccc}
	-1 & 0 & 0\\
	0 & -1 & 0\\
	0 & 0 & 2\\
	\end{array} \right),
	M_3 := \left(
	\begin{array}{ccc}
	2 & 3 & 0\\
	0 & -1 & 0\\
	0 & 0 & -1\\
	\end{array} \right),
	M_4 := \left(
	\begin{array}{ccc}
	1 & 0 & 0\\
	0 & 1 & 0\\
	0 & -3 & -2\\
	\end{array} \right).
	$$
	Then
	$$
	\frac{\partial E(x,y)}{\partial x} = \frac{1}{3x}E(x,y) \cdot M_1, \qquad \qquad
	\frac{\partial E(x,y)}{\partial y} = \frac{1}{3y}E(x,y) \cdot M_2, 
	$$
	$$
	\frac{\partial T(z)}{\partial z} = \frac{1}{3z}T(z) \cdot M_3, \qquad \qquad
	\frac{\partial T(z)^{-1}}{\partial z} = \frac{1}{3z}T(z)^{-1} \cdot M_4.
	$$
\end{lem}

\subsubsection{Proof part 2: tunnels and galleries} \label{subsubsec:proof-part2}

The function $\{f,g\}_{FG}$ does not depend on the chosen triangulation $\tri$ so we may assume, without loss of generality, that:
\begin{itemize}
	\item[(I)] \emph{every edge in $\tri$ is contained in two distinct triangles, or equivalently, every triangle has three distinct edges.}
\end{itemize}
As before, we underline that this assumption is not necessary and we will specify where we use it.

Denote by $\partial_i f := \frac{\partial f }{\partial \overline{q}_i}$, then by the product rule
$$
\partial_i (\tr_\gamma \circ \mu) = \partial_i (\tr \mu(\gamma)) =  \tr (\partial_i \mu(\gamma)) = \sum_j \tr \left(S^\gamma_0 \cdots \partial_i S^\gamma_j \cdots S^\gamma_n\right).
$$
Denote $\partial_i s^\gamma_j := S^\gamma_0 \cdots \partial_i S^\gamma_j \cdots S^\gamma_n$ for all $i,j$. We remind the reader that $s^\gamma_j$ is a segment of a loop homotopic to $\gamma$, so the notation $\partial_i s^\gamma_j$ is chosen to emphasize that the partial derivative is ``taken at $s^\gamma_j$". 

It follows that
$$
\partial_1 (\tr_\alpha \circ \mu) \cdot \partial_2 (\tr_\beta \circ \mu) = \sum_{i,j} \tr (\partial_1 s^\alpha_i) \cdot \tr (\partial_2 s^\beta_j),
$$
and the right hand side of (\ref{eq:poisson-equality}) can be reformulated as
$$
\sum_{i,j} \sum_{q_1,q_2 \in \triangle \cup \underline{E} } \epsilon(q_1,q_2) q_1^* q_2^* \left(\tr (\partial_1 s^\alpha_i) \cdot \tr (\partial_2 s^\beta_j)\right).
$$

\begin{rem} \label{rem_segment_derivative}
	Lemma \ref{lemma_basic_derivatives} implies that $\partial_i s^\gamma_j = \frac{1}{3 q_i^* } \left(S^\gamma_0 \cdots S^\gamma_j \cdot X \cdots S^\gamma_n \right)$, where $X$ is:
	\begin{itemize}
		\item[$i)$] $M_1$ or $M_2$, if $s^\gamma_j$ is an $\underline{E}$--edge crossing an oriented edge of $\tri$ with edge ratio $q_i$;
		\item[$ii)$] $M_3$ or $M_4$, if $s^\gamma_j$ is a $\triangle$--edge contained in a triangle of $\tri$ with triple ratio $q_i$;
		\item[$iii)$] the zero matrix otherwise.
	\end{itemize}
	
\end{rem}

We are now going to investigate the set of pairs $(s^\alpha_i,s^\beta_j)$ and their contribution
$$
C(s_i^{\alpha},s_j^{\beta}) := \sum_{q_1,q_2 \in \triangle \cup \underline{E} } \epsilon(q_1,q_2) q_1^* q_2^* \left(\tr (\partial_1 s^\alpha_i) \cdot \tr (\partial_2 s^\beta_j)\right)
$$
in the above sum.\\

%

Clearly $C(s_i^{\alpha},s_j^{\beta}) = 0$ when $s^\alpha_i $ and $ s^\beta_j$ do not intersect a common triangle. If $\alpha$ and $\beta$ never cross a common triangle, then they must be disjoint, hence both right hand side and left hand side in (\ref{eq:poisson-equality}) are $0$.\\

For that reason, we assume from now on that
\begin{itemize}
	\item[(II)] \emph{there exist segments $s^\alpha_i $ and $ s^\beta_j$ intersecting a common triangle $\hat{T}$.}
\end{itemize}

When $t$ is a common triangle to two segments $s^\alpha_i $ and $ s^\beta_j$ it follows from assumption (I) that there is a unique (oriented) connected component $t(s^\alpha_i)$ (resp. $t(s^\beta_j)$) of $t \cap \alpha$ (resp. $t \cap \beta$) intersecting $s^\alpha_i$ (resp. $s^\beta_j$). We will say that $t$ has the \emph{configuration type} of:
\begin{itemize}
	\item a \emph{starting triangle} (resp. \emph{ending} triangle) with respect to $s^\alpha_i $ and $ s^\beta_j$, if $t(s^\alpha_i)$ and $t(s^\beta_j)$ cross a unique common edge of $t$ and $t(s^\alpha_i)$ leaves (resp. enters) $t$ from such edge;
	\item a \emph{transition triangle} with respect to $s^\alpha_i $ and $ s^\beta_j$, if $t(s^\alpha_i)$ and $t(s^\beta_j)$ cross the same two edges of $t$.
\end{itemize}

Up to interchanging $\alpha$ and $\beta$, there are only $7$ possible configurations. They are listed in Figure~\ref{triangle_config_types} and call $jS$ the configuration of type $j$ after interchanging.

\begin{figure}[h]
	\centering
	\includegraphics[height=5cm]{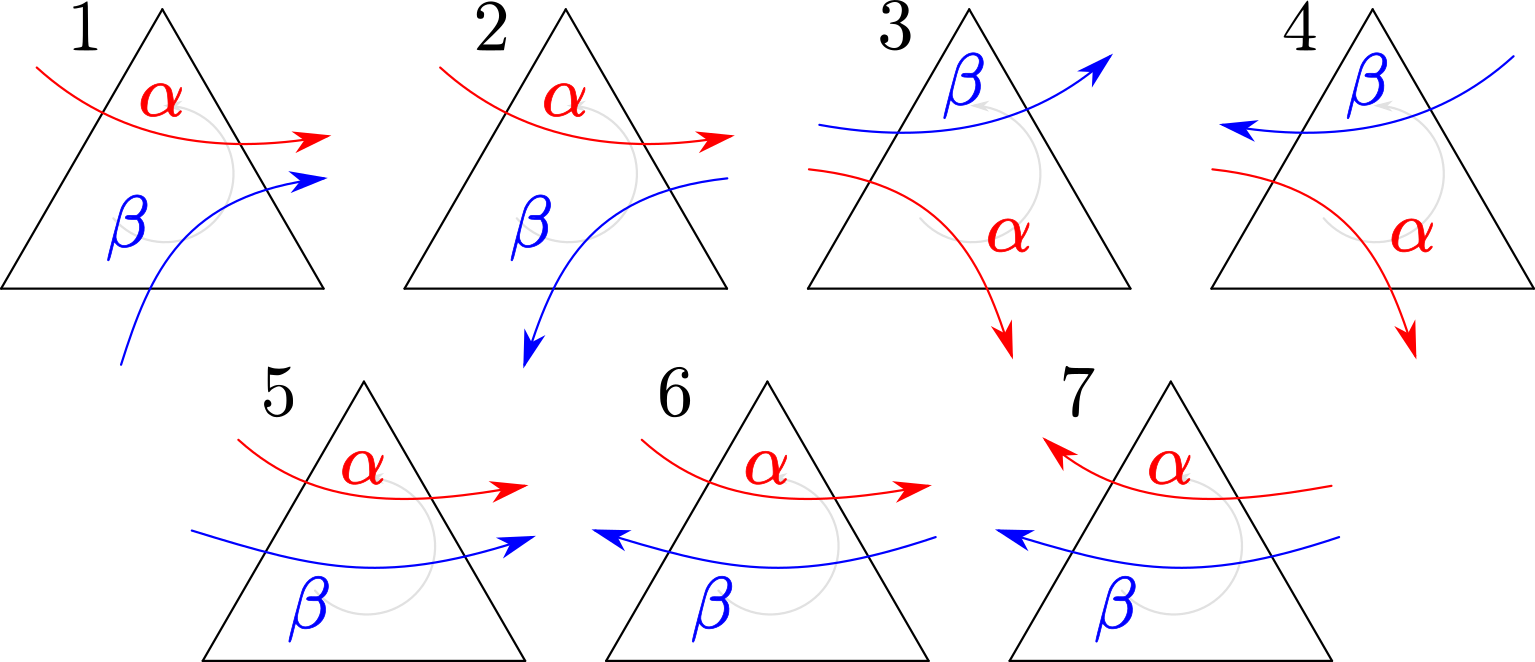}
	\caption{Triangle configuration types.}
	\label{triangle_config_types}
\end{figure}

Furthermore, we will say that $w_\gamma = (w^\gamma_0  \cdots w^\gamma_k)$ is a \emph{walk along $(s^\gamma_0 \cdots s^\gamma_n)$} of length $k+1 \in \NN$ based at $s^\gamma_{j}$ if either 
$$
w^\gamma_i = s^\gamma_{[j+i]_{n+1}} \qquad \mbox{ or } \qquad w^\gamma_i = s^\gamma_{[j-i]_{n+1}},
$$
where $[m]_{n+1} := m \mod (n+1)$. When $\gamma$ is represented by a loop $(s^\gamma_0 \cdots s^\gamma_n)$, we will abbreviate the notation by saying that $w_\gamma$ is a \emph{walk along $\gamma$}.

\begin{rem}
	Broadly speaking $w_\gamma$ is a walk along $(s^\gamma_0 \cdots s^\gamma_n)$ starting at $s^\gamma_{j}$ which might wrap around the curve multiple times, possibly backwards with respect to its orientation. In particular, $(s^\gamma_0 \cdot s^\gamma_1 \cdots s^\gamma_n)$ and $(s^\gamma_0 \cdot s^\gamma_n \cdots s^\gamma_1)$ are walks of length $n+1$ based at $s^\gamma_{0}$ in different directions.
\end{rem}

For the moment, we will make the extra assumption that
\begin{itemize}
	\item[$(\star)$] \emph{$\alpha$ and $\beta$ are \emph{non-parallel}}
\end{itemize}
that is $[\alpha] \not= [\beta^m]$ for all $m \in \ZZ$. We come back to this particular case later in \S~\ref{subsubsec:proof-part4}.\\

Since $\alpha$ and $\beta$ are closed non-parallel curves, if $t$ is a common triangle to two segments $s^\alpha_i $ and $ s^\beta_j$ then there exist:
\begin{itemize}
	\item a sequence of adjacent (not necessarily distinct) triangles $(t_0,\dots,t_q) \subset \tri$ containing $t$;
	\item a walk $w_{\alpha} := (w^\alpha_0 \cdot  \dots \cdot w^\alpha_{2q+2})$ along $\alpha$ containing $s^\alpha_i $;
	\item  a walk $w_{\beta} := (w^\beta_0 \cdot  \dots \cdot w^\beta_{2q+2})$ along $\beta$ containing $s^\beta_j $;
\end{itemize}
such that $t_0$ is a starting triangle with respect to the pairs $(w^\alpha_i,w^\beta_i)_{0\leq i \leq 2}$, $t_q$ is an ending triangle with respect to the pairs $(w^\alpha_i,w^\beta_i)_{2q \leq i \leq 2q+2}$, and $t_j$ is a transition triangle with respect to the pairs $(w^\alpha_i,w^\beta_i)_{2j \leq i \leq 2j+2}$ for all $1\leq j \leq q-1$.\\

In this case we will say that $(w_{\alpha}, w_{\beta})$ is a \emph{tunnel through} $(t_0,\dots,t_q)$ \emph{containing} $(s^\alpha_i ,s^\beta_j)$. Furthermore, for all $k \in \{0 \dots q \}$ there is a preferred triple of pairs $(w^\alpha_i,w^\beta_i)_{2k \leq i \leq 2k+2}$ and a triangle $t_k$ associated to it, which will be called the \emph{$k^{th}$ gallery} of the tunnel. By definition every $j^{th}$ and $(j+1)^{th}$ gallery share a common pair $(w^\alpha_{2j+2},w^\beta_{2j+2})$.

\begin{rem}
	A good way to visualise a tunnel is to pass to a universal cover of the surface. If $(w_{\alpha}, w_{\beta})$ is a tunnel through $(t_0,\dots,t_q)$ and we lift $\alpha$ and $\beta$ to the universal cover $\widetilde{S}$ based at a some point $p$ in any of the triangles $t_i$, then the two curves in $\widetilde{S}$ travel through common triangles only along the tunnel while  apart from each other everywhere else.
\end{rem}

Henceforth, when we talk about a tunnel containing the pair $(s^\alpha_i, s^\beta_j)$, we will further assume that $s^\alpha_i$ and $ s^\beta_j$ are contained in the same gallery. In this way, for every pair intersecting a common triangle there is a unique tunnel containing them. We stress again that the assumption $(\star)$ is necessary for the existence of at least one tunnel, as the reader can be easily convinced by taking simple closed curves $\alpha = \beta$.

\subsubsection{Proof part 3: encoding the contribution along a tunnel} \label{subsubsec:proof-part3}

Let $(w_{\alpha}, w_{\beta})$ be a tunnel $\tau$ through $(t_0,\dots,t_q)$. We now develop a handy way to encode the \emph{contribution $C_k(\tau)$ of the $k^{th}$ gallery of $\tau$}, that is the quantity

\begin{eqnarray} \label{eq_gallery_contrib}
C_k(\tau) := \sum_{2k \leq i,j \leq 2k+2} \left( \sum_{q_1,q_2 \subset t_k } \epsilon(q_1,q_2) q_1^* q_2^* \left(\tr (\partial_1 w^\alpha_i) \cdot \tr (\partial_2 w^\beta_j)\right)\right).
\end{eqnarray}


Following the notation of the previous paragraph, let $W^\gamma_{i}$ be the matrix in $S^\gamma_0\cdot S^\gamma_1 \cdots S^\gamma_n$ corresponding to the segment $w_{i}^\gamma$ in $(s^\gamma_0 \cdot  \dots \cdot s^\gamma_n)$. Then for all $k \in \{0,\dots,q\}$, $m \in \{0,1,2\}$ and $\gamma \in \{\alpha,\beta \}$,

\begin{align*}
q_i^* \tr( \partial_i w_{2k+m}^\gamma ) & = q_i^* \tr( S^\gamma_0 \cdots \partial_i W^\gamma_{2k+m} \cdots S^\gamma_n  )\\
& = \frac{1}{3} \tr( S^\gamma_0 \cdots W^\gamma_{2k+m} \cdot \widetilde{X} \cdots S^\gamma_n  ) \qquad \mbox{ see Remark~\ref{rem_segment_derivative}, }\\
& = \frac{1}{3} \tr( W^\gamma_{2k+m} \cdot \widetilde{X} \cdot \cdots S^\gamma_n \cdot S^\gamma_0 \cdots ) \qquad \mbox{ by conjugation invariance of } \tr,\\
& = \frac{1}{3} \tr( X \cdot W^\gamma_{2k+1} \cdot \cdots S^\gamma_n \cdot S^\gamma_0 \cdots ),\\
\end{align*}
where $\widetilde{X} = \widetilde{X}(\gamma,q_i,k,m)$ is a matrix dependent on $(\gamma,q_i,k,m)$ and
\begin{itemize}
	\item for $\gamma = \alpha$,
	$$
	X = \left(\prod_{i = 1}^{m} W^\alpha_{2k+i}  \right)\cdot \widetilde{X} \cdot  \left(\prod_{i = 1}^{m} W^\alpha_{2k+i}  \right)^{-1};
	$$
	\item for $\gamma = \beta$, if $\beta$ is oriented as $\alpha$ along the tunnel,
	$$X = \left(\prod_{i = 1}^{m} W^\beta_{2k+i}  \right)\cdot \widetilde{X} \cdot  \left(\prod_{i = 1}^{m} W^\beta_{2k+i}  \right)^{-1};
	$$
	otherwise
	$$
	X = \left(\prod_{i = 0}^{1-m} W^\beta_{2k+1-i}  \right)\cdot \widetilde{X} \cdot  \left(\prod_{i = 0}^{1-m} W^\beta_{2k+1-i}  \right)^{-1}.
	$$
\end{itemize}

In short we can write
$$
q_i^* \tr( \partial_i w_{2k+m}^\gamma )  = \frac{1}{3} \tr( X \cdot \mathbf{S}(w^\gamma_{2k+1}) ),
$$
where $X = X(\gamma,q_i,k,m)$ is a matrix dependent on $(\gamma,q_i,k,m)$ and
$$
\mathbf{S}(w^\gamma_{2k+1})  := W^\gamma_{2k+1} \cdot \cdots S^\gamma_n \cdot S^\gamma_0 \cdots,
$$
is the matrix conjugate to $S^\gamma_0 \cdots W^\gamma_{2k+1} \cdots S^\gamma_n$ obtained by moving all the matrices on the left side of $W^\gamma_{2k+1}$ to the right side of $S^\gamma_n$.

\begin{lem} \label{lem_trace_factorisation}
	Let $X,Y \in M_{3 \times 3}(\RR)$ be real matrices and let $\mathbf{E}_{ij} \in M_{3 \times 3}(\RR)$ be the matrix with $(i,j)^{th}$ entry $1$ and zeros everywhere else. Then
	$$
	\tr(X \cdot Y) = \sum_{i,j} x_{ij} \tr(\mathbf{E}_{ij} \cdot Y),
	$$
	where $x_{ij}$ is the $(i,j)^{th}$ entry of $X$.
\end{lem}
%

By Lemma~\ref{lem_trace_factorisation} and the above discussion we can conclude that

\begin{align*}
C_k(\tau) = & \sum_{0 \leq m_1,m_2 \leq 2} \left( \sum_{q_1,q_2 \subset t_k } \epsilon(q_1,q_2) q_1^* q_2^* \left(\tr (\partial_1 w^\alpha_{2k+m_1}) \cdot \tr (\partial_2 w^\beta_{2k+m_2})\right)\right) \\
=& \frac{1}{9}  \sum_{0 \leq m_1,m_2 \leq 2} \left( \sum_{q_1,q_2 \subset t_k } \epsilon(q_1,q_2)  \bigg( \tr \Big( X(\alpha,q_1,k,m_1) \cdot \mathbf{S}(w^\alpha_{2k+1}) \Big) \cdot \tr \Big( X(\beta,q_2,k,m_2) \cdot \mathbf{S}(w^\beta_{2k+1}) \Big) \bigg) \right) \\
=& \frac{1}{9} \sum_{l,m} \sum_{0 \leq m_1,m_2 \leq 2} \left( \sum_{q_1,q_2 \subset t_k } \epsilon(q_1,q_2) \bigg( x^{(\alpha,q_1,k,m_1)}_{lm} \tr \Big( \mathbf{E}_{lm} \cdot \mathbf{S}(w^\alpha_{2k+1}) \Big) \cdot \tr \Big( X(\beta,q_2,k,m_2) \cdot \mathbf{S}(w^\beta_{2k+1}) \Big) \bigg) \right) \\
=& \frac{1}{9} \sum_{l,m} \tr \Big( \mathbf{E}_{lm} \cdot \mathbf{S}(w^\alpha_{2k+1}) \Big)  \sum_{0 \leq m_1,m_2 \leq 2} \left( \sum_{q_1,q_2 \subset t_k } \epsilon(q_1,q_2) \bigg( x^{(\alpha,q_1,k,m_1)}_{lm} \cdot \tr \Big( X(\beta,q_2,k,m_2) \cdot \mathbf{S}(w^\beta_{2k+1}) \Big) \bigg) \right) \\
=& \frac{1}{9} \sum_{l,m} \tr \Big( \mathbf{E}_{lm} \cdot \mathbf{S}(w^\alpha_{2k+1}) \Big) \cdot \tr \Big( \left( \sum_{0 \leq m_1,m_2 \leq 2} \sum_{q_1,q_2 \subset t_k } \epsilon(q_1,q_2) x^{(\alpha,q_1,k,m_1)}_{lm} X(\beta,q_2,k,m_2) \right) \cdot \mathbf{S}(w^\beta_{2k+1})  \Big) \\
=&\frac{1}{9}\sum_{l,m} \tr(\mathbf{E}_{lm} \cdot \mathbf{S}(w^\alpha_{2k+1})) \cdot \tr( U^{(l,m)}  \cdot \mathbf{S}(w^\beta_{2k+1}) ).
\end{align*}

where

\begin{equation}\label{eq:contribution-matrix-entries}
U^{(l,m)} := \sum_{0 \leq m_1,m_2 \leq 2}  \sum_{q_1,q_2 \subset t_k } \epsilon(q_1,q_2) x^{(\alpha,q_1,k,m_1)}_{lm} X(\beta,q_2,k,m_2).
\end{equation}

with $x^{(\alpha,q_1,k,m_1)}_{lm}$ being the $(l,m)^{th}$ entry of $X(\alpha,q_1,k,m_1)$.\\

Let $\mathbf{M}_k(\tau) \in M_{3\times 3} (M_{3\times 3}(\RR))$ be the $3 \times 3$ matrix whose $(l,m)^{th}$ entry is the $3 \times 3$ matrix $U^{(l,m)}$. Then $\mathbf{M}_k(\tau)$ encodes $C_k(\tau)$ in the sense that
$$
C_k(\tau) = \frac{1}{9}\sum_{l,m} \tr(\mathbf{E}_{lm} \cdot \mathbf{S}(w^\alpha_{2k+1})) \cdot \tr( (\mathbf{M}_k(\tau))_{lm}  \cdot \mathbf{S}(w^\beta_{2k+1}) ).
$$

The matrix $\mathbf{M}_k(\tau)$ is unique and uniquely determines $C_k(\tau)$, thus we call $\mathbf{M}_k(\tau)$ the \emph{contribution matrix} of the $k^{th}$ gallery of $\tau$.

\begin{exa}
	We show how to compute the contribution matrix for a gallery with an explicit example. Consider the $0{th}$ gallery of a tunnel $\tau$ where $t_0$ is a starting triangle of type $1$ with respect to the pairs $(w^\alpha_i,w^\beta_i)_{0\leq i \leq 2}$ as in Figure~\ref{example_gallery_contrib}.
	
	\begin{figure}[ht]
		\centering
		\includegraphics[height=5cm]{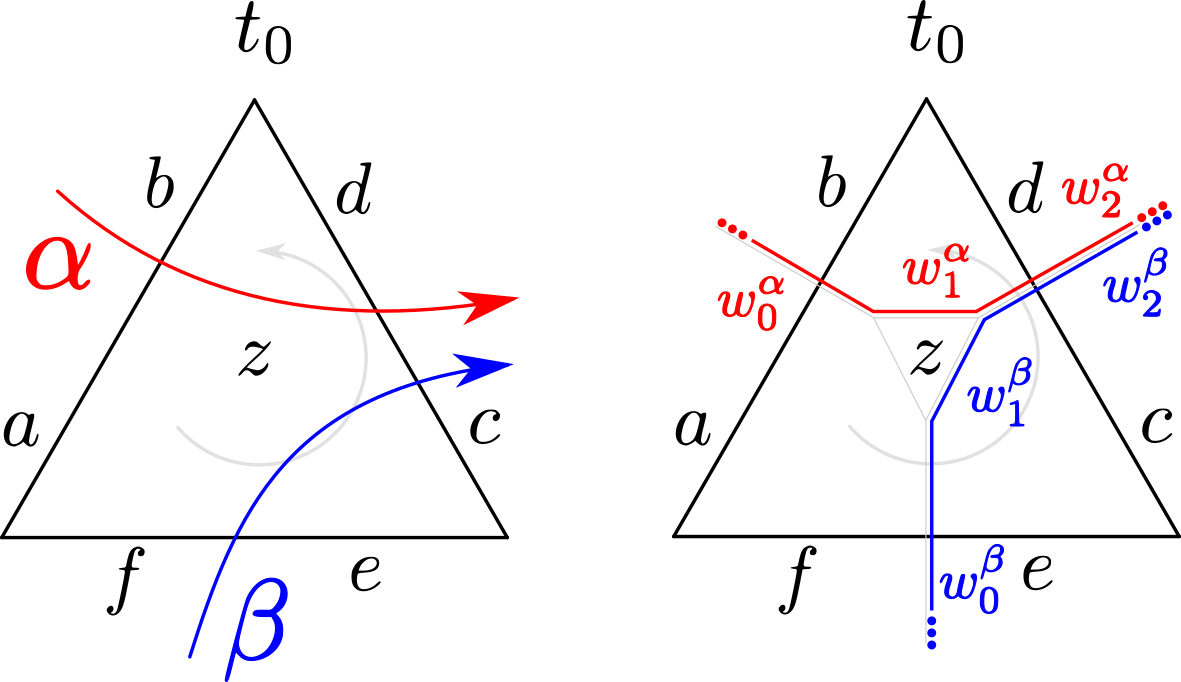}
		\caption{Example of contribution along a gallery with triangle of type $1$.}
		\label{example_gallery_contrib}
	\end{figure}
	
	For some matrices $A,B$,
	\begin{align*}
	\mu(x)(\alpha) &= [E(a,b) \cdot T(z)^{-1} \cdot E(c,d) \cdot A],\\
	\mu(x)(\beta) &= [E(e,f) \cdot T(z) \cdot E(c,d) \cdot B].
	\end{align*}
	
	We apply Lemma~\ref{lemma_basic_derivatives} and the conjugacy invariance of the trace operator to obtain:
	\begin{align*}
	q_j^* \tr (\partial_j w^\alpha_i) &= \frac{1}{3} \tr( X^{(\alpha,q_j,i)}  \cdot T(z)^{-1} \cdot E(c,d) \cdot A \cdot E(a,b) ),\\
	q_j^* \tr (\partial_j w^\beta_i) &= \frac{1}{3} \tr( X^{(\beta,q_j,i)}  \cdot T(z) \cdot E(c,d) \cdot B \cdot E(e,f) ),
	\end{align*}
	where explicitly (see Lemma~\ref{lemma_basic_derivatives} for the notation):
	\begin{align*}
	X^{(\alpha,a,0)} = M_1 = X^{(\beta,e,0)} &\qquad X^{(\alpha,b,0)} = M_2 = X^{(\beta,f,0)}\\
	X^{(\alpha,z,1)} = T(z)^{-1} \cdot M_4 \cdot T(z) &\qquad X^{(\beta,z,1)} = T(z) \cdot M_3 \cdot T(z)^{-1}\\ 
	X^{(\alpha,c,2)} = T(z)^{-1} \cdot E(c,d) \cdot M_1 \cdot (T(z)^{-1} \cdot E(c,d))^{-1} &\qquad X^{(\beta,c,2)} = T(z) \cdot E(c,d) \cdot M_1 \cdot (T(z) \cdot E(c,d))^{-1}\\
	X^{(\alpha,d,2)} = T(z)^{-1} \cdot E(c,d) \cdot M_2 \cdot (T(z)^{-1} \cdot E(c,d))^{-1} &\qquad X^{(\beta,d,2)} = T(z) \cdot E(c,d) \cdot M_2 \cdot (T(z) \cdot E(c,d))^{-1}\\
	X^{(\alpha,q_j,i)} = X^{(\beta,q_j,i)} = &\mbox{ the zero matrix otherwise}.
	\end{align*}
	
	We can now compute the entries of $\mathbf{M}_0(\tau)$ from equation (\ref{eq:contribution-matrix-entries}), for example:
	\begin{align*}
	U^{(1,1)} =  &x^{(\alpha,a,0)}_{11}\left( \epsilon(a,f)X^{(\beta,f,0)} + \epsilon(a,z)X^{(\beta,z,1)} \right) +  x^{(\alpha,b,0)}_{11}\left( \epsilon(b,d)X^{(\beta,d,2)} + \epsilon(b,z)X^{(\beta,z,1)} \right) +\\ &x^{(\alpha,z,1)}_{11}\left( \epsilon(z,f)X^{(\beta,f,0)} + \epsilon(z,e)X^{(\beta,e,0)} + \epsilon(z,c)X^{(\beta,c,2)} + \epsilon(z,d)X^{(\beta,d,2)} \right) +\\ &x^{(\alpha,c,2)}_{11}\left( \epsilon(c,e)X^{(\beta,e,0)} + \epsilon(c,z)X^{(\beta,z,1)} \right) + x^{(\alpha,d,2)}_{11}\left( \epsilon(d,z)X^{(\beta,z,1)} \right).
	\end{align*}
	
	It follows that the contribution matrix of the $0^{th}$ gallery is
	
	$$
	\mathbf{M}_0(\tau) = 
	\left(\begin{array}{ccc}
	
	\left(\begin{array}{ccc}
	3 & 0 & 0\\
	-6 & -3 & 0\\
	3(z+2) & 3(z+1) & 0
	\end{array}\right) &
	\left(\begin{array}{ccc}
	\frac{3(z+1)}{z} & 0 & 0\\
	-9 & -\frac{6(z+1)}{z} & 0\\
	9(z+1) & \frac{9(z+1)^2}{z} & \frac{3(z+1)}{z}
	\end{array}\right) &
	\left(\begin{array}{ccc}
	\frac{3(z-1)}{z} & 0 & 0\\
	0 & -\frac{3(z+2)}{z} & 0\\
	0 & \frac{9(z+1)}{z} & \frac{3(2z+1)}{z}
	\end{array}\right)\\
	\left(\begin{array}{ccc}
	0 & 0 & 0\\
	0 & 0 & 0\\
	0 & 0 & 0
	\end{array}\right) &
	\left(\begin{array}{ccc}
	0 & 0 & 0\\
	3 & 3 & 0\\
	-3(2z+1) & -6(z+1) & -3
	\end{array}\right) &
	\left(\begin{array}{ccc}
	3 & 0 & 0\\
	0 & 3 & 0\\
	0 & -9 & -6
	\end{array}\right)\\
	\left(\begin{array}{ccc}
	0 & 0 & 0\\
	0 & 0 & 0\\
	0 & 0 & 0
	\end{array}\right) &
	\left(\begin{array}{ccc}
	0 & 0 & 0\\
	0 & 0 & 0\\
	0 & 0 & 0
	\end{array}\right) &
	\left(\begin{array}{ccc}
	-3 & 0 & 0\\
	3 & 0 & 0\\
	3(z-1) & 3(z+1) & 3
	\end{array}\right)\\
	\end{array}\right)
	$$

	and
	
	$$
	C_0(\tau) = \frac{1}{9}\sum_{l,m} \tr(\mathbf{E}_{lm} \cdot T(z)^{-1} \cdot E(c,d) \cdot A \cdot E(a,b)) \cdot \tr( (\mathbf{M}_0(\tau))_{lm}  \cdot T(z) \cdot E(c,d) \cdot B \cdot E(e,f) ).
	$$
\end{exa}

\begin{lem} \label{lem:sum-of-contributions}
	Let $C_{k-1}(\tau),C_{k}(\tau)$ and $\mathbf{M}_{k-1}(\tau), \mathbf{M}_{k}(\tau)$ be the contributions and contribution matrices of two consecutive galleries in a tunnel $\tau$, and let
	\begin{itemize}
		\item $a^{(l,m)}_{ij}$ be the $(i,j)$--th entry of the matrix $A^{(l,m)} = (W^\alpha_{2k-1} \cdot W^\alpha_{2k})^{-1} \cdot \mathbf{E}_{lm} \cdot W^\alpha_{2k-1} \cdot W^\alpha_{2k}$,
		\item if $\alpha,\beta$ are oriented in the same direction along the galleries,
		$$
		B^{(l,m)} = (W^\beta_{2k-1} \cdot W^\beta_{2k})^{-1} \cdot (\mathbf{M}_{k-1}(\tau))_{lm} \cdot W^\beta_{2k-1} \cdot W^\beta_{2k},
		$$
		otherwise
		$$
		B^{(l,m)} = W^\beta_{2k+1} \cdot W^\beta_{2k}  \cdot (\mathbf{M}_{k-1}(\tau))_{lm} \cdot ( W^\beta_{2k+1} \cdot W^\beta_{2k} )^{-1}.
		$$
	\end{itemize}
	
	Then
	$$
	C_{k-1}(\tau) + C_{k}(\tau)  = \frac{1}{9}\sum_{l,m} \tr(\mathbf{E}_{lm} \cdot \mathbf{S}(w^\alpha_{2k+1})) \cdot \tr( ({\mathbf{M}_{(k-1,k)}}(\tau))_{lm}  \cdot \mathbf{S}(w^\beta_{2k+1}) ),
	$$
	where $\mathbf{M}_{(k-1,k)}(\tau) \in M_{3\times 3} (M_{3\times 3}(\RR)) $ is the matrix with $(l,m)^{th}$ entry
	$$
	(\mathbf{M}_{(k-1,k)}(\tau))_{lm} = \sum_{i,j} a^{(i,j)}_{lm} B^{(i,j)} + (\mathbf{M}_k(\tau))_{lm}.
	$$
\end{lem}

\begin{proof}
	\begin{align*}
	C_{k-1}(\tau) & = \frac{1}{9}\sum_{l,m} \tr(\mathbf{E}_{lm} \cdot \mathbf{S}(w^\alpha_{2k-1})) \cdot \tr( (\mathbf{M}_{k-1}(\tau))_{lm}  \cdot \mathbf{S}(w^\beta_{2k-1}) )\\
	& = \frac{1}{9}\sum_{l,m} \tr(A^{(l,m)} \cdot \mathbf{S}(w^\alpha_{2k+1})) \cdot \tr( B^{(l,m)}  \cdot \mathbf{S}(w^\beta_{2k+1}) )  \qquad \mbox{ by conjugation,}\\
	& = \frac{1}{9}\sum_{l,m}  \sum_{i,j} a^{(l,m)}_{ij} \tr( \mathbf{E}_{ij} \cdot \mathbf{S}(w^\alpha_{2k+1})) \cdot \tr( B^{(l,m)}  \cdot \mathbf{S}(w^\beta_{2k+1}) ) \qquad \mbox{ from Lemma~\ref{lem_trace_factorisation},}\\
	& = \frac{1}{9}\sum_{i,j} \tr(\mathbf{E}_{ij} \cdot \mathbf{S}(w^\alpha_{2k+1})) \cdot \tr(  \sum_{l,m} a^{(l,m)}_{ij} \cdot B^{(l,m)}  \cdot \mathbf{S}(w^\beta_{2k+1}) )\\
	& = \frac{1}{9}\sum_{l,m} \tr(\mathbf{E}_{lm} \cdot \mathbf{S}(w^\alpha_{2k+1})) \cdot \tr(  \sum_{i,j} a^{(i,j)}_{lm} \cdot B^{(i,j)}  \cdot \mathbf{S}(w^\beta_{2k+1}) )  \qquad \mbox{ after renaming indices.}
	\end{align*}
	But
	$$
	C_k(\tau) = \frac{1}{9}\sum_{l,m} \tr(\mathbf{E}_{lm} \cdot \mathbf{S}(w^\alpha_{2k+1})) \cdot \tr( (\mathbf{M}_k(\tau))_{lm}  \cdot \mathbf{S}(w^\beta_{2k+1}) ),
	$$
	hence the result follows.
\end{proof}

\begin{rem}
	From a computational point of view, it is interesting to observe that the matrix $\sum_{i,j} a^{(i,j)}_{lm} \cdot B^{(i,j)}$ is the $(l,m)^{th}$ entry of $(W^\alpha_{2k-1} \cdot W^\alpha_{2k})^{-1} \cdot \mathbf{B} \cdot W^\alpha_{2k-1} \cdot W^\alpha_{2k}$, where $\mathbf{B} \in M_{3\times 3} (M_{3\times 3}(\RR)) $ is the matrix with matrix entries $B^{(i,j)}$.
\end{rem}

For each starting triangle configuration, there are only two compatible ending triangles, which makes a total of $8$ possible tunnels without transition triangles (cf. Figure~\ref{fig:length-two-tunnel}).

\begin{figure}[h]
	\centering
	\includegraphics[height=5cm]{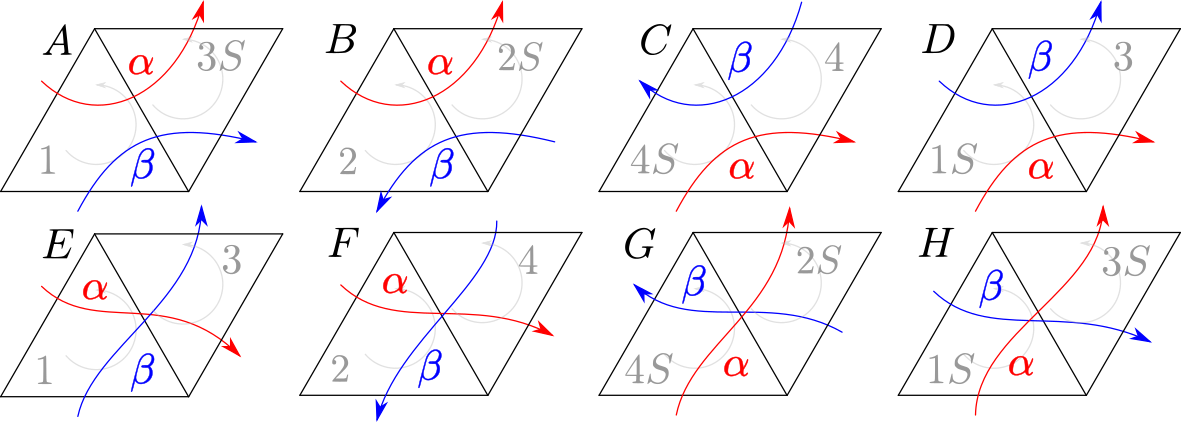}
	\caption{Length two tunnels.}
	\label{fig:length-two-tunnel}
\end{figure}

\begin{lem} \label{lem:two-galleries-contribution}
	The contribution along tunnels of length two equals
	$$
	\epsilon(p,\alpha,\beta) \left( \tr_{\alpha_p \beta_p}\circ \mu - \frac{1}{3} (\tr_\alpha \circ \mu )(\tr_\beta \circ \mu ) \right),
	$$
	where $p$ is the point of intersection (if any) of $\alpha$ and $\beta$ in the tunnel.
\end{lem}

\begin{proof}
	The same technique applies to each case, hence we here our attention to just one of them. We will briefly explain at the end how to treat the others.
	
	Let $(w^{\alpha}_i, w^{\beta}_i)_{0\leq i \leq 4}$ be a tunnel $\tau$ through $(t_0,t_1)$ of type C (as in Figure~\ref{fig:length-two-tunnel}, namely $t_0$ is a starting triangle of type $4S$ (which is type $4$ with $\alpha$ and $\beta$ switched) with respect to the pairs $(w^{\alpha}_i, w^{\beta}_i)_{0\leq i \leq 2}$, and $t_1$ is an ending triangle of type $4$ for $(w^{\alpha}_i, w^{\beta}_i)_{2\leq i \leq 4}$. According to the labels shown in Figure~\ref{example_type_c}, we have
	\begin{align*}
	\mu(x)(\alpha) &= [E(f,e) \cdot T(z) \cdot E(d,c) \cdot T(w) \cdot E(g,h) \cdot A],\\
	\mu(x)(\beta) &= [E(j,i) \cdot T(w) \cdot E(c,d) \cdot T(z) \cdot E(a,b) \cdot B],
	\end{align*}
	for some matrices $A,B$.
	
	
	\begin{figure}[h]
		\centering
		\includegraphics[height=5cm]{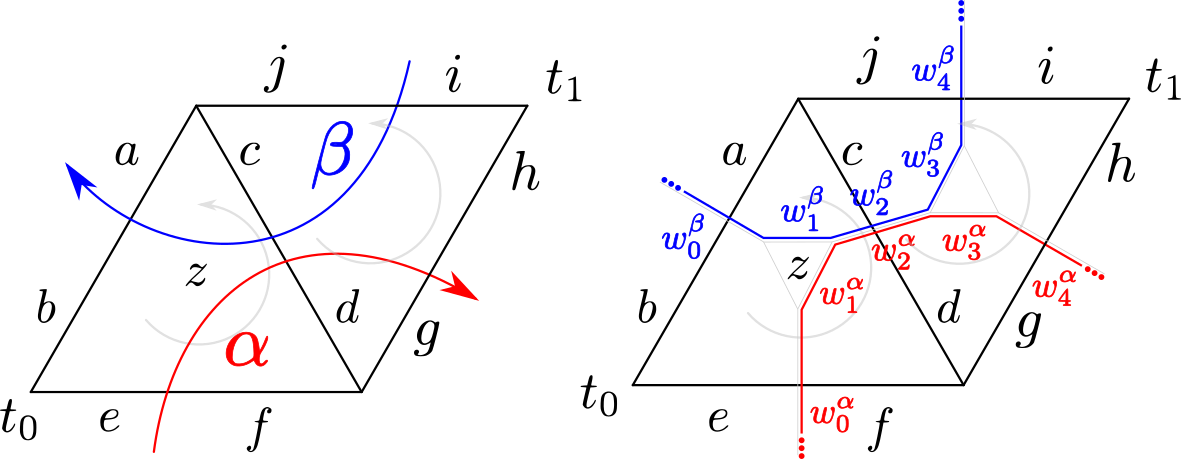}
		\caption{Example of type C.}
		\label{example_type_c}
	\end{figure}
	
	The contribution matrices $\mathbf{M}_0(\tau), \mathbf{M}_1(\tau)$ at $t_0,t_1$ are:
	$$
	\mathbf{M}_0(\tau) = 
	\left(\begin{array}{ccc}
	
	\left(\begin{array}{ccc}
	0 & \frac{3(z+1)}{z} & \frac{3(2z+1)}{z}\\
	0 & -3 & -6\\
	0 & 0 & 3
	\end{array}\right) &
	\left(\begin{array}{ccc}
	0 & 0 & \frac{9(z+1)}{z}\\
	0 & 0 & -9\\
	0 & 0 & 0
	\end{array}\right) &
	\left(\begin{array}{ccc}
	0 & 0 & \frac{9}{z}\\
	0 & 0 & 0\\
	0 & 0 & 0
	\end{array}\right)\\
	\left(\begin{array}{ccc}
	0 & 0 & 0\\
	0 & 0 & 0\\
	0 & 0 & 0
	\end{array}\right) &
	\left(\begin{array}{ccc}
	-3 & -\frac{6(z+1)}{z} & -\frac{3(z+2)}{z}\\
	0 & 3 & 3\\
	0 & 0 & 0
	\end{array}\right) &
	\left(\begin{array}{ccc}
	0 & -\frac{9}{z} & -\frac{9}{z}\\
	0 & 0 & 0\\
	0 & 0 & 0
	\end{array}\right)\\
	\left(\begin{array}{ccc}
	0 & 0 & 0\\
	0 & 0 & 0\\
	0 & 0 & 0
	\end{array}\right) &
	\left(\begin{array}{ccc}
	0 & 0 & 0\\
	0 & 0 & 0\\
	0 & 0 & 0
	\end{array}\right) &
	\left(\begin{array}{ccc}
	3 & \frac{3(z+1)}{z} & -\frac{3(z-1)}{z}\\
	0 & 0 & 3\\
	0 & 0 & -3
	\end{array}\right)\\
	\end{array}\right),
	$$
	
	$$
	\mathbf{M}_1(\tau) = 
	\left(\begin{array}{ccc}
	
	\left(\begin{array}{ccc}
	0 & 0 & 0\\
	0 & 3 & 0\\
	0 & 0 & -3
	\end{array}\right) &
	\left(\begin{array}{ccc}
	-\frac{3(w+1)}{w} & 0 & 0\\
	0 & \frac{6(w+1)}{w} & \frac{9}{w}\\
	0 & 0 & -\frac{3(w+1)}{w}
	\end{array}\right) &
	\left(\begin{array}{ccc}
	-\frac{3(2w+1)}{w} & -\frac{9(w+1)}{w} & -\frac{9}{w}\\
	0 & \frac{3(w+2)}{w} & \frac{9}{w}\\
	0 & 0 & \frac{3(w-1)}{w}
	\end{array}\right)\\
	\left(\begin{array}{ccc}
	0 & 0 & 0\\
	0 & 0 & 0\\
	0 & 0 & 0
	\end{array}\right) &
	\left(\begin{array}{ccc}
	3 & 0 & 0\\
	0 & -3 & 0\\
	0 & 0 & 0
	\end{array}\right) &
	\left(\begin{array}{ccc}
	6 & 9 & 0\\
	0 & -3 & 0\\
	0 & 0 & -3
	\end{array}\right)\\
	\left(\begin{array}{ccc}
	0 & 0 & 0\\
	0 & 0 & 0\\
	0 & 0 & 0
	\end{array}\right) &
	\left(\begin{array}{ccc}
	0 & 0 & 0\\
	0 & 0 & 0\\
	0 & 0 & 0
	\end{array}\right) &
	\left(\begin{array}{ccc}
	-3 & 0 & 0\\
	0 & 0 & 0\\
	0 & 0 & 3
	\end{array}\right)\\
	\end{array}\right).
	$$

	Hence we apply Lemma~\ref{lem:sum-of-contributions} to check that $\mathbf{M}_{(0,1)}(\tau)$ is the zero matrix, which implies that the contribution along the tunnel of type C is zero. Since there is no point of intersection $p$ in the tunnel $\epsilon(p,\alpha,\beta) = 0$ and the statement is proved for type C.
	
	Cases A, B and D are treated verbatim, while for E, F, G and H one needs only to check that
	$$
	\epsilon(p,\alpha,\beta) \frac{\mathbf{M}_{(0,1)}(\tau) }{9} = 
	\left(\begin{array}{ccc}
	
	\left(\begin{array}{ccc}
	\frac{2}{3} & 0 & 0\\
	0 & -\frac{1}{3} & 0\\
	0 & 0 & -\frac{1}{3}
	\end{array}\right) &
	\left(\begin{array}{ccc}
	0 & 0 & 0\\
	1 & 0 & 0\\
	0 & 0 & 0
	\end{array}\right) &
	\left(\begin{array}{ccc}
	0 & 0 & 0\\
	0 & 0 & 0\\
	1 & 0 & 0
	\end{array}\right)\\
	\left(\begin{array}{ccc}
	0 & 1 & 0\\
	0 & 0 & 0\\
	0 & 0 & 0
	\end{array}\right) &
	\left(\begin{array}{ccc}
	-\frac{1}{3} & 0 & 0\\
	0 & \frac{2}{3} & 0\\
	0 & 0 & -\frac{1}{3}
	\end{array}\right) &
	\left(\begin{array}{ccc}
	0 & 0 & 0\\
	0 & 0 & 0\\
	0 & 1 & 0
	\end{array}\right)\\
	\left(\begin{array}{ccc}
	0 & 0 & 1\\
	0 & 0 & 0\\
	0 & 0 & 0
	\end{array}\right) &
	\left(\begin{array}{ccc}
	0 & 0 & 0\\
	0 & 0 & 1\\
	0 & 0 & 0
	\end{array}\right) &
	\left(\begin{array}{ccc}
	-\frac{1}{3} & 0 & 0\\
	0 & -\frac{1}{3} & 0\\
	0 & 0 & \frac{2}{3}
	\end{array}\right)\\
	\end{array}\right).
	$$
	Indeed, it is easy to check that such matrix encodes precisely the contribution
	$$
	\left( \tr_{\alpha_p \beta_p}\circ \mu - \frac{1}{3} (\tr_\alpha \circ \mu )(\tr_\beta \circ \mu ) \right),
	$$
	and the proof is complete.
\end{proof}

\begin{lem} \label{lem:tunnel-contribution}
	The contribution along any tunnel equals
	$$
	\epsilon(p,\alpha,\beta) \left( \tr_{\alpha_p \beta_p}\circ \mu - \frac{1}{3} (\tr_\alpha \circ \mu )(\tr_\beta \circ \mu ) \right),
	$$
	where $p$ is the point of intersection (if any) of $\alpha$ and $\beta$ in the tunnel.
\end{lem}

\begin{proof}
	Let $(w_{\alpha}, w_{\beta})$ be a tunnel $\tau$ through $(t_0,\dots,t_q)$. Lemma~\ref{lem:two-galleries-contribution} applies for $q=1$. Hence assume $q>1$. For all $0<k<q$, $t_k$ is a transition triangle whose type depends on the relative orientation of $\alpha$ and $\beta$ along the tunnel. In particular, if $t_0$ is a starting triangle of type $1, 1S$ (resp. $2, 4S$), then only transition triangles of type $5, 7$ (resp. $6, 6S$) are allowed.
	
	Let $C_k(\tau)$ and $\mathbf{M}_k(\tau)$ be the contribution and the contribution matrix of the $k^{th}$ gallery. Once can apply Lemma~\ref{lem:sum-of-contributions} to check that $\mathbf{M}_{(0,1)}(\tau) = \mathbf{M}_0(\tau)$ for all allowed type combinations of $t_0$ and $t_1$. It follows by induction that $\mathbf{M}_{(0,\dots,k)}(\tau) = \mathbf{M}_0(\tau)$ for all $k<q$, where $\mathbf{M}_{(0,\dots,k)}(\tau) \in M_{3\times 3} (M_{3\times 3}(\RR))$ is the matrix contribution of the first $k$ galleries, namely the matrix for which
	$$
	\sum_{i = 0}^{k}C_i(\tau)  = \frac{2}{9}\sum_{l,m} \tr(\mathbf{E}_{lm} \cdot \mathbf{S}(w^\alpha_{2k+1})) \cdot \tr( ({\mathbf{M}_{(0,\dots,k)}(\tau)})_{lm}  \cdot \mathbf{S}(w^\beta_{2k+1}) ),
	$$
	In conclusion $\mathbf{M}_{(0,\dots,q-1)}(\tau) = \mathbf{M}_0(\tau)$ and we reduced to the case of length two tunnels, for which Lemma~\ref{lem:two-galleries-contribution} applies.
\end{proof}

\subsubsection{Proof part 4: parallel curves and closed tunnels} \label{subsubsec:proof-part4}

When $\alpha$ and $\beta$ are parallel, that is $[\alpha] = [\beta^n]$ for some $n \in \ZZ \setminus \{0\}$, a slightly different kind of tunnel appears.\\

We will say that $(w_{\alpha},w_{\beta})$ is a \emph{closed tunnel through} $(t_0,\dots,t_q)$ if $t_j$ is a transition triangle with respect to the pairs $(w^\alpha_i,w^\beta_i)_{2j \leq i \leq 2j+2}$ for all $j \in \{ 0,\dots, q\}$, and $w^\gamma_0 = w^\gamma_{2k+2}$ for $\gamma \in \{\alpha,\beta\}$.

\begin{lem} \label{lem_closed_tunnel_contribution}
	The contribution along a closed tunnel is $0$.
\end{lem}

\begin{proof}
	Let $(w_{\alpha},w_{\beta})$ be a closed tunnel $\tau$ through $(t_0,\dots,t_q)$. We will show that $(w_{\alpha},w_{\beta})$ can be modified to an ``open" tunnel with contribution $0$. More precisely, we claim that the contribution along $\tau$ equals the contribution along the ``open" tunnel $\widetilde{\tau} = (\widetilde{w_{\alpha}},\widetilde{w_{\beta}})$ through $(t_0,\dots,t_q,t_{q+1})$ defined by:
	\begin{itemize}
		\item $t_{q+1} := t_0$;
		\item $\widetilde{w_{\alpha}} := ( w_0^{\alpha} w_1^{\alpha} w_2^{\alpha} w_3^{\alpha}  \dots w_{2q+2}^{\alpha} w_1^{\alpha}  w_2^{\alpha} )$;
		\item $\widetilde{w_{\beta}} := ( \widetilde{w_0} \widetilde{w_1} w_2^{\beta} w_3^{\beta}  \dots w_{2q+2}^{\beta} \widetilde{w_2} \widetilde{w_3} )$ where $\widetilde{w_0}, \widetilde{w_1}, \widetilde{w_2}, \widetilde{w_3}$ are as in Figure~\ref{fig:closed-tunnel}.
	\end{itemize}
	
	\begin{figure}[h]
		\centering
		\includegraphics[height=5cm]{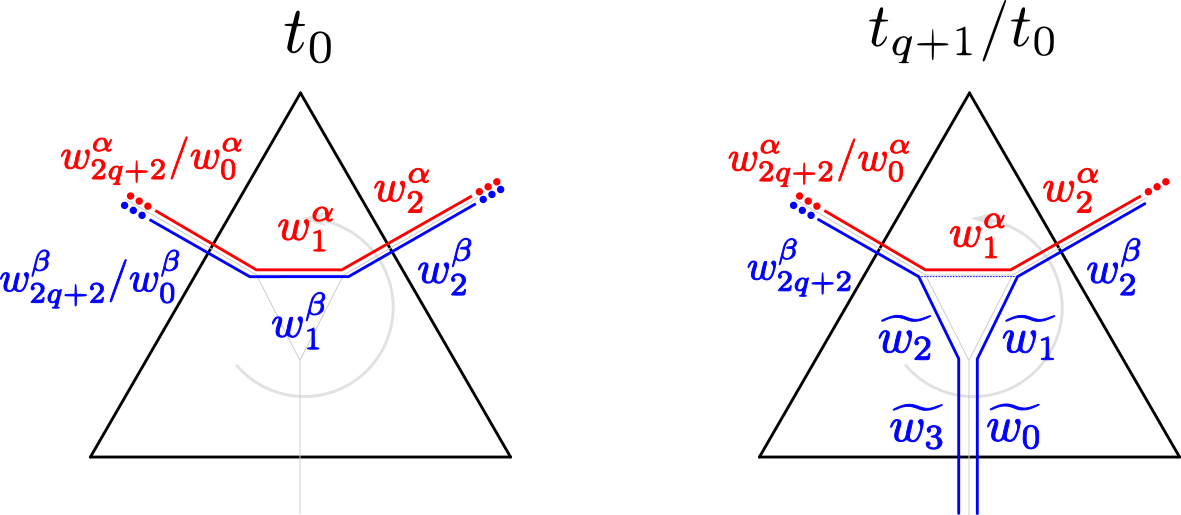}
		\caption{Local modification of $\beta$ in a closed tunnel.}
		\label{fig:closed-tunnel}
	\end{figure}
	
	It is clear that $C_k(\tau) = C_k(\widetilde{\tau})$ for all $1 \leq k \leq q$, as we have only modified $\beta$ locally at $t_0$. Nevertheless, the contribution of the $0^{th}$ gallery in $(w_{\alpha},w_{\beta})$ is
	$$
		C_0(\tau) =  \sum_{0 \leq i,j \leq 2} \left( \sum_{q_1,q_2 \subset t_0 } \epsilon(q_1,q_2) q_1^* q_2^* \left(\tr (\partial_1 w^\alpha_i) \cdot \tr (\partial_2 w^\beta_j)\right)\right).
	$$
	In such a gallery, we replace
	\begin{align*}
	\tr (\partial_i w^\beta_1) & = \tr( S^\beta_0 \cdots \partial_i W^\beta_{1} \cdots S^\beta_n  )\\
	& = \tr( S^\beta_0 \cdots \partial_i \left((W^\beta_{1})^{-1} \widetilde{W} \ \widetilde{W}^{-1} (W^\beta_{1})^{-1}\right) \cdots S^\beta_n  ) \\
	& =  \tr (\partial_i \widetilde{w_0}) + \tr (\partial_i \widetilde{w_1} ) + \tr (\partial_i \widetilde{w_2} ) + \tr (\partial_i \widetilde{w_3} ),
	\end{align*}
	where $\widetilde{W}$ is the matrix associated to the $\underline{E}$--edge corresponding to $\widetilde{w_0}$ (and $\widetilde{w_3}$). It follows that $C_0(\tau) = C_0(\widetilde{\tau}) + C_{q+1}(\widetilde{\tau})$ and the claim is proved.
	
	Finally, since $(\widetilde{w_{\alpha}},\widetilde{w_{\beta}})$ is a tunnel with no intersection points, the conclusion is implied by Lemma~\ref{lem:tunnel-contribution}.
\end{proof}

\subsubsection{Proof part 5: conclusion} \label{subsubsec:proof-part5}

Now it is time to put all the pieces together and prove Theorem~\ref{thm:compatibility-poisson-bracket}. Recall that our goal is to show that
$$
\sum_{p \in \alpha \cap \beta} \epsilon(p,\alpha,\beta) \left( \tr_{\alpha_p \beta_p}\circ \mu - \frac{1}{3} (\tr_\alpha \circ \mu )(\tr_\beta \circ \mu ) \right) = \sum_{i,j} \sum_{q_1,q_2 \in \triangle \cup \underline{E} } \epsilon(q_1,q_2) q_1^* q_2^* \left(\tr (\partial_1 s^\alpha_i) \cdot \tr (\partial_2 s^\beta_j)\right).
$$
Firstly, assume $\alpha$ and $\beta$ are not parallel. If $(s^\alpha_i,s^\beta_j)$ do not cross a common triangle, then we can discard such pair from the right hand sum as it contributes nothing. Otherwise, $(s^\alpha_i,s^\beta_j)$ is in the gallery of a unique tunnel. By Lemma~\ref{lem:tunnel-contribution}, if $\alpha$ and $\beta$ are not intersecting along such tunnel, then the sum of the contributions of all pairs in the tunnel is zero and we can discard them too. Hence we are left only with pairs which are contained in a tunnel where $\alpha$ and $\beta$ do intersect. But since each intersection point is contained in a unique tunnel, we can apply Lemma~\ref{lem:tunnel-contribution} again to conclude the proof.

When $\alpha$ and $\beta$ are parallel, the only difference from the previous argument is the possible presence of closed tunnels. However Lemma~\ref{lem_closed_tunnel_contribution} ensures that there is no contribution from those tunnels so we may discard them and follow the above construction verbatim. This concludes the proof of Theorem~\ref{thm:compatibility-poisson-bracket}.

\newpage
\section{Examples} \label{sec:Examples}

In this section we provide examples of the properties of $\Ttp(S_{g,n})$ discussed throughout this paper. As usual $S_{g,n}$ is the oriented surface of genus $g$ with $n$ punctures and we fix an ideal triangulation of $S_{g,n}$ whose set of oriented edges is $\underline{E}$ and whose set of triangles is $\triangle$.

\subsection{The once-punctured torus}
Let $S=S_{1,1}$ as depicted by the polygonal map as in Figure~\ref{fig:fund_dom_s11}. The edge pairings are indicated by common edge invariants.
	\begin{figure}[h]
		\centering
		\includegraphics[height=5cm]{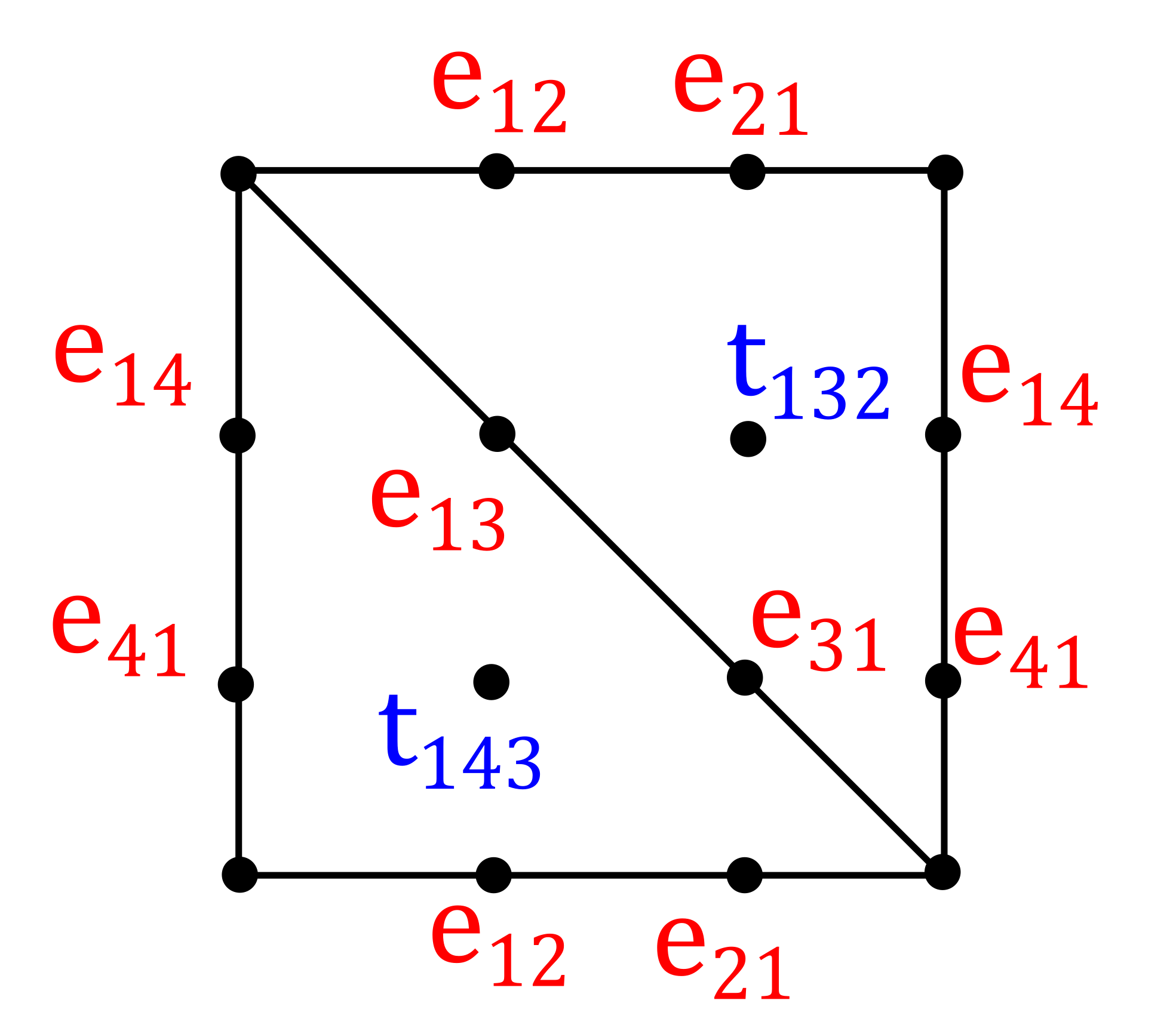}
		\caption{A polygonal map for $S_{1,1}$.}
		\label{fig:fund_dom_s11}
	\end{figure}

Fock and Goncharov's Theorem~\ref{thm:global-coord} shows that the assignment of triple ratios and edge ratios indicated in Figure~\ref{fig:fund_dom_s11} uniquely determines an element of $\Ttp(S)$. Moreover we can read off information about the corresponding holonomy group $\Gamma$ with knowledge of the edge ratios and triple ratios alone using the method of \S\ref{subsubsec:computing-monodromy}. We denote by $\alpha, \beta \in \Gamma$ the transformations corresponding to the face pairings between the vertical and horizontal edges respectively. It follows that $\Gamma \cong \langle \alpha, \beta \rangle$. The method of \S\ref{subsubsec:computing-monodromy} determines that, up to a common conjugation
\begin{align*}
	& \alpha = T(t_{143}) E(e_{13}, e_{31}) T(t_{132}) E(e_{14}, e_{41}) T(t_{143}) \\
	& \beta = E(e_{21}, e_{12}) T(t_{132}) E(e_{31}, e_{13}) T(t_{143})^{-1} \\
	\text{where} \hspace{0.5cm} & T(z) := \frac{1}{\sqrt[3]{z}}
	\left( 
	\begin{array}{ccc}
	0 & 0 & 1 \\
	0 & -1 & -1\\
	z & z+1 & 1
	\end{array}
	\right), \hspace{2cm}
	E(x,y) := \sqrt[3]{\frac{x}{y}}
	\left(
	\begin{array}{ccc}
	0 & 0 & y  \\
	0 & -1 & 0\\
	\frac{1}{x} & 0 & 0
	\end{array}
	\right)
\end{align*}
A direct calculation of the generator $\alpha \beta \alpha^{-1} \beta^{-1}$ of the peripheral subgroup of $\hol(S_{1,1})$ yields the requisite information for further inquiry into the geometry of $\Ttp(S_{1,1})$.
\[
\alpha \beta \alpha^{-1} \beta^{-1} = \frac{1}{ t_{132} t_{143} } 
\begin{bmatrix} (e_{12} e_{13} e_{21} e_{31} e_{41} e_{14})^{-1} & 0 & 0 \\
* & 1 & 0 \\
* & * & e_{12} e_{13} e_{21} e_{31} e_{41} e_{14} (t_{132} t_{143})^3
\end{bmatrix}
\]
This is the normal form we expect to arise following the construction in \S\ref{subsubsec:computing-monodromy}. Recall from \S\ref{subsec:fin_vol}, or by inspection of the above matrix, that the following monomials determine whether the end of $S_{1,1}$ is hyperbolic, special or cuspidal.
\begin{align*}
\mathcal{X}_1 :=& e_{12}e_{21}e_{14}e_{41}e_{13}e_{31} \\
\mathcal{Y}_1 :=& (t_{132}t_{143})^3 e_{12}e_{21}e_{14}e_{41}e_{13}e_{31}
\end{align*}
These monomials are our primary tools in identifying interesting subvarieties of $\Ttp(S_{1,1})$. Let us begin by identifying which coordinates give rise to structures which have minimal and maximal ends. This is a simple calculation following the classification in \S\ref{subsec:fin_vol}. The region in which coordinates give rise to structures on $S_{1,1}$ with a maximal end are:
\begin{align*}
&\mathcal{X}_1 < \mathcal{Y}_1\mathcal{X}_1  < 1 \quad \text{or} \\
&\mathcal{X}_1 > \mathcal{Y}_1\mathcal{X}_1  > 1
\end{align*}
Similarly, those structures with minimal ends are given by:
\begin{align*}
 \mathcal{Y}_1 &< 1 \quad \text{and} \quad \mathcal{X}_1  < 1 \quad \text{or} \\
 \mathcal{Y}_1\mathcal{X}_1 &< 1 \quad \text{and} \quad \mathcal{X}_1  > 1 \quad \text{or} \\
 \mathcal{Y}_1 &< 1 \quad \text{and} \quad \mathcal{X}_1  > 1 \quad \text{or} \\
 \mathcal{Y}_1\mathcal{X}_1 &> 1 \quad \text{and} \quad \mathcal{X}_1  < 1
\end{align*}
We see that generically (on structures with hyperbolic ends) there is a natural $2:1$ correspondence between structures with minimal ends and those with maximal ends, as noted in the discussion of the $\Sym(3)$-action on $\Ttp(S_{1,1})$ in \S\ref{subsubsec:sym_gp}. 

We will now identify the symplectic leaves of $\Ttp(S_{1,1})$. Following Goldman~\cite{Goldman-symplectic-1984}, or by a brief calculation, we find that $\mathcal{X}_1$ and $\mathcal{Y}_1$ are Casimirs for the Poisson structure on $\Ttp(S_{1,1})$ featured in \S\ref{sec:poisson-structure}. Note that Casimirs are constant on symplectic leaves. The condition of $\mathcal{X}_1$ and $\mathcal{Y}_1$ being constant simplifies to
\[
t_{132} t_{143} = c_0 \quad \text{ and } \quad e_{12} e_{21} e_{14} e_{41} e_{13} e_{31}= c_1 \label{level_set}
\]
for any positive real constants $c_0$ and $c_1$. In particular the symplectic leaves of $\Ttp(S_{1,1})$ are, at largest, open cells of $6$ real dimensions defined by the conditions~\ref{level_set}. Denote a generic level set of equations~\ref{level_set} by $\mathcal{A}$. Suppose there exists function $f \in C^{\infty}(A)$ such that 
\[
\{ f, \cdot \} : C^{\infty}(\mathcal{A}) \rightarrow C^{\infty}(\mathcal{A}) \label{Poisson_contraction}
\]
is the zero function. A coordinate system on $\mathcal{A}$ is determined by the functions 
\[
\mathcal{P} = \{ t^*_{143}, e^*_{12}, e^*_{21}, e^*_{14}, e^*_{41}, e^*_{13} \} \subset \C^{\infty}(\mathcal{A})
\]
where $q^*$ for $q \in \triangle \cup \underline{E}$ are \emph{coordinate functions} on $\Ttp(S_{1,1})$ defined in \S\ref{subsec:FG-poisson-structure}. Computing $\{ f, q^* \}$ for $ q^* \in \mathcal{P}$ shows that $\partial f / \partial q^* = 0$ for all such $q*$. In particular $f$ is constant on $\mathcal{A}$. Therefore, $\{ \cdot, \cdot \}$ is non-degenerate on $\mathcal{A}$ in the sense that the function \ref{Poisson_contraction} is uniformly the zero function if and only if $f$ is constant. It follows that $\mathcal{A}$ is a symplectic leaf.

Increasing once more in specificity, recall from \S\ref{subsec:fin_vol} that the finite-area structures on $S_{1,1}$ arise when $\mathcal{X}_1= \mathcal{Y}_1 =1$. Therefore the finite-area conditions for $S_{1,1}$ reduce to the particularly simple condition that $c_0=c_1=1$ in equations~\ref{level_set}.

To determine $\Teich(S_{1,1}) \subset \Ttp(S_{1,1})$ recall that the conditions for FG coordinates to represent a finite-area hyperbolic structure described in \S\ref{subsec:classical-teich} are:
	\begin{align*} 
		&\phi(x)(e) = \phi(x)(e^{-1}), \quad \forall \; e \in \underline{E} \quad \text{and } \quad \phi(x)(t) = 1, \quad \forall \; t \in \triangle; \\
		&\prod_{k=1}^{m_i} \phi(e_{ik}) = 1, \qquad \forall \; i =1,\dots,n.
	\end{align*}
In the present example this simplifies to
\[
 t_{143} = t_{132} = 1, \qquad e_{ij} = e_{ji} \quad \forall i,j \quad \text{and} \qquad e_{21}e_{41}e_{31}=1
\]
So $\Teich(S_{1,1}) \cong \RR^{2}$ as we should expect from the classical theory.


\subsection{The thrice-punctured sphere}

We use the same setup as above with $S=S_{0,3}$. Just as in the previous example Figure~\ref{fig:fund_dom_s03} denotes a topological $S_{0,3}$ with edge pairings indicated by the common edge ratios. Moreover the edge ratios and triple ratios uniquely determine an element of $\Ttp(S_{0,3})$.

	\begin{figure}[h]
		\centering
		\includegraphics[height=5cm]{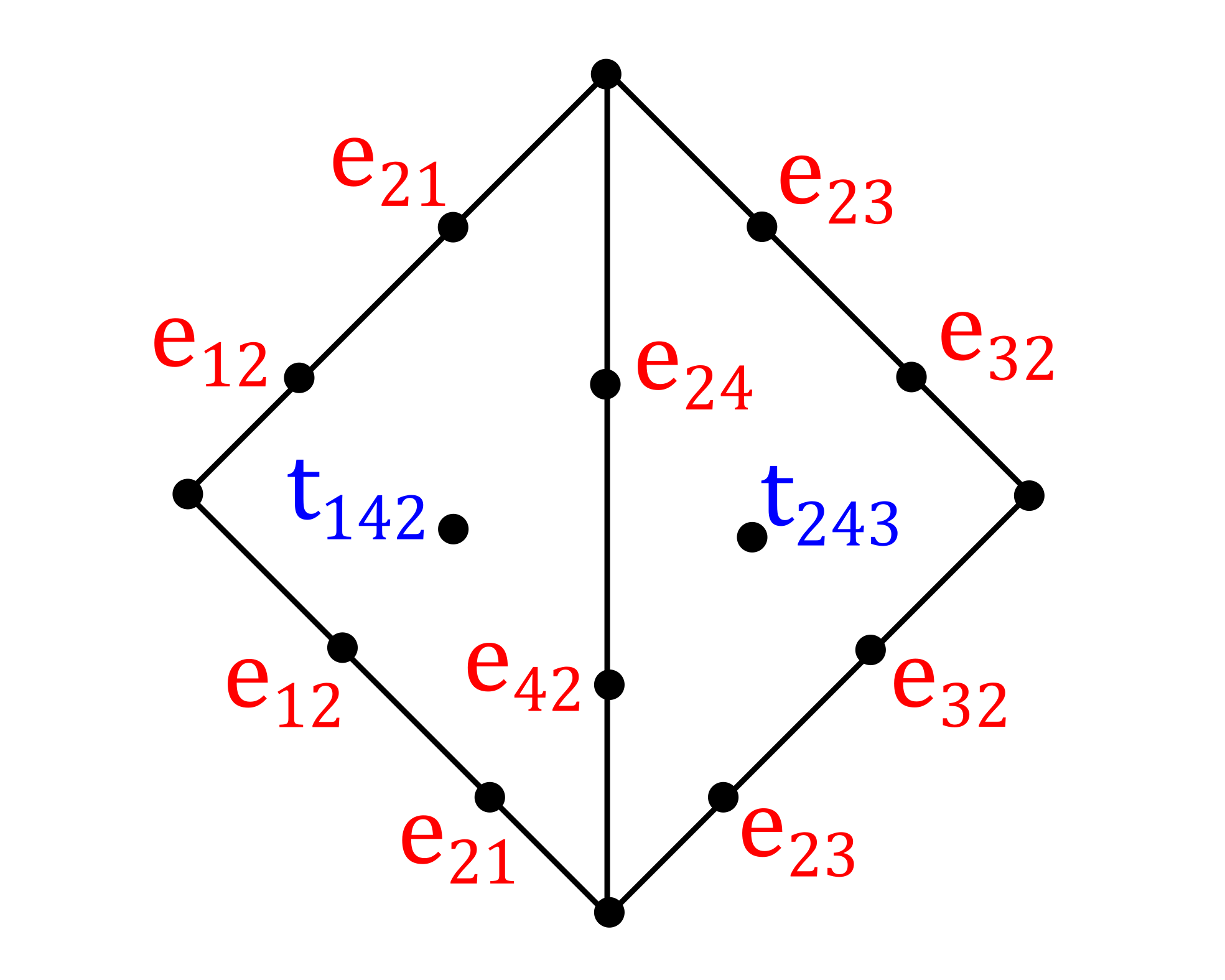}
		\caption{A polygonal map for $S_{0,3}$.}
		\label{fig:fund_dom_s03}
	\end{figure}
The surface $S_{0,3}$ has holonomy group isomorphic to the free group on two generators. Once again the holonomy group is generated by transformations corresponding to the edge pairings $\gamma$ and $\delta$ which glue the two edges on the left and the two edges on the right respecitvely. Their $SL(3, \RR)$-representatives are given below, up to a common conjugation, using the method of \S\ref{subsubsec:computing-monodromy}:
\[
\gamma =  E(e_{21}, e_{12}) T(t_{142}) = \begin{bmatrix} 
e_{12}t_{142} & e_{12}(t_{132}+1) & e_{12} \\
0 & 1 & 1\\
0 & 0 & e_{21}^{-1}
\end{bmatrix}
\]
\[
\delta = T(t_{142}) E(e_{24}, e_{42}) T(t_{243}) E(e_{32}, e_{23}) T(t_{243}) E( e_{42}, e_{24} ) T(t_{142})^{-1}
\]
We have chosen a conjugacy class of the holonomy group so that $\gamma$ is in a convenient form. However as there are three distinct peripheral subgroups of $\pi_1(S_{0,3})$, the other generators of these subgroups, $\delta$ and $\gamma \cdot \delta$ are not written as succinctly. Their eigenvalues are still easily calculable however for more complicated examples it is more convenient to make use of Theorem~\ref{thm:monodromy} to infer that the following monomials govern the boundary geometry of $S_{0,3}$:
\begin{align*}
\mathcal{X}_1 :=& e_{21} \qquad \text{and} \qquad \mathcal{Y}_1 := e_{12} t_{142};\\
\mathcal{X}_2 :=& e_{32}e_{42}e_{12}e_{24} \qquad \text{and} \qquad \mathcal{Y}_2 := e_{23}t_{243}^2 e_{24}t_{142}^2 e_{21}e_{42};\\
\mathcal{X}_3 :=& e_{23} \qquad \text{and} \qquad \mathcal{Y}_3 := e_{32}t_{243}
\end{align*}
Therefore there are at least six independent Casimirs described by the equations $\mathcal{X}_i = c_i$ and $\mathcal{Y}_i = d_i$ for positive real constants $c_i$ and $d_i$, $i= 1, 2, 3$. The equations defining the level sets of these Casimirs simplify somewhat to:
\begin{align*}
e_{21} =& c_1 \qquad \text{and} \qquad e_{12}t_{142} = d_1; \\
\frac{e_{42}e_{24}}{e_{32}e_{12} }  =& c_2 \qquad \text{and} \qquad e_{32}e_{12}  = \left( \frac{c_3 c_2 c_1 (d_1d_3)^2}{d_2} \right)^{\frac{1}{3}}; \\
e_{23}=& c_3 \qquad \text{and} \qquad e_{32}t_{243} = d_3
\end{align*}
This ensures that the maximal symplectic leaves are at most $2$-dimensional. Symplectic manifolds have even dimension so if $\{ \cdot, \cdot \}$ is not uniformly zero then the subvarieties defined by level sets of the six Casimirs above must determine symplectic leaves of the Poisson structure on $\RR^{\triangle \cup \underline{E}}$. 

A direct calculation shows that $\{ t^*_{142}, e^*_{24} \} \neq 0$ where the arguments of the Poisson bracket are two \emph{coordinate functions} defined in \S\ref{subsec:FG-poisson-structure}. In particular, the maximal symplectic leaves of $\Ttp(S_{0,3})$ are $2$-dimensional whereas the symplectic leaves of $\Ttp(S_{1,1})$ are $6$-dimensional. This allows us to distinguish between $\Ttp(S_{1,1})$ and $\Ttp(S_{0,3})$ using their intrinsic Poisson geometry despite the fact that they are homeomorphic.

The structures on $S_{0,3}$ with finite area are described by $c_i = d_i = 1$ for $i=1,2,3$ so $\Ttfin(S_{0,3})$ is defined by the conditions
\begin{align*}
e_{21}=e_{23}=1; \\
e_{42}e_{24}=e_{32}e_{12}=1; \\
e_{12}t_{142} = e_{32}t_{243} = 1
\end{align*}
Therefore $\Ttfin(S_{0,3}) \cong \RR^3$. Imposing the finite-area conditions and the criteria presented in \S\ref{subsec:classical-teich} shows that the classical Teichm\"uller space is identified with:
\[
e_{ij} = 1 \quad \forall \ i, j \qquad \text{and} \qquad t_{142} = t_{243} = 1
\]
so $\Teich(S_{0,3})$ is a single point as we should expect from the classical theory.


\subsection{The twice-punctured torus}

The surface $S = S_{1,2}$ is denoted in Figure~\ref{fig:fund_dom_s12}.
	\begin{figure}[h]
		\centering
		\includegraphics[height=5cm]{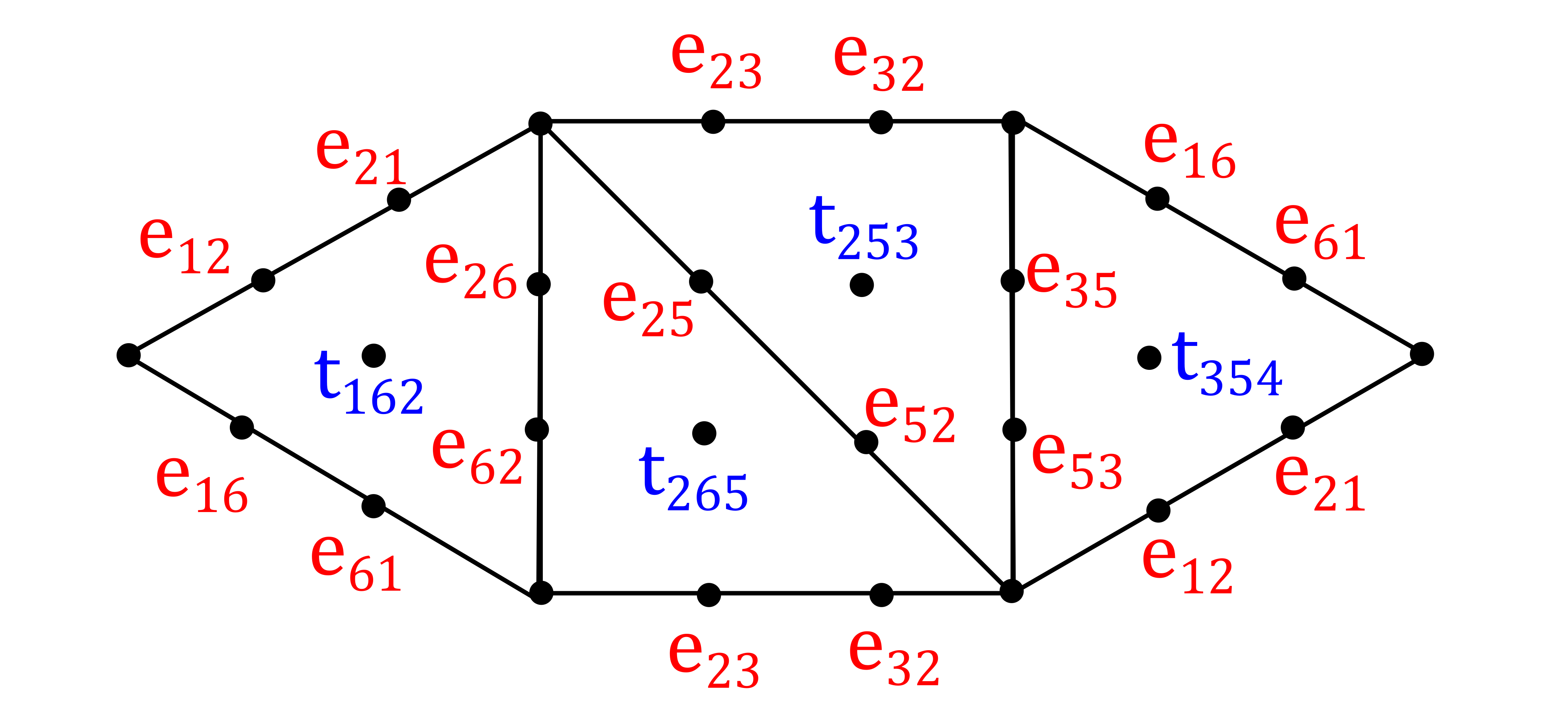}
		\caption{A polygonal map for $S_{1,2}$.}
		\label{fig:fund_dom_s12}
	\end{figure}

The four Casimirs arising from the peripheral elements of $\pi_1(S_{1,2})$ are determined by the monomials $\mathcal{X}_i$ and $\mathcal{Y}_i$ for $i=1,2$ as defined in \S\ref{subsec:fin_vol}. They are given explicitly as follows:
\begin{align*}
\mathcal{X}_1 :=& e_{21} e_{61}e_{53} e_{23} e_{25} e_{35} \qquad \text{and} \qquad \mathcal{Y}_1 := e_{12} t_{162} e_{16} t_{354}^2 e_{35} t_{253}^2 e_{32} t_{265} e_{52}e_{53} \\
\mathcal{X}_2 :=& e_{12} e_{62} e_{52} e_{32} e_{26} e_{16} \qquad \text{and} \qquad \mathcal{Y}_2 := e_{21} t_{162}^2 e_{26} t_{265}^2 e_{25} t_{253} e_{23} e_{62} e_{61} t_{354}
\end{align*}
The boundary geometry of $S_{1,2}$ is determined by the monomials $\mathcal{X}_i = c_i$ and $\mathcal{Y}_i = d_i$ for positive real constants $c_i$ and $d_i$ for $i=1,2$.

Setting $c_i  = d_i = 1$ for $i=1,2$ we find that $\Ttfin(S_{1,2})$ is defined by
\begin{align*}
e_{21} e_{61}e_{53} e_{23} e_{25} e_{35}=1 \qquad \text{and} \qquad t_{162}t_{354}^2 e_{35} t_{253}^2 t_{265} e_{53} = e_{62}e_{26} \\
e_{12} e_{62} e_{52} e_{32} e_{26} e_{16}=1 \qquad \text{and} \qquad t_{162}^2e_{26}t_{265}^2t_{253}e_{62}t_{354}=e_{53}e_{35}
\end{align*}
It follows that $\Ttfin(S_{1,2}) \cong \RR^{12}.$

Imposing both the conditions for finite-area and hyperbolicity as seen in \S\ref{subsec:classical-teich} we identify $\Teich(S_{1,2})$ with the subvariety defined by:
\begin{align*}
& e_{ij} = e_{ji} \quad \forall \ i \text{ and } j \qquad \text{and} \qquad t_{ijk} = 1 \quad \forall \ i,j \text{ and } k\\
& e_{53}^2 e_{21} e_{61} e_{23} e_{25} = 1 \qquad \text{and} \qquad e_{62} = e_{53}
\end{align*}
Therefore $\Teich(S_{1,2}) \cong \RR^4$ in accordance with the classical theory.


\subsection{The closed genus two surface}

Now we will consider $S=S_{2,0}$ and use the method of \S\ref{subsec:closed_surfaces} to construct $\Tt(S_{2,0})$ from $\Ttp(S_{1,2})$ by gluing boundary components. 

In order to glue the boundary components of $S_{1,2}$ in the manner of Goldman~\cite{Goldman-convex-1990} we must ensure that both ends of $S_{1,2}$ are minimal. Using the notation from the previous example those conditions are, for $i=1$ and $i=2$:
\begin{align}
&\mathcal{Y}_i >1 \quad \text{and} \quad \mathcal{X}_i >1 \quad \text{or} \label{minimality1_s12} \\
&\mathcal{X}_i <1 \quad \text{and} \quad \mathcal{Y}_i <1 \quad \text{or} \\
&\mathcal{X}_i\mathcal{Y}_i <1 \quad \text{and} \quad \mathcal{X}_i > 1 \quad \text{or} \\
&\mathcal{X}_i\mathcal{Y}_i >1 \quad \text{and} \quad \mathcal{X}_i < 1 \label{minimality4_s12}
\end{align}
Note that $\Tt(S_{2,0})$ does not contain information regarding framing so we will `forget' the framing information contained in $\Ttp(S_{1,2})$. This is tantamount to restricting our attention to exactly one of the regions defined by \ref{minimality1_s12}--\ref{minimality4_s12} for each of $i=1$ and $i=2$. We assume without loss of generality that equation \ref{minimality1_s12} holds for $i=1$ and $i=2$. 

The second requirement in performing convex gluing is that the peripheral transformations at the ends of $S_{1,2}$ be conjugate. Therefore, following the classification of peripheral elements and their eigenvalues in \S\ref{subsubsec:max-min-coordinates} we assume that
\[
\mathcal{X}_1\mathcal{Y}_1^2 = \mathcal{X}_2\mathcal{Y}_2^2 \qquad \text{and}\qquad \mathcal{X}_1^2\mathcal{Y}_1 =\mathcal{X}_2^2 \mathcal{Y}_2 
\]
Given that all variables must be positive and real we have one pair of solutions of the form $\mathcal{X}_1 = \mathcal{X}_2$ and $\mathcal{Y}_1 = \mathcal{Y}_2$. The solutions $\mathcal{X}_1=\mathcal{Y}_1$ and $\mathcal{X}_2 = \mathcal{Y}_2$ are disregarded in order to ensure that the convex gluing induces a coherent orientation on $S_{2,0}$. We are left with the following conditions which parameterise convex projective structures on $S_{1,2}$ whose ends are `gluable'.
\begin{align}
e_{21} e_{61} e_{53} e_{23} e_{25} e_{35} = e_{12} e_{62} e_{52} e_{32} e_{26} e_{16} > 1 \label{gluable1_s12} \\
(e_{53} e_{35})^2 t_{354} t_{253}= (e_{26} e_{62})^2 t_{162} t_{265} > 1 \label{gluable2_s12} 
\end{align}
It remains to show that these equations define a topological cell of real dimension $16$ thus verifying Goldman's~\cite{Goldman-convex-1990} result. Following the method used in Theorem~\ref{thm:Marquis-main}, it is sufficient to note that the equations~\ref{gluable1_s12}-\ref{gluable2_s12} are independent. Recall that as long as the peripheral elements of $S_{1,2}$ are hyperbolic, there is a $2$-real parameter space of potential gluings for each pair of `gluable' ends. For an explanation of this see \S\ref{subsec:closed_surfaces} or the original proof in Goldman~\cite{Goldman-convex-1990}. Therefore $\Tt(S_{2,0})$ is a trivial $2$-dimensional vector bundle over the subvariety of $\Ttp(S_{1,2})$ defined by the conditions~\ref{gluable1_s12}-\ref{gluable2_s12}. Ultimately this shows that $\Tt(S_{2,0}) \cong \RR^{16}$, reproducing Goldman's result.

Once again we can use this to calculate a subvariety of $\Tt(S_{2,0})$ which is homeomorphic to $\Teich(S_{2,0})$. First note that any hyperbolic structure on $S_{2,1}$ must restrict to a hyperbolic structure on $S_{1,2}$. Therefore we impose the hyperbolicity and finite-area conditions on $\Ttp(S_{1,2})$ defined in \S\ref{subsec:classical-teich} to obtain:
\[
t_{ijk} =1 \quad \forall \ i,j \ \text{and } k \qquad \text{and} \qquad e_{ij}=e_{ji} \quad \forall \  i \ \text{and } j
\]
Further imposing these conditions on the `gluable' equations we are left with:
\begin{align}
& e_{35} = e_{62} \label{fin_hyp1_S20} \\
& t_{ijk} =1 \quad \forall \ i,j \ \text{and }  k; \label{fin_hyp2_S20} \\
& e_{ij}=e_{ji} \quad \forall \ i \ \text{and }  j \label{fin_hyp3_S20}
\end{align}
To recover $\Teich(S_{2,0})$ note that instead of a $2$ parameter family of convex gluings we may only glue the ends of $S_{1,2}$ by hyperbolic isometries if we wish to give rise to a hyperbolic structure on $S_{2,0}$. Therefore, $\Teich(S_{2,0})$ may be realised as a trivial real $1$-dimensional vector bundle over the subvariety defined by the conditions~\ref{fin_hyp1_S20}-\ref{fin_hyp3_S20}. In particular $\Teich(S_{2,0}) \cong \RR^{6}$ which accords with the classical theory.

\newpage


\addcontentsline{toc}{section}{References}
\bibliographystyle{plain}
\bibliography{projective}

\address{School of Mathematics and Statistics F07,\\ The University of Sydney,\\ NSW 2006 Australia}

\Addresses


\end{document}